\documentclass[11pt,reqno]{amsart}

\usepackage{amsmath,amssymb,mathrsfs}
\usepackage{graphicx,cite,times}
\usepackage{hyperref}
\hypersetup{
    colorlinks=true,
    linkcolor=blue,
    filecolor=gray,
    urlcolor=blue,
    citecolor=blue,
}

\usepackage[doipre={doi:~}]{uri}
\usepackage{cases}
\usepackage{dsfont}
\usepackage[mathscr]{eucal}
\usepackage{appendix}
\usepackage{color}
\newtheorem{theo}{Theorem}[section]

\newtheorem{lemm}[theo]{Lemma}
\newtheorem{prop}[theo]{Proposition}
\newtheorem{rema}[theo]{Remark}
\newtheorem{defi}[theo]{Definition}

\renewcommand{\Im}{\operatorname{Im}}
\numberwithin{equation}{section}

\begin{document}

\title[Modal approximation for nano-bubbles  in elastic materials]{
 Resonant modal approximation of  time-domain elastic scattering from nano-bubbles  in  elastic materials}

\author{Bochao Chen}
\address{School of
Mathematics and Statistics, Center for Mathematics and
Interdisciplinary Sciences, Northeast Normal University, Changchun, Jilin 130024, China. }
\email{chenbc758@nenu.edu.cn}

\author{Yixian Gao}
\address{School of
Mathematics and Statistics, Center for Mathematics and
Interdisciplinary Sciences, Northeast Normal University, Changchun, Jilin 130024, China. }
\email{gaoyx643@nenu.edu.cn}

\author{Yong  Li}
\address{School of Mathematics, Jilin University, Changchun, Jilin 130012;
School of Mathematics and Statistics, Center for Mathematics and
Interdisciplinary Sciences, Northeast Normal University, Changchun, Jilin
130024, China}
\email{yongli@nenu.edu.cn}

\author{Hongyu Liu}
\address{Department of Mathematics, City University of Hong Kong, Kowloon, Hong Kong, China }
\email{hongyu.liuip@gmail.com, hongyliu@cityu.edu.hk}

\thanks{The research of BC was supported in part by  NSFC grants (project number, 11901232) and the Fundamental Research Funds for the Central Universities (project number, 2412022QD032). The research of YG was is supported by NSFC grants (project numbers, 11871140 and 12071065) and  FRFCU2412019BJ005. The research of YL was supported in part by NSFC grants (project number, 12071175).
The research of HL was supported by Hong Kong RGC General Research Funds (project numbers, 11311122,
11300821 and 12301420) and the NSFC-RGC Joint Research Grant (project number,
 N\_CityU101/21), and the ANR/RGC Joint Research Grant (project number, A\_CityU203/19).}



\begin{abstract}
This paper is devoted to establishing the resonant modal expansion  of the low-frequency part of the scattered
field for acoustic bubbles embedded in  elastic materials in the time domain. Due to the nano-bubble with damping, Minnaert resonance can be induced at certain discrete resonant frequencies, which forms the fundamental basis of effectively constructing elastic metamaterials via the composite material theory. There are two major contributions in this work. First, we ansatz a special form of the density, approximate the incident field with a finite number of modes,   and then obtain  an expansion with a finite number of modes for the acoustic-elastic wave scattering in the time-harmonic regime. Second, we  show that  the low-frequency part of the scattered field in the time domain can be well approximated by using the resonant modal expansion with sharp error estimates. Interestingly, we find that the $0$-th mode is the main contribution to the resonant modal expansion.

\end{abstract}

\maketitle

\section{Introduction}

\subsection{Background and motivation}

In this paper, we aim at quantitatively understanding the elastic wave propagation in a bubbly elastic medium. Consider an air bubble with damping of nanoscale. For an outgoing elastic point source,  its impingement on the air bubble  generates the scattering phenomenon. It can be mathematical viewed as a time-dependent interaction problem in three-dimensional space between an elastic wave and an acoustic wave. The elastic
and acoustic fields are coupled together via the transmission conditions across the boundary.  This leads to a coupled PDE system. The purpose of this paper is to establish a resonance-type expansion for the low-frequency part of the scattered
field for bubble-elastic structures in the time domain.

Concerning the resonance of the bubbly elastic medium, we have to refer to the Minnaert resonance, which is a widely known acoustic phenomenon. The Minnaert resonance is a kind of low-frequency resonance in which the wavelength is much larger than the size of the bubble; see \cite{xvi1933musical}.  It was first discovered for acoustic wave propagation in liquids with air bubbles \cite{commander1989linear,hwang2000low}. The exceptional acoustic properties can be used to design new materials. Nevertheless, as mentioned in \cite{leroy2009design}, since the bubbles are unstable in the
liquid, it is difficult to control. In order to overcome this difficulty, soft elastic materials (namely, the shear modulus is small) will substitute for liquids. In fact, a small volume fraction of air bubbles in a soft elastic medium can have a significant influence on the effective propagation of the elastic wave in the medium due to a certain resonance phenomenon \cite{calvo2012low,leroy2015superabsorption}. Recently, several mathematical theories have been derived and effectively used
in understanding the Minnaert resonances of bubbles. Due to the sharp contrast of the density between the air bubble and the liquid or elastic medium, the authors in \cite{Ammari2018Minnaert,Li2022Minnaert} provided  a rigorous and systematic mathematical study of the Minnaert resonance in the case of a single bubble immersed in liquids and soft elastic materials, respectively. Moreover, one significant application of the Minnaert resonance and bubbly material structures is that they can be used to effectively realize a variety of metamaterials with exotic material properties via the theory of composite materials; see e.g. \cite{AZ,Li2022Minnaert} and the references cited therein for more related background discussions.

A Several important literatures mentioned
above on typical bubble-elastic structures are confined to the time-harmonic setting.  However, the model setting in the physical setup depends not only on the space, but also on the time. We refer to the recent works  \cite{Ammari2022Asympototic,Baldassari2021Modal} and \cite{CGL} on time-dependent modal approximations as well. In \cite{Ammari2022Asympototic,Baldassari2021Modal}, the authors gave the plasmon resonant expansion of the electromagnetic field scattered
by nanoparticles with dispersive material parameters placed in a homogeneous medium in the low-frequency regime. In \cite{CGL}, the three authors of the present article established the polariton resonant expansion of the elastic field scattered by elastic metamaterial nanoparticles in the low-frequency regime.
Motivated by the aforementioned physical and mathematical studies, we consider
in this paper the elastic wave propagation in a bubbly elastic medium in the time domain.  Specifically, we show that the low-frequency part of the elastic scattered field can be well approximated by the Minnaert resonant modal expansion. In fact, the corresponding resonance is induced by the nano-bubble with damping, which does not depends on, in addition to  the high contrast
of the density, the high contrast
of  the shear modulus and the compression modulus as seen in \cite{Li2022Minnaert}. Nevertheless, for terminological convenience, we still
call it the Minnaert resonance in the present paper.

In addition, we highlight several mathematical and technical developments achieved in the current study. First, we are mainly concerned with a coupled-physics PDE system which is technically more involved than single-physics PDE systems. Second, the Minnaert resonance plays a major role in our modal analysis which is of a different physical nature from the polariton resonance in \cite{CGL}, though they might be connected via the effective medium theory. In particular, we would like to emphasize that the wavenumber is involved in the transmission conditions of the Fourier-transformed acoustic-elastic system, which poses substantial challenges to both the resonant and modal analyses. To derive the Minnaert resonance, we must require that the order of asymptotic expansions of the elastic field with respect to the wavenumber is  higher  than the one given by the perturbation theory used in the literatures \cite{Ammari2022Asympototic,Baldassari2021Modal,CGL}. Hence we shall mainly confine our study within the radial geometry. Nevertheless, this presents sufficient technical challenges and moreover is physically unobjectionable since in the effective composite medium theory, the radial geometry is frequently implemented in the construction of elastic metamaterials based on bubble-elastic structures (cf. \cite{SMG,Li2022Minnaert}). Indeed, the results derived in this article may have practical implications to the effective construction of time-modulated elastic metamaterials (cf. \cite{TX}), which we shall explore in a forthcoming work. Finally, we would like to mention an interesting discovery that the $0$-th mode is only needed to reconstruct the information of the low-frequency part of the scattered field in the time domain. Our findings could
thus bring fresh insight to design metamaterials for applications (cf. \cite{Pierre2014}).

\subsection{Problem formulation}

Focusing on the
mathematics, but not the physics, we present the problem formulation of our study. Consider a nano-bubble $D$ of the following form:
\begin{align*}
D=\epsilon B+\boldsymbol z,
\end{align*}
where $\epsilon\in\mathbb{R}_+$ with $\epsilon\ll 1$, the point $\boldsymbol z=(z_1,z_2,z_3)\in\mathbb{R}^3$ and $B$ is a unit ball containing the origin in $\mathbb{R}^3$. The medium configuration of the air bubble is characterised by the damping coefficient $\gamma>0$, the density $\rho_b\in\mathbb{R}_+$, and the bulk modulus $\kappa\in\mathbb{R}_+$. We define $c_b=\sqrt{\kappa/\rho_b}$ which stands for the velocity of the acoustic wave inside $D$. We assume that the background $\mathbb{R}^3\backslash\overline{D}$ is occupied by a regular and isotropic elastic material parameterized by the density $\rho_e\in\mathbb{R}_+$, and the Lam\'{e} constant $(\tilde{\lambda},\tilde{\mu})$ with
\begin{align}\label{mu}
\mathrm{(i)}~\tilde{\mu}>0, \qquad\text{and}\qquad\mathrm{(ii)}~3\tilde{\lambda}+2\tilde{\mu}>0.
\end{align}

In the following, we let $\tilde{\boldsymbol x}=(\tilde{x}_1, \tilde{x}_2, \tilde{x}_3)\in\mathbb{R}^3$ and $t\in\mathbb{R}$ denote the space and time variables, respectively. Let $\boldsymbol{\delta}_{\tilde{\boldsymbol x}}(\mathbf{s})\hat{f}''(t)\tilde{\mathbf{p}}$ signify a wide-band impinging signal, where $\hat{f}:t\mapsto \hat{f}(t)\in C^{\infty}_0([0,C_1])$ with $C_1>0$, $\boldsymbol{\delta}_{\tilde{\boldsymbol{x}}}(\mathbf{s})$ is the Kronecker delta with $\mathbf{s}=(s_1,s_2,s_3)\in\mathbb{R}^3$ being the source location, and $\tilde{\mathbf p}=(\tilde{p}_i)_{i=1}^3\in\mathbb{R}^3$ signifies a polarization vector. Let $\hat{\tilde{\mathbf{u}}}(\tilde{\boldsymbol x},t)=(\hat{\tilde{u}}_i(\tilde{\boldsymbol x}, t))_{i=1}^3$ and $\hat{\tilde{p}}(\tilde{\boldsymbol x},t)$ be, respectively, the total elastic displacement field outside the domain $D$ and the acoustic pressure inside the domain $D$. Define the symmetric gradient (strain tensor) $\nabla^s$  by
\begin{align*}
\nabla^s\hat{\tilde{\mathbf{u}}}:=\frac12(\nabla \hat{\tilde{\mathbf{u}}}+\nabla \hat{\tilde{\mathbf{u}}}^\top),
\end{align*}
where $\top$  denotes the transpose and  $\nabla \hat{\tilde{\mathbf{u}}}$ is the displacement gradient given by $\nabla\hat{\tilde{\mathbf{u}}}=(\partial_{\tilde{x}_j}\hat{\tilde{u}}_i)_{i,j=1}^3$. The elastostatic operator and the conormal derivative on $\partial D$ are defined, respectively, as
\begin{align*}
\mathcal{L}_{\tilde{\lambda},\tilde{\mu}}\hat{\tilde{\mathbf{u}}}&:=\tilde{\mathbf{\mu}}\Delta\hat{\tilde{\mathbf{u}}}+(\tilde{\lambda}+\tilde{\mu})\nabla\nabla\cdot\hat{\tilde{\mathbf{u}}}  && \text{in}~~\mathbb R^3\times\mathbb{R},\\
\frac{\partial}{\partial{\boldsymbol{\nu}}}\hat{\tilde{\mathbf{u}}}&:=\tilde{\lambda}(\nabla \cdot \hat{\tilde{\mathbf{u}}})\boldsymbol{\nu}+2\tilde{\mu}(\nabla^s \hat{\tilde{\mathbf{u}}})\boldsymbol{\nu}&&\text{on}~~\partial D\times\mathbb{R},
\end{align*}
where $\boldsymbol{\nu}$ denotes the exterior unit normal vector to the boundary $\partial D$.

The acoustic-elastic wave interaction  can be
  described by the following coupled PDE system
\begin{align}\label{S:Lame}
\left\{ \begin{aligned}
&\mathcal{L}_{\tilde{\lambda},\tilde{\mu}}\hat{\tilde{\mathbf{u}}}(\tilde{\boldsymbol x},t)-\rho_e\partial^2_{t}\hat{\tilde{\mathbf{u}}}(\tilde{\boldsymbol{x}},t)=\boldsymbol\delta_{\tilde{\boldsymbol x}}(\mathbf{s})\hat{f}''(t)\tilde{\mathbf{p}},\quad&&(\tilde{\boldsymbol{x}},t)\in\mathbb{R}^3\backslash\overline{D}\times\mathbb{R},\\
&\Delta \hat{\tilde{p}}(\tilde{\boldsymbol{x}},t)-\frac{\gamma}{c^2_b} \partial_t\hat{\tilde{p}}(\tilde{\boldsymbol{x}},t)-\frac{1}{c^2_b}\partial^2_{t}\hat{\tilde{p}}(\tilde{\boldsymbol{x}},t)=0,\quad&&(\tilde{\boldsymbol{x}},t)\in{D}\times\mathbb{R},\\
&\nabla \hat{\tilde{p}}(\tilde{\boldsymbol{x}},t)\cdot \boldsymbol{\nu}=-\rho_b\partial^2_{t}\hat{\tilde{\mathbf{u}}}(\tilde{\boldsymbol{x}},t)\cdot \boldsymbol \nu,\quad&&(\tilde{\boldsymbol{x}},t)\in\partial D\times\mathbb{R},\\
&-\hat{\tilde{p}}(\tilde{\boldsymbol{x}},t)\boldsymbol{\nu}=\frac{\partial}{\partial\boldsymbol{\nu}}\hat{\tilde{\mathbf{u}}}(\tilde{\boldsymbol{x}},t),\quad&&(\tilde{\boldsymbol{x}},t)\in\partial D\times\mathbb R,\\
&\hat{\tilde{\mathbf{u}}}-\hat{\tilde{\mathbf{u}}}^{\mathrm{in}} \text{ satisfies the radiation condition}.
\end{aligned}\right.
\end{align}
In the physical setup, the first equation in  \eqref{S:Lame} is known as the Lam\'{e} system which describes
the propagation of elastic deformation, whereas the second one is the wave equation
which governs the acoustic wave propagation.  The elastic and acoustic fields are
coupled together via the   kinematic and dynamic
interface conditions as seen in \cite{BGL2018,zhang2022time}.   The third condition in  \eqref{S:Lame} presents the continuity of the normal component of the displacement field on the boundary $\partial D$, whereas the fourth one signifies the continuity of the stress across the boundary $\partial D$; see \cite{Li2022Minnaert}. Denote by $\hat{\tilde{\mathbf{u}}}^{\mathrm{in}}$ an outgoing elastic point source to
\begin{align*}
\mathcal{L}_{\tilde{\lambda},\tilde{\mu}}\hat{\tilde{\mathbf{u}}}(\tilde{\boldsymbol x},t)-\rho_e\partial^2_{t}\hat{\tilde{\mathbf{u}}}(\tilde{\boldsymbol{x}},t)=\boldsymbol\delta_{\tilde{\boldsymbol x}}(\mathbf{s})\hat{f}''(t)\tilde{\mathbf{p}}.
\end{align*}
The radiation condition in \eqref{S:Lame} designates the following
two asymptotic relations as $|\tilde{\boldsymbol x}|\rightarrow+\infty$:
 \begin{align*}
(\nabla\times\nabla\times(\hat{\tilde{\mathbf{u}}}-\hat{\tilde{\mathbf{u}}}^{\mathrm{in}}))(\tilde{\boldsymbol x}, t)\times\frac{\tilde{\boldsymbol x}}{|\tilde{\boldsymbol x}|}-\frac{1}{\tilde{c}_s}\nabla\times\partial_t(\hat{\tilde{\mathbf{u}}}-\hat{\tilde{\mathbf{u}}}^{\mathrm{in}})(\tilde{\boldsymbol x}, t)&=\mathcal{O}(|\tilde{\boldsymbol x}|^{-2}),\\
\frac{\tilde{\boldsymbol x}}{|\tilde{\boldsymbol x}|}\cdot (\nabla\nabla\cdot(\hat{\tilde{\mathbf{u}}}-\hat{\tilde{\mathbf{u}}}^{\mathrm{in}}))(\tilde{\boldsymbol x}, t)-\frac{1}{\tilde{c}_p}\nabla\partial_t(\hat{\tilde{\mathbf{u}}}-\hat{\tilde{\mathbf{u}}}^{\mathrm{in}})(\tilde{\boldsymbol x}, t)&=\mathcal{O}(|\tilde{\boldsymbol x}|^{-2}),
\end{align*}
where
\begin{align}\label{E:speed}
\tilde{c}_s=\sqrt{\tilde{\mu}/\rho_e},\quad \tilde{c}_p=\sqrt{(\tilde{\lambda}+2\tilde{\mu})/\rho_e}.
\end{align}

In this article, we shall establish under a certain physical scenario when the medium parameters and the incident wave fulfil some general conditions that the scattered field $\hat{\tilde{\mathbf{u}}}-\hat{\tilde{\mathbf{u}}}^{\mathrm{in}}$ to \eqref{S:Lame} can be well approximated by a Minnaert resonant expansion in the low-frequency regime. The rest of the paper is outlined as follows.  Section \ref{sec:2} aims at giving the time-harmonic form  of the system \eqref{S:Lame} via the temporal Fourier transform and presenting the equivalent integral representations by the layer-potential
methods. In Section \ref{sec:3}, we present the asymptotic and spectral analyses of the  layer-potential operators and establish the modal decomposition with a finite number of modes  of the acoustic-elastic wave scattering in the time-harmonic regime. The purpose of Section \ref{sec:4} aims to calculate the size- and frequency-dependent Minnaert resonance. In Section \ref{sec:5}, our goal is to establish an approximation for the truncated scattered field of the bubble-elastic structure in the time domain.


\section{Time-harmonic form and preliminaries}\label{sec:2}

The objective of this section is to apply the temporal Fourier transform to obtain the time-harmonic form of the system \eqref{S:Lame}, and then give the corresponding equivalent integral formulations by means of the potential theory.

\subsection{Temporal Fourier transforms}
Introduce the temporal Fourier transforms and inverse Fourier transforms as follows:
\begin{align*}
(\mathcal{F}\hat{\tilde{p}})(\tilde{\boldsymbol x},\omega)&:=\frac{1}{2\pi}\int^{+\infty}_{-\infty}\hat{\tilde{p}}(\tilde{\boldsymbol x},t)e^{\mathrm{i}\omega t}\mathrm{d}t,\qquad
(\mathcal{F}\hat{\tilde{\mathbf{u}}})(\tilde{\boldsymbol x},\omega):=\frac{1}{2\pi}\int^{+\infty}_{-\infty}\hat{\tilde{\mathbf{u}}}(\tilde{\boldsymbol x},t)e^{\mathrm{i}\omega t}\mathrm{d}t,\\
(\mathcal{F}^{-1}\tilde{p})(\tilde{\boldsymbol x},t)&:=\int^{+\infty}_{-\infty}\tilde{p}(\tilde{\boldsymbol x},\omega)e^{-\mathrm{i}\omega t}\mathrm{d}\omega,\qquad
(\mathcal{F}^{-1}\tilde{\mathbf{u}})(\tilde{\boldsymbol x},t):=\int^{+\infty}_{-\infty}\tilde{\mathbf{u}}(\tilde{\boldsymbol x},\omega)e^{-\mathrm{i}\omega t}\mathrm{d}\omega,
\end{align*}
where $\mathrm{i}:=\sqrt{-1}$ is the imaginary unit.

Throughout the rest of the paper, we impose that physical conditions that {$\hat{\tilde{p}}(\tilde{\boldsymbol x},\cdot)\in L^2(\mathbb{R}),\hat{\tilde{\mathbf{u}}}(\tilde{\boldsymbol x},\cdot)\in L^2(\mathbb{R})^3$} for a fixed $\tilde{\boldsymbol x}$, and that the wave and its first derivative with respect to time decay to zero as the time tends to infinity, namely, $\hat{\tilde{p}}(\tilde{\boldsymbol x},\cdot),\hat{\tilde{\mathbf{u}}}(\tilde{\boldsymbol x},\cdot)\rightarrow0$ and $\frac{\partial}{\partial t}\hat{p}(\tilde{\boldsymbol x},\cdot),\frac{\partial}{\partial t}\hat{\tilde{\mathbf{u}}}(\tilde{\boldsymbol x},\cdot)\rightarrow0$ as $|t|\rightarrow+\infty$. Let $\tilde{p},\tilde{\mathbf{u}}$ and $f$ stand for the Fourier transforms of $\hat{\tilde{p}},\hat{\tilde{\mathbf u}}$ and $\hat{f}$, respectively. Using the temporal Fourier transform to \eqref{S:Lame} yields that
\begin{align}\label{SS:Lame}
\left\{ \begin{aligned}
&\mathcal{L}_{\tilde{\lambda},\tilde{\mu}}\tilde{\mathbf{u}}(\tilde{\boldsymbol x},\omega)+\omega^2\rho_e\tilde{\mathbf{u}}(\tilde{\boldsymbol x},\omega)=-\boldsymbol\delta_{\tilde{\boldsymbol x}}(\mathbf{s})\omega^2{f}(\omega)\tilde{\mathbf{p}},\quad&&\tilde{\boldsymbol x}\in\mathbb{R}^3\backslash\overline{D},\\
&\Delta \tilde{p}(\tilde{\boldsymbol x},\omega)+c(\omega)\tilde{k}^2\tilde{p}(\tilde{\boldsymbol x},\omega)=0,&&\tilde{\boldsymbol x}\in{D},\\
&\nabla \tilde{p}(\tilde{\boldsymbol x},\omega)\cdot \boldsymbol{\nu}=\rho_b\omega^2\tilde{\mathbf{u}}(\tilde{\boldsymbol x},\omega)\cdot \boldsymbol \nu,&&\tilde{\boldsymbol x}\in\partial D,\\
&-\tilde{p}(\tilde{\boldsymbol x},\omega)\boldsymbol{\nu}=\frac{\partial}{\partial{\boldsymbol{\nu}}}\tilde{\mathbf{u}}(\tilde{\boldsymbol x},\omega),&&\tilde{\boldsymbol x}\in\partial D,\\
&\tilde{\mathbf{u}}-\tilde{\mathbf{u}}^{\mathrm{in}} \text{ satisfies the radiation condition},
\end{aligned}\right.
\end{align}
where
\begin{align}\label{c:omega}
\tilde{k}={\omega}/{c_b}\quad\text{and}\quad c(\omega)=1+{\mathrm{i}\gamma}/{\omega}.
\end{align}
For simplicity, we no longer write the dependence of $\omega$ for  $\tilde{p},\tilde{\mathbf{u}}$.  The radiation condition in \eqref{SS:Lame} designates the following conditions as $|\tilde{\boldsymbol x}|\rightarrow+\infty$:
\begin{align*}
(\nabla\times\nabla\times(\tilde{\mathbf{u}}-\tilde{\mathbf{u}}^{\mathrm{in}}))(\tilde{\boldsymbol x})\times\frac{\tilde{\boldsymbol x}}{|\tilde{\boldsymbol x}|}+\mathrm{i}\tilde{k}_s\nabla\times(\tilde{\mathbf{u}}-\tilde{\mathbf{u}}^{\mathrm{in}})(\tilde{\boldsymbol x})&=\mathcal{O}(|\tilde{\boldsymbol x}|^{-2}),\\
\frac{\tilde{\boldsymbol x}}{|\tilde{\boldsymbol x}|}\cdot (\nabla\nabla\cdot(\tilde{\mathbf{u}}-\tilde{\mathbf{u}}^{\mathrm{in}}))(\tilde{\boldsymbol x})+\mathrm{i}\tilde{k}_p\nabla(\tilde{\mathbf{u}}-\tilde{\mathbf{u}}^{\mathrm{in}})(\tilde{\boldsymbol x})&=\mathcal{O}(|\tilde{\boldsymbol x}|^{-2}).
\end{align*}
Here
\begin{align*}
\tilde{k}_s=\frac{\omega}{\tilde{c}_s}=\frac{\omega}{\sqrt{\tilde{\mu}/\rho_e}},\quad \tilde{k}_p=\frac{\omega}{\tilde{c}_p}=\frac{\omega}{\sqrt{(\tilde{\lambda}+2\tilde{\mu})/\rho_e}}.
\end{align*}

\subsection{Boundary integral operators}\label{sec:2.1}

We introduce the layer-potential  operators for the Helmholtz equation
\begin{align*}
\Delta \breve{p}(\tilde{\boldsymbol x})+\tilde{k}^2\breve{p}(\tilde{\boldsymbol x})=0.
\end{align*}
The corresponding fundamental solution is given by
\begin{align*}
\mathscr{G}^{\tilde{k}}(\tilde{\boldsymbol{x}})=-\frac{e^{\mathrm{i}\tilde{k}|\tilde{\boldsymbol{x}}|}}{4\pi|\tilde{\boldsymbol{x}}|}.
\end{align*}
For $\breve{\varphi}\in L^2(\partial D)$, the single layer potential and the Neumann-Poincar\'{e} operator involved for the fundamental solution $\mathscr{G}^{\tilde{k}}$ are defined, respectively, by
\begin{align}\label{Single}
\mathscr{S}^{\tilde{k}}_{\partial D}[\breve{\varphi}](\tilde{\boldsymbol{x}})&:=\int_{\partial D}\mathscr{G}^{\tilde{k}}(\tilde{\boldsymbol{x}},\tilde{\boldsymbol{y}})\breve{\varphi}(\tilde{\boldsymbol{y}})\mathrm{d}\sigma(\tilde{\boldsymbol{y}}), &&\tilde{\boldsymbol{x}}\in\mathbb{R}^3,\\
\mathscr{K}^{\tilde{k},*}_{\partial D}[\breve{\varphi}](\tilde{\boldsymbol{x}})&:=\mathrm{p.v.}\int_{\partial D} \nabla_{\tilde{\boldsymbol{x}}}\mathscr{G}^{\tilde{k}}(\tilde{\boldsymbol{x}},\tilde{\boldsymbol{y}})\cdot\boldsymbol{\nu}_{\tilde{\boldsymbol{x}}}\breve{\varphi}(\tilde{\boldsymbol{y}})\mathrm{d}\sigma(\tilde{\boldsymbol{y}}), &&\tilde{\boldsymbol{x}}\in\partial D,\label{NP}
\end{align}
where $\mathscr{G}^{\tilde{k}}(\tilde{\boldsymbol x},\tilde{\boldsymbol y}):=\mathscr{G}^{\tilde{k}}(\tilde{\boldsymbol x}-\tilde{\boldsymbol y})$ and $\mathrm{p.v.}$ stands for the Cauchy principle value. Then the conormal derivative of the single layer potential enjoys the jump formula
\begin{align}\label{Jump}
\nabla \mathscr{S}^{\tilde{k}}_{\partial D}[\breve{\varphi}]\cdot\boldsymbol{\nu}|_{\pm}(\tilde{\boldsymbol{x}})=\big(\pm\frac{1}{2}
\mathcal {I}+\mathscr{K}^{\tilde{k},*}_{\partial D}\big)[\breve{\varphi}](\tilde{\boldsymbol{x}}),\quad \tilde{\boldsymbol{x}}\in\partial D,
\end{align}
where $\mathcal {I}$ is the identity operator and $\pm$ denotes the traces on $\partial D$ taken from outside and inside of the domain $D$, respectively. Moreover,  the $L^2$-adjoint of the operator $\mathscr{K}^{\tilde{k},*}_{\partial D}$ is given by
\begin{align*}
\mathscr{K}^{\tilde{k}}_{\partial D}[\breve{\varphi}](\tilde{\boldsymbol{x}})=\mathrm{p.v.}\int_{\partial D} \nabla_{\tilde{\boldsymbol{y}}}\mathscr{G}^{\tilde{k}}(\tilde{\boldsymbol{x}},\tilde{\boldsymbol{y}})\cdot\boldsymbol{\nu}_{\tilde{\boldsymbol{y}}}\breve{\varphi}(\tilde{\boldsymbol{y}})\mathrm{d}\sigma(\tilde{\boldsymbol{y}}),\quad \tilde{\boldsymbol{x}}\in\partial D.
\end{align*}
For simplicity,  we denote   $ \mathscr{S}^{\tilde{k}}_{\partial D},\mathscr{K}^{\tilde{k},*}_{\partial D},\mathscr{K}^{\tilde{k}}_{\partial D}$ with $\tilde{k}=0$ by $\mathscr{S}_{\partial D},\mathscr{K}^{*}_{\partial D},\mathscr{K}_{\partial D}$.
We would like to point out that the operators $\mathscr{S}_{\partial D},\mathscr{K}^{*}_{\partial D},\mathscr{K}_{\partial D}$ have the following expressions in three dimensions:
\begin{align*}
\mathscr{S}_{\partial D}[\breve{\varphi}](\tilde{\boldsymbol{x}})&=-\int_{\partial D}\frac{1}{4\pi|\tilde{\boldsymbol x}-\tilde{\boldsymbol y}|}\breve{\varphi}(\tilde{\boldsymbol{y}})\mathrm{d}\sigma(\tilde{\boldsymbol{y}}),&& \tilde{\boldsymbol{x}}\in\mathbb{R}^3,\\
\mathscr{K}^{*}_{\partial D}[\breve{\varphi}](\tilde{\boldsymbol{x}})&=\int_{\partial D} \frac{( \tilde{\boldsymbol x}-\tilde{\boldsymbol y})\cdot\boldsymbol{\nu}_{\tilde{\boldsymbol{x}}}}{4\pi|\tilde{\boldsymbol x}-\tilde{\boldsymbol y}|^3}\breve{\varphi}(\tilde{\boldsymbol{y}})\mathrm{d}\sigma(\tilde{\boldsymbol{y}}),&& \tilde{\boldsymbol{x}}\in\partial D,\\
\mathscr{K}_{\partial D}[\breve{\varphi}](\tilde{\boldsymbol{x}})&=\int_{\partial D} \frac{( \tilde{\boldsymbol y}-\tilde{\boldsymbol x})\cdot\boldsymbol{\nu}_{\tilde{\boldsymbol{y}}}}{4\pi|\tilde{\boldsymbol x}-\tilde{\boldsymbol y}|^3}\breve{\varphi}(\tilde{\boldsymbol{y}})\mathrm{d}\sigma(\tilde{\boldsymbol{y}}),&& \tilde{\boldsymbol{x}}\in\partial D.
\end{align*}

On the other hand, the  fundamental solution of the Lam\'{e} system
\begin{align*}
\mathcal{L}_{\lambda,\mu}\breve{\mathbf{u}}(\tilde{\boldsymbol x})+\tilde{k}^2\breve{\mathbf{u}}(\tilde{\boldsymbol x})=0
\end{align*}
 is the Kupradze matrix $\boldsymbol\Gamma^{\tilde{k}}=(\Gamma^{\tilde{k}}_{ij})^{3}_{i,j=1}$ with $\tilde{k}\neq0$ given by \begin{align}\label{Fundamental2}
\boldsymbol\Gamma^{\tilde{k}}(\tilde{\boldsymbol x}) =-\frac{e^{\mathrm{i}\frac{\tilde{k}}{c_s}|\tilde{\boldsymbol x}|}}{4\pi\mu|\tilde{\boldsymbol x}|}\boldsymbol{\mathcal{I}}+\frac{1}{4\pi \tilde{k}^2}\nabla\nabla\big(\frac{e^{\mathrm{i}\frac{\tilde{k}}{c_p}|\tilde{\boldsymbol x}|}
-e^{\mathrm{i}\frac{\tilde{k}}{c_s}|\tilde{\boldsymbol x}|}}{|\tilde{\boldsymbol x}|}\big),
\end{align}
where $c_s=\sqrt{\mu},c_p=\sqrt{\lambda+2\mu}$ and $\boldsymbol{\mathcal{I}}$ is the $3\times 3$ identity matrix.
The $p$-wavenumber and $s$-wavenumber satisfy
\begin{align*}
\breve{k}_s=\frac{\tilde{k}}{c_s}=\frac{\tilde{k}}{\sqrt{\mu}},\quad\breve{k}_p=\frac{\tilde{k}}{c_p}=\frac{\tilde{k}}{\sqrt{\lambda+2\mu}}. \end{align*}
The fundamental solution to the elastostatic system is the Kelvin matrix $\boldsymbol\Gamma^{0}=(\Gamma^{0}_{ij})^{3}_{i,j=1}$ with
\begin{align*}
 \boldsymbol\Gamma^{0}(\tilde{\boldsymbol x})=-\frac{\kappa_1}{4\pi}\frac{1}{|\tilde{\boldsymbol x}|}\boldsymbol{\mathcal{I}}-\frac{\kappa_2}{4\pi}\frac{\tilde{\boldsymbol x} {\tilde{\boldsymbol x}}^\top}{|\tilde{\boldsymbol x}|^3},
\end{align*}
 where 
\begin{align*}
\kappa_1=\frac{1}{2}\big(\frac{1}{c^2_s}+\frac{1}{c^2_p}\big),\quad \kappa_2=\frac{1}{2}\big(\frac{1}{c^2_s}-\frac{1}{c^2_p}\big).
\end{align*}
For $\breve{\boldsymbol\varphi}\in L^{2}(\partial D)^3$, we define, respectively, the single layer potential and the Neumann-Poincar\'{e} operator associated with the fundamental solution $\boldsymbol\Gamma^{\tilde{k}}$ by
\begin{align}\label{Single2}
\mathbf{S}^{\tilde{k}}_{\partial D}[\breve{\boldsymbol\varphi}](\tilde{\boldsymbol x})&:=\int_{\partial D}\boldsymbol\Gamma^{\tilde{k}}(\tilde{\boldsymbol x},\tilde{\boldsymbol y})\breve{\boldsymbol\varphi}(\tilde{\boldsymbol y})\mathrm{d}\sigma(\tilde{\boldsymbol y}),&&\tilde{\boldsymbol x}\in\mathbb{R}^3,\\
\mathbf{K}^{\tilde{k},*}_{\partial D}[\breve{\boldsymbol\varphi}](\tilde{\boldsymbol x})&:=\mathrm{p.v.}\int_{\partial D} \frac{\partial}{\partial\boldsymbol{\nu}_{\tilde{\boldsymbol x}}}\boldsymbol\Gamma^{\tilde{k}}(\tilde{\boldsymbol x},\tilde{\boldsymbol y})\breve{\boldsymbol\varphi}(\tilde{\boldsymbol y})\mathrm{d}\sigma(\tilde{\boldsymbol y}),&& \tilde{\boldsymbol x}\in\partial{D},\label{NP2}
\end{align}
where  $\boldsymbol\Gamma^{\tilde{k}}(\tilde{\boldsymbol x},\tilde{\boldsymbol y}):=\boldsymbol\Gamma^{\tilde{k}}(\tilde{\boldsymbol x}-\tilde{\boldsymbol y})$.
The conormal derivative of the single layer potential enjoys the jump relation as follows:
\begin{align}\label{Jump2}
\frac{\partial}{\partial\boldsymbol{\nu}}\mathbf{S}^{\tilde{k}}_{\partial D}[\breve{\boldsymbol\varphi}]\big|_{\pm}(\tilde{\boldsymbol x})=\big(\pm\frac{1}{2}\mathbf{I}+\mathbf{K}^{\tilde{k},*}_{\partial D}\big)[\breve{\boldsymbol\varphi}](\tilde{\boldsymbol x}),\quad \tilde{\boldsymbol x}\in\partial{D},
\end{align}
where $\mathbf{I}$ is the identity operator. Moreover,  the  $L^2$-adjoint of the operator $\mathbf{K}^{k,*}_{\partial D}$ is given by
\begin{align*}
\mathbf{K}^{\tilde{k}}_{\partial D}[\breve{\boldsymbol\varphi}](\tilde{\boldsymbol x})=\mathrm{p.v.}\int_{\partial D} \frac{\partial}{\partial\boldsymbol{\nu}_{\tilde{\boldsymbol y}}}\boldsymbol\Gamma^{\tilde{k}}(\tilde{\boldsymbol x},\tilde{\boldsymbol y})\breve{\boldsymbol\varphi}(\tilde{\boldsymbol y})\mathrm{d}\sigma(\tilde{\boldsymbol y}),\quad \tilde{\boldsymbol x}\in\partial D.
\end{align*}

\subsection{Rescaling system}

Let us set
\begin{align*}
\delta:=\frac{\rho_b}{\rho_e},\quad\tau:=\frac{c_b}{\tilde{c}_{s}+\tilde{c}_p}=\frac{\sqrt{\kappa/\rho_b}}{\sqrt{\tilde{\mu}/\rho_e}+\sqrt{(\tilde{\lambda}+2\tilde{\mu})/\rho_e}},
\end{align*}
where $\delta$ states the contrast of the densities of the bubble and the elastic material and $\tau$ stands for the contrast of the  velocities of the bubble and the elastic material.
Then we introduce the following rescaling  terms and the non-dimensional parameters:
\begin{align}\label{E:transform}
\breve{\mathbf{u}}=\tilde{\mathbf{u}},\quad \breve{p}=\frac{\tilde{p}}{\rho_bc^2_b},\quad\mu=\frac{\tilde{\mu}}{\big(\sqrt{\tilde{\mu}}+\sqrt{\tilde{\lambda}+2\tilde{\mu}}\big)^2},\quad \lambda=\frac{\tilde{\lambda}}{\big(\sqrt{\tilde{\mu}}+\sqrt{\tilde{\lambda}+2\tilde{\mu}}\big)^2}.
\end{align}
By these transformations in \eqref{E:transform}, one reads that
\begin{align}\label{N:system}
\left\{ \begin{aligned}
&\mathcal{L}_{\lambda,\mu}\breve{\mathbf{u}}(\tilde{\boldsymbol x})+\tilde{k}^2\tau^2\breve{\mathbf{u}}(\tilde{\boldsymbol x})=-\boldsymbol\delta_{{\tilde{\boldsymbol x}}}(\mathbf{s})\omega^2{f}(\omega)\mathbf{p},\quad&&\tilde{\boldsymbol x}\in\mathbb{R}^3\backslash\overline{D},\\
&\Delta \breve{p}(\tilde{\boldsymbol x})+c(\omega)\tilde{k}^2\breve{p}(\tilde{\boldsymbol x})=0,&&\tilde{\boldsymbol x}\in{D},\\
&\nabla \breve{p}(\tilde{\boldsymbol x})\cdot \boldsymbol{\nu}-\tilde{k}^2\breve{\mathbf{u}}(\tilde{\boldsymbol x})\cdot \boldsymbol \nu=0,&&\tilde{\boldsymbol x}\in\partial D,\\
&\delta\tau^2\breve{p}(\tilde{\boldsymbol x})\boldsymbol{\nu}+\frac{\partial}{\partial{\boldsymbol{\nu}}}\breve{\mathbf{u}}(\tilde{\boldsymbol x})=0,&&\tilde{\boldsymbol x}\in\partial D,\\
&\breve{\mathbf{u}}-\breve{\mathbf{u}}^{\mathrm{in}} \text{ satisfies the radiation condition},
\end{aligned}\right.
\end{align}
where $\mathbf{p}=\frac{1}{\big(\sqrt{\tilde{\mu}}+\sqrt{\tilde{\lambda}+2\tilde{\mu}}\big)^2}\tilde{\mathbf{p}}$.

With the help of the layer-potential operators shown in Subsection \ref{sec:2.1}, the solution to \eqref{N:system} can be written as
\begin{align}\label{E:Solution}
\left\{ \begin{aligned}
&\breve{p}(\tilde{\boldsymbol x})=\mathscr{S}^{\tilde{k}_1}_{\partial D}[\breve{\varphi}_b](\tilde{\boldsymbol x}),\quad&&\tilde{\boldsymbol x}\in D,\\
&\breve{\mathbf{u}}(\tilde{\boldsymbol x})=\mathbf{S}^{\tilde{k}\tau}_{\partial D}[\breve{\boldsymbol{\varphi}}_e](\tilde{\boldsymbol x})+\breve{\mathbf{u}}^{\mathrm{in}}(\tilde{\boldsymbol x}),\quad && \tilde{\boldsymbol x}\in\mathbb{R}^3\backslash\overline{D}
\end{aligned}\right.
\end{align}
for the density $(\breve{\varphi}_b,\breve{\boldsymbol{\varphi}}_e)\in L^2(\partial D)\times L^2(\partial D)^3$, where
\begin{align*}
\tilde{k}_1=\tilde{k}\sqrt{c(\omega)}.
\end{align*}
Combining the transmission conditions in \eqref{N:system} across $\partial D$ with \eqref{Jump} and \eqref{Jump2}, the pair $(\breve{\varphi}_b,\breve{\boldsymbol{\varphi}}_e)$ is the unique solution to
\begin{align}\label{E:integral}
\left\{ \begin{aligned}
&\big(-\frac{1}{2}\mathcal{I}+\mathscr K^{\tilde{k}_1,*}_{\partial D}\big)[\breve{\varphi}_b](\tilde{\boldsymbol{x}})-\tilde{k}^2\boldsymbol\nu\cdot\mathbf{S}^{\tilde{k}\tau}_{\partial D}[\breve{\boldsymbol\varphi}_e](\tilde{\boldsymbol{x}})=\breve{\mathrm{F}}_1(\tilde{\boldsymbol{x}}),\quad&& \tilde{\boldsymbol{x}}\in\partial{D},\\
&\delta\tau^2\boldsymbol\nu \mathscr S^{\tilde{k}_1}_{\partial D}[\breve{\varphi}_b](\tilde{\boldsymbol{x}})+\big(\frac{1}{2}\mathbf{I}+\mathbf{K}^{\tilde{k}\tau,*}_{\partial D}\big)[\breve{\boldsymbol\varphi}_e](\tilde{\boldsymbol{x}})=\breve{\mathbf{F}}_2(\tilde{\boldsymbol{x}}),\quad&& \tilde{\boldsymbol{x}}\in\partial{D},
\end{aligned}\right.
\end{align}
where
\begin{align*}
\breve{\mathrm{F}}_1(\tilde{\boldsymbol{x}})=\tilde{k}^2\boldsymbol\nu\cdot \breve{\mathbf{u}}^{\mathrm{in}}(\tilde{\boldsymbol x}),\quad\breve{\mathbf{F}}_2(\tilde{\boldsymbol{x}})=-\frac{\partial}{\partial\boldsymbol\nu}\breve{\mathbf{u}}^{\mathrm{in}}(\tilde{\boldsymbol x}).
\end{align*}
Here an incident field to the system \eqref{N:system} can be written as
\begin{align}\label{E:incident}
\breve{\mathbf{u}}^{\mathrm{in}}(\tilde{\boldsymbol x})=-\omega^2f(\omega)\boldsymbol\Gamma^{\tilde{k}\tau}(\tilde{\boldsymbol x},\mathbf{s})\mathbf{p},\quad \tilde{\boldsymbol x}\in\mathbb{R}^3\setminus\overline{D}.
\end{align}
We also define
\begin{align*}
\breve{\mathbf{u}}^{\mathrm{sca}}(\tilde{\boldsymbol x}):=\breve{\mathbf{u}}(\tilde{\boldsymbol x})-\breve{\mathbf{u}}^{\mathrm{in}}(\tilde{\boldsymbol x}),\quad \tilde{\boldsymbol x}\in\mathbb{R}^3\setminus\overline{D}.
\end{align*}

In order to present the small parameter in \eqref{E:integral} explicitly, we introduce the transform $\tilde{\boldsymbol x}=\boldsymbol z+\epsilon \boldsymbol x$ with $\tilde{\boldsymbol x}\in\partial D,\boldsymbol x\in\partial B$. For  two function $\breve{\varphi},\breve{\boldsymbol\varphi}$ on $\partial D$, we define the corresponding functions on $\partial B$ by
\begin{align*}
\varphi(\boldsymbol x):=\breve{\varphi}(\boldsymbol z+\epsilon \boldsymbol x),\quad\boldsymbol\varphi(\boldsymbol x):=\breve{\boldsymbol\varphi}(\boldsymbol z+\epsilon \boldsymbol x).
\end{align*}
Moreover,
\begin{align*}
\mathbf{u}^{\mathrm{sca}}(\boldsymbol x):=\breve{\mathbf{u}}^{\mathrm{sca}}(\boldsymbol z+\epsilon \boldsymbol x).
\end{align*}
\begin{lemm}\label{le:scale}
For $\tilde{\boldsymbol x}\in \partial D,\boldsymbol x\in \partial B$, the following identities hold:
\begin{align*}
\mathscr{S}^{\tilde{k}}_{\partial D}[\breve{\varphi}](\tilde{\boldsymbol x})&=\epsilon \mathscr{S}^{k}_{\partial B}[\varphi](\boldsymbol x),\quad
\mathscr{K}^{\tilde{k},*}_{\partial D}[\breve{\varphi}](\tilde{\boldsymbol x})=\mathscr{K}^{k,*}_{\partial B}[\varphi](\boldsymbol x),\\
\mathbf{S}^{\tilde{k}}_{\partial D}[\breve{\boldsymbol\varphi}](\tilde{\boldsymbol x})&=\epsilon \mathbf{S}^{k}_{\partial B}[\boldsymbol\varphi](\boldsymbol x),\quad
\mathbf{K}^{\tilde{k},*}_{\partial D}[\breve{\boldsymbol\varphi}](\tilde{\boldsymbol x})=\mathbf{K}^{k,*}_{\partial B}[\boldsymbol\varphi](\boldsymbol x),
\end{align*}
where
\begin{align}\label{E:KKK}
k=\tilde{k}\epsilon.
\end{align}
\end{lemm}
\begin{proof}
These identities follow directly from the transform $\tilde{\boldsymbol x}=\boldsymbol z+\epsilon \boldsymbol x$ and \eqref{Single}, \eqref{NP}, \eqref{Single2} and \eqref{NP2}.
\end{proof}
We need to the invertibility of the  single layer potential $\mathscr{S}^{\tilde{k}\epsilon}_{\partial B}$ in three dimensions as well.
\begin{lemm}\label{le:inverse}
If $k>0$ is small enough, then the three-dimensional single layer potential $\mathscr{S}^{k}_{\partial B}: L^2(\partial B)\rightarrow L^{2}(\partial B)$ is invertible.
\end{lemm}
In view of Lemma \ref{le:scale},  the solution $(\breve{p},\breve{\mathbf{u}})$ in \eqref{E:Solution} can be rewritten as
\begin{align}\label{Solution2}
\left\{ \begin{aligned}
&p(\boldsymbol x)=\epsilon \mathscr{S}^{k_1}_{\partial B}[\varphi_b](\boldsymbol x),\quad&&\boldsymbol x\in B,\\
&\mathbf{u}(\boldsymbol x)=\epsilon\mathbf{S}^{k\tau}_{\partial B}[\boldsymbol{\varphi}_e](\boldsymbol x)+\mathbf{u}^{\mathrm{in}}(\boldsymbol x),\quad &&\boldsymbol x\in\mathbb{R}^3\backslash\overline{B},
\end{aligned}\right.
\end{align}
where $\mathbf{u}^{\mathrm{in}}(\boldsymbol x):=\breve{\mathbf{u}}^{\mathrm{in}}(\boldsymbol z+\epsilon \boldsymbol x)$ and
\begin{align}\label{E:KK}
k_1=\tilde{k}_1\epsilon=k\sqrt{c(\omega)}.
\end{align}
Then the system  \eqref{E:integral} becomes
\begin{align}\label{E:integral1}
\left\{ \begin{aligned}
&\big(-\frac{1}{2}\mathcal{I}+\mathscr K^{k_1,*}_{\partial B}\big)[\varphi_b](\boldsymbol{x})-\frac{k^2}{\epsilon}\boldsymbol\nu\cdot\mathbf{S}^{k\tau}_{\partial B}[\boldsymbol\varphi_e](\boldsymbol{x})=\mathrm{F}_1(\boldsymbol{x}),\quad&& \boldsymbol{x}\in\partial{B},\\
&\delta\tau^2\boldsymbol\nu \mathscr S^{k_1}_{\partial B}[\varphi_b](\boldsymbol{x})+\frac{1}{\epsilon}\big(\frac{1}{2}\mathbf{I}+\mathbf{K}^{k\tau,*}_{\partial B}\big)[\boldsymbol\varphi_e](\boldsymbol{x})=\frac{1}{\epsilon^2}\mathbf{F}_2(\boldsymbol{x}),\quad&& \boldsymbol{x}\in\partial{B},
\end{aligned}\right.
\end{align}
where
\begin{align*}
\mathrm{F}_1(\boldsymbol{x})=\tilde{k}^2\boldsymbol\nu\cdot \mathbf{u}^{\mathrm{in}}(\boldsymbol x),\quad\mathbf{F}_2(\boldsymbol{x})=-\frac{\partial}{\partial\boldsymbol\nu}\mathbf{u}^{\mathrm{in}}(\boldsymbol x).
\end{align*}
By the second term in \eqref{E:integral1} and Lemma \ref{le:inverse}, we obtain
\begin{align*}
\varphi_b(\boldsymbol x)=\frac{1}{\delta\tau^2\epsilon}(\mathscr S^{k_1}_{\partial B})^{-1}\left(-\boldsymbol\nu\cdot\big(\frac{1}{2}\mathbf{I}+\mathbf{K}^{k\tau,*}_{\partial B}\big)[\boldsymbol\varphi_e](\boldsymbol x)+\frac{1}{\epsilon}\boldsymbol\nu\cdot\mathbf{F}_2(\boldsymbol{x})\right).
\end{align*}
Inserting this into the first term in \eqref{E:integral1} yields that
\begin{align*}
\mathcal{A}(k,\delta)[\boldsymbol{\varphi}_e](\boldsymbol x)=\frac{1}{\epsilon}\mathrm{F}(\boldsymbol{x}),
\end{align*}
where
\begin{align}
\mathcal{A}(k,\delta)[\boldsymbol{\varphi}_e](\boldsymbol x)&=\left(\big(-\frac{1}{2}\mathcal{I}+\mathscr K^{k_1,*}_{\partial B}\big)(\mathscr S^{k_1}_{\partial B})^{-1}\boldsymbol\nu\cdot\big(\frac{1}{2}\mathbf{I}+\mathbf{K}^{k\tau,*}_{\partial B}\big)+\delta\tau^2k^2\boldsymbol\nu\cdot\mathbf{S}^{k\tau}_{\partial B}\right)[\boldsymbol\varphi_e](\boldsymbol x),\label{Op:A}\\
\mathrm{F}(\boldsymbol{x})&:=-{\delta\tau^2} \epsilon^2\mathrm{F}_1(\boldsymbol{x})+\big(-\frac{1}{2}\mathcal{I}+\mathscr K^{k_1,*}_{\partial B}\big)(\mathscr S^{k_1}_{\partial B})^{-1}[\boldsymbol\nu\cdot\mathbf{F}_2](\boldsymbol{x}).\label{F:incident}
\end{align}
This is an equivalent integral equation of the system \eqref{E:integral1}. We then will be devoted to studying the corresponding scattered field in time-harmonic regime.


\section{Modal decomposition of the time-harmonic field}\label{sec:3}

The purpose of this section  focuses on the modal decomposition of the acoustic-elastic wave scattering from a bubble-elastic structure. We first choose  an orthogonal basis of $L^2(\partial B)^3$. Then we approximate the incident field with a finite number of modes. Finally,  we obtain the finite modal expansion of the time-harmonic field. Generally, the method as seen in \cite{Ammari2022Asympototic,Baldassari2021Modal,CGL} is based on the research of the static system together with the perturbation theory. In our problem, there is a fact which transmission conditions across the boundary contain the wavenumber. In order to analyze the Minnaert resonance, we have to require higher order asymptotic expansions of the wavenumber, which cannot be given by the perturbation theory. Moreover,  we also need the asymptotic expansions of the corresponding spectrum of the layer-potential operators $\mathscr K^{k_1,*}_{\partial B}, \mathscr S^{k_1}_{\partial B},\boldsymbol\nu\cdot\big(\frac{1}{2}\mathbf{I}+\mathbf{K}^{k\tau,*}_{\partial B}\big)$ and $ \boldsymbol\nu\cdot\mathbf{S}^{k\tau}_{\partial B}$. For this, we  consider a  special form of the density $\boldsymbol{\varphi}_e$. More precisely, ansatz
\begin{align}\label{varphi-e}
\boldsymbol{\varphi}_e={\varphi}_e\boldsymbol\nu
\end{align}
with $\varphi_e \in L^2(\partial B)$.
Now let us study
\begin{align}\label{E:op-eq}
\mathcal{A}(k,\delta)[{\varphi}_e\boldsymbol\nu](\boldsymbol x)=\frac{1}{\epsilon}\mathrm{F}(\boldsymbol{x}),
\end{align}
where $\mathcal{A}(k,\delta)$ and $\mathrm{F}$ are as seen in \eqref{Op:A} and \eqref{F:incident}, respectively.


\subsection{Spectrum of layer-potential operators}\label{subsec:3.1}
Let us define the double factorial as follows
\begin{align*}
(2n+1)!!:=1\cdot 3\cdot~\cdots~\cdot (2n-1)\cdot(2n+1).
\end{align*}
Denote by $j_n(z),h_n(z),n\in \mathbb{N}$ the spherical Bessel and Hankel function of the first kind of order $n$, respectively. The following  series representation (cf. \cite{Colton2019Inverse}) shall be needed
\begin{align*}
j_n(z)&=\sum_{p=0}^{+\infty}\frac{(-1)^pz^{n+2p}}{2^pp!1\cdot3\cdot~\cdots~\cdot(2n+2p+1)},\\
y_n(z)&=-\frac{(2n)!}{2^n n!}\sum_{p=0}^{+\infty}\frac{(-1)^pz^{2p-n-1}}{2^pp!(-2n+1)\cdot(-2n+3)\cdot~\cdots~\cdot(-2n+2p-1)},\\
h_{n}(z)&=j_n(z)+\mathrm{i}y_n(z).
\end{align*}
Notice that $j_n$ is analytic on $0<|z|\ll1$ as well as $y_n$ is analytic on $0<|z|\ll1$.
One has
\begin{align}\label{E:NN}
j_n(z)=\frac{z^n}{(2n+1)!!}(1+\mathcal{O}(\frac1n)),\quad h_{n}(z)=\frac{(2n-1)!!}{\mathrm{i} z^{n+1}}(1+\mathcal{O}(\frac1n)),\quad n\rightarrow+\infty,
\end{align}
uniformly on $0<|z|\ll1$. Moreover, we obtain  that for fixed $n\in\mathbb{N}$ and $0<|z|\ll 1$,
\begin{align}
j_n(z)&=\frac{z^n}{(2n+1)!!}\left(1-\frac{z^2}{2(2n+3)}+\mathcal{O}(z^3)\right),\label{E:bessel}
\end{align}
and that for fixed $n\in\mathbb{N}$ and $0<|z|\ll 1$,
\begin{align}
h_0(z)&=\frac{1}{\mathrm{i} z}\left(1+\mathrm{i}z-\frac{z^2}{2}+\mathcal{O}(z^3)\right),\label{E:hankel1}\\
h_n(z)&=\frac{(2n-1)!!}{\mathrm{i}z^{n+1}}\left(1-\frac{z^2}{2(-2n+1)}+\mathcal{O}(z^3)\right),\quad n\geq1.\label{E:hankel}
\end{align}
Because of the series representations of $j_n$ and $y_n$,  the following two recurrence relations (cf. \cite{Colton2019Inverse}) hold:
\begin{align}
j_{n+1}(z)+j_{n-1}(z)&=\frac{2n+1}{z}j_n(z),\quad n=1,2,\cdots,\label{E:recu1}\\
y_{n+1}(z)+y_{n-1}(z)&=\frac{2n+1}{z}y_n(z),\quad n=1,2,\cdots,\nonumber
\end{align}
which then give
\begin{align}\label{E:recu2}
h_{n+1}(z)+h_{n-1}(z)&=\frac{2n+1}{z}h_n(z),\quad n=1,2,\cdots.
\end{align}
In the following,  we shall introduce the eigenfunctions and the corresponding eigenvalues  which will play  important roles in the analysis of modal expansion.

\begin{lemm}\label{le:harmonic}
Denote by $Y^m_n$ with $n\in \mathbb{N},|m|\leq n$ the spherical harmonic functions, which  form an orthogonal basis of $L^2(\partial B)$. Moreover, the eigenvalues of the Laplace-Beltrami operator $\Delta_{\partial B}$ associated with the eigenfunctions $Y^m_n$ are $-n(n+1)$.
\end{lemm}

The following lemma addresses the spectral system of $\mathscr{S}_{\partial B}^{k}$ and  $\mathscr{K}_{\partial B}^{k,*}$ as seen in \eqref{Single} and \eqref{NP}.
\begin{lemm}\label{le:SK}
The eigenvalues of  $\mathscr{S}_{\partial B}^{k}$ and  $\mathscr{K}_{\partial B}^{k,*}$ on $\partial B$ corresponding to the spherical harmonic functions $Y^m_n$ with $n\in \mathbb{N},|m|\leq n$ are given, respectively, by
\begin{align}\label{E:SP}
\mathscr{S}_{\partial B}^{k}[Y^m_n]=\xi_n(k)Y^m_n,\quad\mathscr{K}_{\partial B}^{k,*}[Y^m_n]=\zeta_n(k) Y^m_n
\end{align}
with
\begin{align}
\xi_n(k)&=-\mathrm{i}kj_n(k)h_n(k),\label{G:xin}\\
\zeta_n(k)&=\frac{1}{2}-\mathrm{i}k^2j'_n(k)h_{n}(k)=-\frac{1}{2}-\mathrm{i}k^2j_n(k)h'_n(k).\label{G:zetan}
\end{align}

Moreover, for fixed $n\in\mathbb{N}$ and $0<k\ll1$, $\xi_{n}(k)$ and $\zeta_n(k)$ have the following asymptotic expansions:
\begin{align}
\xi_0(k)&=-1-\mathrm{i}k+\frac{2}{3}k^2+\mathcal{O}(k^3),\label{E:xi0}\\
\xi_{n}(k)&=-\frac{1}{2n+1}-\frac{2}{(2n+3)(2n+1)(2n-1)}k^2+\mathcal{O}(k^3),\quad n\geq1,\label{E:xin}\\
\zeta_0(k)&=\frac{1}{2}+\frac{1}{3}k^2+\mathcal{O}(k^3),\label{E:zeta0}\\
\zeta_n(k)&=\frac{1}{2(2n+1)}-\frac{1}{(2n+3)(2n+1)(2n-1)}k^2+\mathcal{O}(k^3),\quad n\geq1.\label{E:zetan}
\end{align}
\end{lemm}
\begin{proof}
Formulas \eqref{E:SP}--\eqref{G:zetan}  follow from Theorem 3.2 in \cite{li2018anomalous}. In addition, thanks to \eqref{E:bessel}--\eqref{E:hankel}, one has \eqref{E:xi0}--\eqref{E:zetan}.

The proof is complete.
\end{proof}

Next let us  define the vectorial spherical harmonics of order $n$ as follows:
\begin{align*}
\mathcal{I}^m_n&=\nabla_{\partial B} Y^m_{n+1}+(n+1)Y^m_{n+1}\boldsymbol \nu, &&n\geq0,|m|\leq n+1,\\
\mathcal{T}^m_n&=\nabla_{\partial B}Y^m_{n+1}\times\boldsymbol\nu,&&n\geq 1, |m|\leq n,\\
\mathcal{N}^m_n&=-\nabla_{\partial B}Y^m_{n-1}+nY^m_{n-1}\boldsymbol\nu,&&n\geq1,|m|\leq n-1.
\end{align*}
The family $(\mathcal{I}^m_n,\mathcal{T}^m_n,\mathcal{N}^m_n)$ forms an orthogonal basis of $L^2(\partial B)^3$.

Observe that
\begin{align*}
\boldsymbol\nu\cdot\mathcal{I}^{m}_{n}&=\boldsymbol\nu\cdot\big(\nabla_{\partial B} Y^m_{n+1}+(n+1)Y^m_{n+1}\boldsymbol \nu\big)=(n+1)Y^m_{n+1},\\
\boldsymbol\nu\cdot\mathcal{T}^{m}_{n}&=\boldsymbol\nu\cdot\big(\nabla_{\partial B}Y^m_{n+1}\times\boldsymbol\nu\big)=0,\\
\boldsymbol\nu\cdot\mathcal{N}^{m}_{n}&=\boldsymbol\nu\cdot\big(-\nabla_{\partial B}Y^m_{n-1}+nY^m_{n-1}\boldsymbol\nu\big)=nY^m_{n-1}.
\end{align*}

The following Lemmas \ref{le:S} and \ref{le:traction} address some facts shown in \cite{deng2020spectral} to be used in the sequel.
\begin{lemm}\label{le:S}
The single layer potential operator $\mathbf{S}_{\partial B}^{k}$ on $\partial B$ satisfies that
\begin{align}
\mathbf{S}^{k}_{\partial B}[\mathcal {T}^{m}_n]&=b_n(k)\mathcal {T}^{m}_n,\nonumber\\
\mathbf{S}^{k}_{\partial B}[\mathcal{I}_{n-1}^m]&=c_{1n}(k)\mathcal{I}^m_{n-1}+d_{1n}(k)\mathcal {N}^m_{n+1},\label{E:SI}\\
\mathbf{S}^{k}_{\partial B}[\mathcal {N}^m_{n+1}]&=c_{2n}(k)\mathcal{I}^{m}_{n-1}+d_{2n}(k)\mathcal{N}^{m}_{n+1}.\label{E:SN}
\end{align}
Here
\begin{align*}
b_n(k)&=-\frac{\mathrm{i}k_s j_n(k_s)h_n(k_s)}{\mu},\\
c_{1n}(k)&=-\mathrm{i}\left(\frac{j_{n-1}(k_s)h_{n-1}(k_s)k_s(n+1)}{\mu(2n+1)}+\frac{j_{n-1}(k_p)h_{n-1}(k_p)k_p n}{(\lambda+2\mu)(2n+1)}\right),\\
d_{1n}(k)&=-\mathrm{i}\left(\frac{j_{n-1}(k_s)h_{n+1}(k_s)k_sn}{\mu(2n+1)}-\frac{j_{n-1}(k_p)h_{n+1}(k_p)k_p n}{(\lambda+2\mu)(2n+1)}\right),\\
c_{2n}(k)&=-\mathrm{i}\left(\frac{j_{n+1}(k_s)h_{n-1}(k_s)k_s(n+1)}{\mu(2n+1)}-\frac{j_{n+1}(k_p)h_{n-1}(k_p)k_p (n+1)}{(\lambda+2\mu)(2n+1)}\right),
\end{align*}
and
\begin{align*}
d_{2n}(k)&=-\mathrm{i}\left(\frac{j_{n+1}(k_s)h_{n+1}(k_s)k_sn}{\mu(2n+1)}+\frac{j_{n+1}(k_p)h_{n+1}(k_p)k_p(n+1) }{(\lambda+2\mu)(2n+1)}\right),
\end{align*}
where
\begin{align}\label{E:wave-ps}
k_s=\frac{k}{c_s}=\frac{k}{\sqrt{\mu}},\quad k_p=\frac{k}{c_p}=\frac{k}{\sqrt{\lambda+2\mu}}.
\end{align}
\end{lemm}
\begin{lemm}\label{le:traction}
The tractions of the single layer potentials $\mathbf{S}^{k}_{\partial B}[\mathcal {T}^{m}_n],\mathbf{S}^{k}_{\partial B}[\mathcal {I}^{m}_{n-1}]$ and $\mathbf{S}^{k}_{\partial B}[\mathcal {N}^{m}_{n+1}]$ on $\partial B$ satisfy that
\begin{align}
\frac{\partial}{\partial\boldsymbol{\nu}}\mathbf{S}^{k}_{\partial B}[\mathcal{T}^{m}_{n}]\big|_{+}&=\mathfrak{b}_{n}(k)\mathcal{T}^{m}_{n},\nonumber\\
\frac{\partial}{\partial\boldsymbol{\nu}}\mathbf{S}^{k}_{\partial B}[\mathcal{I}^{m}_{n-1}]\big|_{+}&=\mathfrak{c}_{1n}(k)\mathcal{I}^{m}_{n-1}+\mathfrak{d}_{1n}(k)\mathcal{N}^m_{n+1},\label{E:PSI}\\
\frac{\partial}{\partial\boldsymbol{\nu}}\mathbf{S}^{k}_{\partial B}[\mathcal{N}^{m}_{n+1}]\big|_{+}&=\mathfrak{c}_{2n}(k)\mathcal{I}^{m}_{n-1}+\mathfrak{d}_{2n}(k)\mathcal{N}^m_{n+1}.\label{E:PSN}
\end{align}
Here
\begin{align*}
\mathfrak{b}_n(k)=&-\mathrm{i}k_sj_n(k_s)(k_sh'_{n}(k_s)-h_n(k_s)),\\
\mathfrak{c}_{1n}(k)=&-2(n-1)\mathrm{i}\left(\frac{j_{n-1}(k_s)h_{n-1}(k_s)k_s(n+1)}{2n+1}+\frac{j_{n-1}(k_p)h_{n-1}(k_p)k_p\mu n}{(\lambda+2\mu)(2n+1)}\right)\\
&+\mathrm{i}\left(\frac{j_{n-1}(k_s)h_{n}(k_s)k^2_s(n+1)+j_{n-1}(k_p)h_{n}(k_p)k^2_pn}{2n+1}\right),\\
\mathfrak{d}_{1n}(k)=&2n(n+2)\mathrm{i}\left(\frac{j_{n-1}(k_s)h_{n+1}(k_s)k_s}{2n+1}-\frac{j_{n-1}(k_p)h_{n+1}(k_p)k_p\mu }{(\lambda+2\mu)(2n+1)}\right)\\
&+n\mathrm{i}\left(\frac{-j_{n-1}(k_s)h_{n}(k_s)k^2_s+j_{n-1}(k_p)h_{n}(k_p)k^2_p}{2n+1}\right),\\
\mathfrak{c}_{2n}(k)=&-2(n^2-1)\mathrm{i}\left(\frac{j_{n+1}(k_s)h_{n-1}(k_s)k_s}{2n+1}-\frac{j_{n+1}(k_p)h_{n-1}(k_p)k_p\mu}{(\lambda+2\mu)(2n+1)}\right)\\
&-(n+1)\mathrm{i}\left(\frac{-j_{n-1}(k_s)h_{n}(k_s)k^2_s+j_{n-1}(k_p)h_{n}(k_p)k^2_p}{2n+1}\right),
\end{align*}
and
\begin{align*}
\mathfrak{d}_{2n}(k)=&2(n+2)\mathrm{i}\left(\frac{j_{n+1}(k_s)h_{n+1}(k_s)k_sn}{2n+1}+\frac{j_{n+1}(k_p)h_{n+1}(k_p)k_p\mu(n+1)}{(\lambda+2\mu)(2n+1)}\right)\\
&-\mathrm{i}\left(\frac{j_{n+1}(k_s)h_{n}(k_s)k^2_sn+j_{n+1}(k_p)h_{n}(k_p)k^2_p(n+1)}{2n+1}\right),
\end{align*}
where $k_s$ and $k_p$ are given in \eqref{E:wave-ps}.
\end{lemm}

In the following two lemmas we present the properties of the operators $\boldsymbol\nu \cdot\mathbf{S}_{\partial B}^{k}$ and $\boldsymbol\nu \cdot\big(\frac{1}{2}\mathbf{I}+\mathbf{K}^{k,*}_{\partial B}\big)$ on $\partial B$, respectively.

\begin{lemm}\label{le:NUS}
The operator $\boldsymbol\nu \cdot\mathbf{S}_{\partial B}^{k}$ on $\partial B$ satisfies that
\begin{align}\label{E:NUS}
\boldsymbol\nu \cdot\mathbf{S}_{\partial B}^{k}[Y^m_n \boldsymbol\nu]=\eta_n(k)Y^m_n
\end{align}
with
\begin{align}
\eta_n(k)
=\frac{n(c_{1n}(k)+c_{2n}(k))+(n+1)(d_{1n}(k)+d_{2n}(k))}{2n+1},\label{G:etan}
\end{align}
where $c_{1n}(k), c_{2n}(k), d_{1n}(k)$ and $d_{2n}(k)$ are defined in Lemma \ref{le:S}.

Furthermore, for fixed $n\in\mathbb{N}$ and $0<k\ll1$, $\eta_{n}(k)$ has the following asymptotic expansion:
\begin{align}
\eta_0(k)&=-\frac{1}{3(\lambda+2\mu)}-\frac{2}{15(\lambda+2\mu)}k^2_p+\mathcal{O}(k^3),\label{S:eta0}\\
\eta_n(k)&=-\frac{2(\lambda+\mu)n(n+1)+\mu(4n^4+4n-1)}{\mu(\lambda+2\mu)(2n+3)(2n+1)(2n-1)}+\mathcal{O}(k),\quad n\geq1.\label{S:eta}
\end{align}
\end{lemm}
\begin{proof}
See Appendix \ref{App:A.1}.
\end{proof}

\begin{lemm}\label{le:NUK}
The operator $\boldsymbol\nu \cdot\big(\frac{1}{2}\mathbf{I}+\mathbf{K}^{k,*}_{\partial B}\big)$ on $\partial B$ satisfies that
\begin{align}\label{E:NUK}
\boldsymbol\nu \cdot\big(\frac{1}{2}\mathbf{I}+\mathbf{K}^{k,*}_{\partial B}\big)[Y^m_n \boldsymbol\nu]=\rho_n(k)Y^m_n
\end{align}
with
\begin{align}
\rho_n(k)=&\frac{n(\mathfrak{c}_{1n}(k)+\mathfrak{c}_{2n}(k))+(n+1)(\mathfrak{d}_{1n}(k)+\mathfrak{d}_{2n}(k))}{2n+1},
\label{G:rhon}
\end{align}
where $\mathfrak{c}_{1n}(k),\mathfrak{c}_{2n}(k),\mathfrak{d}_{1n}(k)$ and $\mathfrak{d}_{2n}(k)$ are as seen in Lemma \ref{le:traction}.

In addition, for fixed $n\in\mathbb{N}$ and $0<k\ll1$, $\rho_{n}(k)$ has the following asymptotic expansion:
\begin{align}
\rho_0(k)=&\frac{4\mu}{3(\lambda+2\mu)}-\frac{5\lambda+2\mu}{15(\lambda+2\mu)}k^2_p+\mathcal{O}(k^3),\label{S:rho0}\\
\rho_n(k)=&\frac{(\lambda+\mu)n(8n^3+16n^2+4n-1)+\mu(2n+1)(4n^3+12n^2+5n-4)}{(\lambda+2\mu)(2n+3)(2n+1)^2(2n-1)}\nonumber\\
&+\frac{-2n(n+1)(4n^2+4n+33)}{(2n+5)(2n+3)(2n+1)^3(2n-1)(2n-3)}k^2_s\nonumber\\
&+\frac{-2\mu(2n+1)(10n^3+15n^2-19n-12)+(\lambda+2\mu)(2n+5)(2n-3)}{(\lambda+2\mu)(2n+5)(2n+3)(2n+1)^3(2n-1)(2n-3)}k^2_p\nonumber\\
&+\mathcal{O}(k^3),\quad n\geq1.\label{S:rho}
\end{align}
\end{lemm}
\begin{proof}
See Appendix \ref{App:A.2}.
\end{proof}
\begin{rema}
By the asymptotic expansions \eqref{E:xi0}--\eqref{E:zetan}, \eqref{S:eta0}, \eqref{S:eta}, \eqref{S:rho0} and \eqref{S:rho} , if $k$ towards to $0^+$,  then the results given by Lemmas \ref{le:SK}, \ref{le:NUS} and \ref{le:NUK} correspond to the static case.
\end{rema}

\subsection{Time-harmonic approximation of the scattered field}
Based on the above preparations in Subsection \ref{subsec:3.1}, we can obtain the modal decomposition of the operator  $\mathcal{A}(k,\delta)$    defined by  \eqref{Op:A} acting on $L^2(\partial B)^3$.
\begin{prop}\label{prop:A}
If $\boldsymbol \varphi_e$ has the form given by \eqref{varphi-e}, then the operator  $\mathcal{A}(k,\delta)$   acting on $L^2(\partial B)^3$ has the modal decomposition:
\begin{align*}
\mathcal{A}(k,\delta)[\boldsymbol{\varphi}_e]=\sum_{n=0}^{+\infty}\sum_{m=-n}^{n}\lambda_n(k)\langle \varphi_e,Y^{m}_{n}\rangle_{L^{2}(\partial B)}Y^{m}_{n}
\end{align*}
with
\begin{align}\label{E:lambda}
\lambda_{n}(k)=(-\frac{1}{2}+\zeta_n(k_1))\xi^{-1}_n(k_1)\rho_n(k\tau)+\delta\tau^2k^2\eta_n(k\tau),
\end{align}
where $\xi_{n}(k_1),\zeta_{n}(k_1),\eta_n(k\tau),\rho_n(k\tau)$ are defined in \eqref{G:xin}, \eqref{G:zetan}, \eqref{G:etan} and \eqref{G:rhon}, respectively.
\end{prop}
\begin{proof}
According to \eqref{varphi-e} and Lemma \ref{le:harmonic}, one has 
\begin{align*}
\boldsymbol{\varphi}_e={\varphi}_e\boldsymbol\nu=\sum_{n=0}^{+\infty}\sum_{m=-n}^{n}\langle\varphi_e,Y^m_n\rangle_{L^2(\partial B)} Y^m_n\boldsymbol\nu.
\end{align*}
With the help the expression \eqref{Op:A} of $\mathcal{A}(k,\delta)$, Lemmas \ref{le:SK}, \ref{le:NUS} and \ref{le:NUK}, we obtain the conclusion of the lemma.

The proof is complete.
\end{proof}

Recall that $B$ is a unit ball containing the origin in $\mathbb{R}^3$.  Denote by $H^s(\partial B)$ the  Sobolev space equipped with the corresponding norm (cf. \cite{liu2013enhanced})
\begin{align*}
\|\psi\|^2_{H^s(\partial B)}=\sum^{+\infty}_{n=0}\sum^n_{m=-n}(1+n(n+1))^s|\langle\psi,Y^{m}_{n}\rangle_{L^{2}(\partial B)}|^2,
\end{align*}
where
\begin{align*}
\psi=\sum^{+\infty}_{n=0}\sum^n_{m=-n}\langle\psi,Y^{m}_{n}\rangle_{L^{2}(\partial B)}Y^{m}_{n}.
\end{align*}
Moreover, using \eqref{E:NN} yields that for $n\gg1$ and $0<k\ll1$,
\begin{align}\label{E:zetaN}
\zeta_n(k)\stackrel{\eqref{G:zetan}}{=}\mathcal{O}(1/n).
\end{align}
For each real number $s\geq0$, there exist $c>0,C>0$ such that for $\psi\in H^{s}(\partial B)$,
\begin{align}\label{E:estimate}
c\|\psi\|_{H^{s}(\partial B)}\leq\|\mathscr{K}^{k,*}_{\partial B}[\psi]\|_{H^{s+1}(\partial B)}\leq C\|\psi\|_{H^{s}(\partial B)}.
\end{align}

Let us establish the decay estimates of Fourier coefficients.

\begin{prop}
There exists an integer $S\geq1$ such that if $\mathrm{F}\in H^{S}(\partial B)$, then the corresponding Fourier coefficients  satisfy that
\begin{align}
\langle\mathrm{F},Y^{m}_{n}\rangle_{L^{2}(\partial B)}=o(n^{-\frac{S}{2}})\quad\text{as $n\rightarrow+\infty$}.\label{E:decay}
\end{align}
\end{prop}
\begin{proof}
Denote by $\ker(\mathscr{K}^{k,*}_{\partial B})$ the kernel of $\mathscr{K}^{k,*}_{\partial B}$. It is clear to see that
\begin{align*}
H^{s}(\partial B)=\mathscr{K}^{k,*}_{\partial B}(H^{s-1}(\partial B))\oplus \ker(\mathscr{K}^{k,*}_{\partial B}),
\end{align*}
where the symbol $\oplus $ is understood in the sense of $L^2$ scalar product.  Since $\mathrm{F}\in H^{S}(\partial B)$, we obtain that for  $\mathrm{G}^{(1)}\in H^{S-1}(\partial B)$,
\begin{align*}
\mathrm{F}=\mathscr{K}^{k,*}_{\partial B}(\mathrm{G}^{(1)})+\mathrm{F}^{(1)}_{\ker}.
\end{align*}
A similar argument yields that $\mathrm{G}^{(1)}=\mathscr{K}^{k,*}_{\partial B}(\mathrm{G}^{(2)})+\mathrm{F}^{(2)}_{\ker}$, which then gives
\begin{align}\label{E:decomposition}
\mathrm{F}=(\mathscr{K}^{k,*}_{\partial B})^{S}(\mathrm{G}^{(S)})+\mathrm{F}^{(1)}_{\ker}.
\end{align}

Furthermore, for $n\gg1,-n\leq m\leq n$,
\begin{align*}
\langle\mathrm{F},Y^{m}_{n}\rangle_{L^{2}(\partial B)}&=\langle\mathscr{K}^{k,*}_{\partial B}(\mathrm{G}^{(1)}),Y^{m}_{n}\rangle_{L^{2}(\partial B)}+\langle\mathrm{F}^{(1)}_{\ker},Y^{m}_{n}\rangle_{L^{2}(\partial B)}\\
&=\zeta_{n}(k)\langle\mathrm{G}^{(1)},Y^{m}_{n}\rangle_{L^{2}(\partial B)}+\langle\mathrm{F}^{(1)}_{\ker},Y^{m}_{n}\rangle_{L^{2}(\partial B)}\\
&=\zeta_{n}(k)\langle\mathrm{G}^{(1)},Y^{m}_{n}\rangle_{L^{2}(\partial B)},
\end{align*}
which leads to
\begin{align*}
\langle\mathrm{F},Y^{m}_{n}\rangle_{L^{2}(\partial B)}=(\zeta_{n}(k))^{S}\langle\mathrm{G}^{(S)},Y^{m}_{n}\rangle_{L^{2}(\partial B)}.
\end{align*}

Because of \eqref{E:decomposition}, we  can deduce
\begin{align*}
\mathscr{K}^{k,*}_{\partial B}(\mathrm{F})=(\mathscr{K}^{k,*}_{\partial B})^{S+1}(\mathrm{G}^{(S)}).
\end{align*}
Then using \eqref{E:estimate} yields that
\begin{align*}
\|(\mathscr{K}^*_{\partial B})^{S+1}(\mathrm{G}^{(S)})\|_{H^{S+1}(\partial B)}\leq C\|\mathrm{F}\|_{H^{S}(\partial B)}.
\end{align*}
In addition,
\begin{align*}
\|\mathrm{G}^{(S)}\|_{L^{2}(\partial B)}\stackrel{\eqref{E:estimate}}{\leq}\frac{1}{c^{S+1}}\|(\mathscr{K}^{k,*}_{\partial B})^{S+1}(\mathrm{G}^{(S)})\|_{{H^{S+1}}(\partial B)}
\leq \frac{C}{c^{S+1}} \|\mathrm{F}\|_{H^{S}(\partial B)}.
\end{align*}
Hence it follows from \eqref{E:zetaN} that
\begin{align*}
|\langle\mathrm{F},Y^{m}_{n}\rangle_{L^{2}(\partial B)}|&\leq|\zeta_{n}(k)|^{S}\|\mathrm{G}^{(S)}\|_{L^{2}(\partial B)}\|Y^{m}_{n}\|_{L^{2}(\partial B)}\\
&\leq \frac{C}{c^{S+1}}|\zeta_{n}(k)|^{S}\|\mathrm{F}\|_{H^{S}(\partial B)}\\
&=o(n^{-\frac{S}{2}})
\end{align*}
as $n\rightarrow+\infty$. Thus formula \eqref{E:decay} holds.

The proof is complete.
\end{proof}

Let us give the modal approximation of the scattered field for the system  \eqref{E:op-eq}.
\begin{prop}\label{Prop:scatting}
Suppose that for an integer $N>1$ large enough,
\begin{align}\label{hy}
\mathrm{F}=\sum_{n=0}^{N}\sum_{m=-n}^{n}\langle\mathrm{F},Y^{m}_{n}\rangle_{L^{2}(\partial B)}Y^{m}_{n}.
\end{align}
If $\boldsymbol \varphi_e$ has the form given by \eqref{varphi-e}, then the modal approximation of the scattered field satisfies that for $0< k\ll1$ and $\boldsymbol x\in\mathbb{R}^3\backslash \overline{B}$,
\begin{align*}
\mathbf{u}^{\mathrm{sca}}(\boldsymbol x)
=&\sum_{n=0}^{N}\sum_{m=-n}^{n}\frac{1}{\lambda_{n}(k)}\langle\mathrm{F},Y^{m}_{n}\rangle_{L^{2}(\partial B)}\mathbf{S}^{k\tau}_{\partial B}[Y^{m}_{n}\boldsymbol \nu](\boldsymbol x),
\end{align*}
where $\lambda_{n}(k),n\geq0$ are as seen in \eqref{E:lambda}.
\end{prop}
\begin{proof}
In view of Proposition \ref{prop:A}, we deduce that
\begin{align*}
\sum_{n=0}^{+\infty}\sum_{m=-n}^{n}\lambda_n(k)\langle \varphi_e,Y^{m}_{n}\rangle_{L^{2}(\partial B)}Y^{m}_{n}=\frac{1}{\epsilon}\sum_{n=0}^{N}\sum_{m=-n}^{n}\langle\mathrm{F},Y^{m}_{n}\rangle_{L^{2}(\partial B)}Y^{m}_{n}.
\end{align*}
One arrives at
\begin{align*}
\langle \varphi_e,Y^{m}_{n}\rangle_{L^{2}(\partial B)}&=\frac{1}{\epsilon}\frac{1}{\lambda_n(k)}\langle\mathrm{F},Y^{m}_{n}\rangle_{L^{2}(\partial B)},&&\text{if } 0\leq n\leq N,|m|\leq n,\\
\langle \varphi_e,Y^{m}_{n}\rangle_{L^{2}(\partial B)}&=0,&&\text{if } n> N,|m|\leq n.
\end{align*}
Therefore,
\begin{align*}
\boldsymbol{\varphi}_e={\varphi}_e\boldsymbol\nu=\sum_{n=0}^{N}\sum_{m=-n}^{n}\frac{1}{\epsilon}\frac{1}{\lambda_n(k)}\langle\mathrm{F},Y^{m}_{n}\rangle_{L^{2}(\partial B)} Y^m_n\boldsymbol\nu.
\end{align*}
By the second expression in \eqref{Solution2}
, we get the conclusion of the lemma.

The proof is complete.
\end{proof}

The remainder of this section aims  to return to the original unscaled problem. We first need the following lemma.
\begin{lemm}\label{le:scalar}
If $\breve{f},\breve{g}$ are defined on $\partial D$, there are the  corresponding functions $f,g$ defined on $\partial B$, respectively. Then
\begin{align*}
\langle \breve{f},\breve{g}\rangle_{{L^{2}(\partial D)}}&=\epsilon^2\langle f,g\rangle_{{L^{2}(\partial B)}},\\
\|\breve{f}\|_{{L^{2}(\partial D)}}&=\epsilon\|f\|_{{L^{2}(\partial B)}}.
\end{align*}
\end{lemm}
\begin{proof}
Notice that
\begin{align*}
\langle \breve{f},\breve{g}\rangle_{L^{2}(\partial D)}=&\int_{\partial D}\breve{f}(\tilde{\boldsymbol x})\cdot\overline{\breve{g}(\tilde{\boldsymbol x})}\mathrm{d}\sigma(\tilde{\boldsymbol x})\\
=&\epsilon^2\int_{\partial B}f(\boldsymbol z+\epsilon\boldsymbol x)\cdot\overline{ g(\boldsymbol z+\epsilon\boldsymbol x)}\mathrm{d}\sigma(\boldsymbol x)=\epsilon^2\langle f,g\rangle_{L^{2}(\partial B)}.
\end{align*}
Hence $\|\breve{f}\|_{L^{2}(\partial D)}=\epsilon\|f\|_{L^{2}(\partial B)}$.

This ends the proof of the lemma.
\end{proof}

For each  spherical harmonic function $Y^m_n$ on $\partial B$, we consider the following function on $\partial D$:
\begin{align*}
\breve{Y}^m_n(\tilde{\boldsymbol x})=\frac{1}{\epsilon}Y^m_n\big(\frac{\tilde{\boldsymbol x}-\boldsymbol z}{\epsilon}\big).
\end{align*}
It is clear that $\|\breve{Y}^{m}_{n}\|_{{L^{2}(\partial D)}}=1$. Then $\breve{Y}^m_n,n\geq0,|m|\leq n$ are the normalized orthogonal functions.

The next proposition corresponds to the original unscaled problem.

\begin{prop}\label{prop:approximation}
If  $N>1$ is an integer large enough, then the model approximation of the scattered field for the system \eqref{N:system}  satisfies that for $0< k\ll1$ and $\tilde{\boldsymbol x}\in\mathbb{R}^3\backslash \overline{D}$,
\begin{align}\label{E:SCA}
\breve{\mathbf{u}}^{\mathrm{sca}}(\tilde{\boldsymbol x})
=&\frac{1}{\epsilon}\sum_{n=0}^{N}\sum_{m=-n}^{n}\frac{1}{\lambda_{n}(k)}\langle\breve{\mathrm{F}},\breve{Y}^{m}_{n}\rangle_{L^{2}(\partial D)}\mathbf{S}^{\tilde{k}\tau}_{\partial D}[\breve{Y}^{m}_{n}\boldsymbol\nu](\tilde{\boldsymbol x}),
\end{align}
where $\lambda_n(k),n\geq0$ are defined by \eqref{E:lambda}, and
\begin{align}\label{D:F}
\breve{\mathrm{F}}(\tilde{\boldsymbol x}):=\mathrm{F}(\frac{\tilde{\boldsymbol x}-\boldsymbol z}{\epsilon})
\end{align}
with $\mathrm{F}$ as seen in \eqref{F:incident}.
\end{prop}
\begin{proof}
It follows from Lemmas \ref{le:scale} and \ref{le:scalar} that
\begin{align*}
\mathbf{S}^{k\tau}_{\partial B}[Y^{m}_{n}\boldsymbol\nu](\boldsymbol x)&=\mathbf{S}^{\tilde{k}\tau}_{\partial D}[\breve{Y}^{m}_{n}\boldsymbol\nu](\tilde{\boldsymbol x}),\\
\langle\mathrm{F}, {Y}^{m}_{n}\rangle_{L^{2}(\partial B)}&=\frac{1}{\epsilon}\langle\breve{\mathrm{F}}, \breve{{Y}}^{m}_{n}\rangle_{L^{2}(\partial D)}.
\end{align*}
Combining two terms as above with Proposition \ref{Prop:scatting} yields the conclusion of the lemma.

The proof is complete.
\end{proof}


\section{Minnaert resonances}\label{sec:4}

This section is devoted to calculating size and frequency dependent Minnaert resonances. We first show the asymptotic expansion of $\lambda_n(k)$.
\begin{lemm}\label{le:resonance}
For fixed $n\in\mathbb{N}$ and $0<k\ll1$, $\lambda_{n}(k)$ defined by \eqref{E:lambda} has the following asymptotic expansion:
\begin{align}
\lambda_n(k)=&\tilde{\lambda}_{n}+k^2\tilde{\lambda}_{n,1}
+\mathcal{O}(k^3),\label{lan:ex}
\end{align}
where
\begin{align*}
\tilde{\lambda}_0=&0,\\
\tilde{\lambda}_n=&\frac{(\lambda+\mu)n^2(8n^3+16n^2+4n-1)+\mu n(2n+1)(4n^3+12n^2+5n-4)}{(\lambda+2\mu)(2n+3)(2n+1)^2(2n-1)},\quad n\geq1,\\
\tilde{\lambda}_{0,1}=&-\frac{4\mu c(\omega)}{9(\lambda+2\mu)}-\frac{\delta\tau^2}{3(\lambda+2\mu)},\\
\tilde{\lambda}_{n,1}=&\frac{-2n^2(n+1)(4n^2+4n+33)}{(2n+5)(2n+3)(2n+1)^3(2n-1)(2n-3)}\frac{\tau^2}{\mu}\nonumber\\
&+\frac{-2\mu n(2n+1)(10n^3+15n^2-19n-12)+(\lambda+2\mu)n(2n+5)(2n-3)}{(\lambda+2\mu)(2n+5)(2n+3)(2n+1)^3(2n-1)(2n-3)}\frac{\tau^2}{\lambda+2\mu}\nonumber\\
&-\frac{c(\omega)(\lambda+\mu)n(8n^3+16n^2+4n-1)+\mu(2n+1)(4n^3+12n^2+5n-4)}{(\lambda+2\mu)(2n+3)^2(2n+1)^2(2n-1)}\nonumber\\
&-\delta\frac{2(\lambda+\mu)n(n+1)+\mu(4n^4+4n-1)}{(\lambda+2\mu)(2n+3)(2n+1)(2n-1)}\frac{\tau^2}{\mu}, \quad n\geq1.
\end{align*}
\end{lemm}
\begin{proof}
See Appendix \ref{App:A.3}.
\end{proof}
Moreover, because of  \eqref{c:omega}, \eqref{E:KKK} and \eqref{E:KK}, one  has
 \begin{align*}
k=\omega\epsilon/c_b,\quad k_1=\omega\epsilon\sqrt{c(\omega)}/c_b.
\end{align*}
Let $\lambda_{n}(\omega\epsilon c^{-1}_b)$ be given by \eqref{lan:ex}. Define a static Minnaert resonance as follows:
\begin{defi}
Set $n\in\{0,1,2,\cdots,N\}$ for some positive integer $N$. If $|\tilde{\lambda}_{n}|=0$, then the corresponding frequency $\omega$ is called a static Minnaert resonance.
\end{defi}

\begin{defi}
Set $n\in\{0,1,2,\cdots,N\}$ for some positive integer $N$. The frequency $\omega$ is called a first-order corrected Minnaert resonance  if $|\tilde{\lambda}_n+(\omega\epsilon c^{-1}_b)^2\tilde{\lambda}_{n,1}|=0$.
\end{defi}
\begin{rema}
We can verify that $\tilde{\lambda}_{n}$ with $1\leq n\leq N$ are of size one. In fact, if $1\leq n\leq N$, then
\begin{align*}
\tilde{\lambda}_n=\frac{(\lambda+\mu)n^2(8n^3+16n^2+4n-1)+\mu n(2n+1)(4n^3+12n^2+5n-4)}{(\lambda+2\mu)(2n+3)(2n+1)^2(2n-1)}>0.
\end{align*}
Based on the above fact, we can exclude the modes $n=1,2,\cdots, N$ from the set of resonances.

Furthermore, if   a first-order corrected Minnaert resonance happens for $n=0$, then
\begin{align*}
\tilde{\lambda}_0+(\omega\epsilon c^{-1}_b)^2\tilde{\lambda}_{0,1}=0.
\end{align*}
Since $\tilde{\lambda}_0=0$, one has
\begin{align}\label{E:resnance}
\omega\epsilon c^{-1}_b=0
\end{align}
or
\begin{align}\label{E:resnance1}
-\frac{4\mu c(\omega)}{9(\lambda+2\mu)}-\frac{\delta\tau^2}{3(\lambda+2\mu)}=0.
\end{align}
\end{rema}

In what follows, denote by the lower-case character $\omega$ and the upper-case character $\Omega$ real frequencies and  complex frequencies, respectively.

In next proposition we obtain the corresponding complex frequencies if the  first-order corrected Minnaert resonances happen in three dimensions.

\begin{prop}
If $c(\Omega)$ has the form as \eqref{c:omega}, then equation \eqref{E:resnance} has the solution 
\begin{align}
&\Omega_{00}=0,\label{solution1}
\end{align}
and equation \eqref{E:resnance1} admits the solution
\begin{align}\label{solution2}
\Omega_{0}=\Omega'_{0}+\mathrm{i}\Omega''_{0},\quad\text{with}\quad \Omega'_{0}\in\mathbb{R},\Omega''_{0}\in\mathbb{R},
\end{align}
where
\begin{align}\label{E:Omega0}
&\Omega'_{0}=0,\quad \Omega''_{0}=-\frac{4\mu\gamma}{4\mu+3\delta\tau^2}<0.
\end{align}
\end{prop}
\begin{proof}
Formulas \eqref{solution1}--\eqref{E:Omega0}  follow from \eqref{E:resnance}, \eqref{E:resnance1} and \eqref{c:omega}.
\end{proof}
\begin{rema}
Thanks to the expression \eqref{E:SCA}, we have that $\tilde{\boldsymbol x}\in\mathbb{R}^3\backslash \overline{D}$,
\begin{align*}
\breve{\mathbf{u}}^{\mathrm{sca}}(\tilde{\boldsymbol x})
=&\frac{1}{\epsilon}\frac{1}{\lambda_0(k)}\langle \breve{\mathrm{F}},\breve{Y^0_0}\rangle_{L^2(\partial D)}\mathbf{S}^{\tilde{k}\tau}_{\partial D}[\breve{Y}^{0}_{0}\boldsymbol\nu](\tilde{\boldsymbol x})+\frac{1}{\epsilon}\sum_{n=1}^{N}\sum_{m=-n}^{n}\frac{1}{\lambda_{n}(k)}\langle\breve{\mathrm{F}},\breve{Y}^{m}_{n}\rangle_{L^{2}(\partial D)}\mathbf{S}^{\tilde{k}\tau}_{\partial D}[\breve{Y}^{m}_{n}\boldsymbol\nu](\tilde{\boldsymbol x}).
\end{align*}
Notice that
\begin{align*}
&\lambda^{-1}_0(k)\langle \breve{\mathrm{F}},\breve{Y^0_0}\rangle_{L^2(\partial D)}\\
=&\frac{\epsilon \langle \mathrm{F},Y^0_0\rangle_{L^2(\partial B)}}{k^2(-\frac{4\mu c(\omega)}{9(\lambda+2\mu)}-\frac{\delta\tau^2}{3(\lambda+2\mu)})+\mathcal{O}(k^3)}\\
{=}&\frac{\epsilon \langle -{\delta\tau^2} k^2\boldsymbol\nu\cdot \mathbf{u}^{\mathrm{in}}-\big(-\frac{1}{2}\mathcal{I}+\mathscr K^{k_1,*}_{\partial B}\big)(\mathscr S^{k_1}_{\partial B})^{-1}[\boldsymbol\nu\cdot\frac{\partial}{\partial\boldsymbol\nu}\mathbf{u}^{\mathrm{in}}],Y^0_0\rangle_{L^2(\partial B)}}{k^2(-\frac{4\mu c(\omega)}{9(\lambda+2\mu)}-\frac{\delta\tau^2}{3(\lambda+2\mu)})+\mathcal{O}(k^3)}\\
{=}&\frac{-\epsilon\delta\tau^2 k^2 \langle\boldsymbol\nu\cdot \mathbf{u}^{\mathrm{in}},Y^0_0\rangle_{L^2(\partial B)}-\epsilon\langle\big(-\frac{1}{2}\mathcal{I}+\mathscr K^{k_1,*}_{\partial B}\big)(\mathscr S^{k_1}_{\partial B})^{-1}[\boldsymbol\nu\cdot\frac{\partial}{\partial\boldsymbol\nu}\mathbf{u}^{\mathrm{in}}],Y^0_0\rangle_{L^2(\partial B)}}{k^2(-\frac{4\mu c(\omega)}{9(\lambda+2\mu)}-\frac{\delta\tau^2}{3(\lambda+2\mu)})+\mathcal{O}(k^3)}\\
{=}&\frac{-\epsilon\delta\tau^2 k^2 \langle\boldsymbol\nu\cdot \mathbf{u}^{\mathrm{in}},Y^0_0\rangle_{L^2(\partial B)}}{k^2(-\frac{4\mu c(\omega)}{9(\lambda+2\mu)}-\frac{\delta\tau^2}{3(\lambda+2\mu)})+\mathcal{O}(k^3)}+\frac{\epsilon k^2(\frac{1}{3}c(\omega)+\frac{17}{90}k^2(c(\omega))^2+\mathcal{O}(k^3))\langle\boldsymbol\nu\cdot\frac{\partial}{\partial\boldsymbol\nu}\mathbf{u}^{\mathrm{in}},Y^0_0\rangle_{L^2(\partial B)}}{k^2(-\frac{4\mu c(\omega)}{9(\lambda+2\mu)}-\frac{\delta\tau^2}{3(\lambda+2\mu)})+\mathcal{O}(k^3)}
\end{align*}
by means of \eqref{F:incident}.

As a result, $\Omega_{00}$ is a removable singularity and  $\Omega_{0}$ is a simple pole.
\end{rema}
\begin{defi}
Define the resonance radius as
\begin{align*}
\mathscr{R}:=\max\left\{|\Omega'_{0}|,|\Omega''_{0}|\right\}.
\end{align*}
\end{defi}
\begin{rema}
In this paper, the resonance radius is taken as
\begin{align*}
\mathscr{R}=\frac{4\mu\gamma}{4\mu+3\delta\tau^2}+1.
\end{align*}
\end{rema}


\section{Time domain approximation of the scattered field}\label{sec:5}

This section aims at establishing  a resonance expansion for the low-frequency part of the acoustic-elastic wave scattering from a bubble-elastic structure  in the time domain.

\subsection{Resonance expansion of the scattered field}

Recall that  $f:\omega\rightarrow f(\omega)$ is the Fourier transform of $\hat{f}$. Let  $\rho=\mathscr{R}>0$.  We suppose that most of the energy of the excitation is concentrated in the low frequencies, that is, for $0<\gamma_1\ll1$,
\begin{align*}
\int_{\mathbb{R}\setminus{ [-\rho,\rho]}}\omega^4|f(\omega)|^2\mathrm{d}{\omega}\leq{\gamma_1}.
\end{align*}
Moreover, for $0<\gamma_2\ll1$,
\begin{align*}
\rho\epsilon c^{-1}_b\leq \gamma_2.
\end{align*}

Because of \eqref{Fundamental2}, \eqref{c:omega} and \eqref{E:incident}, the incident field in time domain is given by
\begin{align*}
\hat{\breve{\mathbf{u}}}^{\mathrm{in}}(\tilde{\boldsymbol x},t)=&\int_{\mathbb{R}}-\omega^2f(\omega)\boldsymbol\Gamma^{\tilde{k}\tau}(\tilde{\boldsymbol x},\mathbf{s})\mathbf{p}e^{-\mathrm{i}\omega t}\mathrm{d}\omega\\
=&\int_{\mathbb{R}}\omega^2 f(\omega)\frac{e^{-\mathrm{i}\omega (t-\frac{\tau}{c_bc_s}|\tilde{\boldsymbol x}-\mathbf{s}|)}}{4\pi\mu|\tilde{\boldsymbol x}-\mathbf{s}|}\mathrm{d}\omega\mathbf{p}
+\nabla\nabla\big(\int_{\mathbb{R}}f(\omega)\frac{e^{-{\mathrm{i}\omega(t-\frac{\tau}{c_bc_s}|\tilde{\boldsymbol x}-\mathbf{s}|)}}}{4\pi c^{-2}_b\tau^2|\tilde{\boldsymbol x}-\mathbf{s}|}\mathrm{d}\omega\big)\mathbf{p}\\
&-\nabla\nabla\big(\int_{\mathbb{R}}f(\omega)\frac{e^{-\mathrm{i}\omega(t-\frac{\tau}{c_bc_p}|\tilde{\boldsymbol x}-\mathbf{s}|)}}{4\pi c^{-2}_b\tau^2|\tilde{\boldsymbol x}-\mathbf{s}|}\mathrm{d}\omega\big)\mathbf{p}\\
=&-\frac{\hat{f}''(t-\frac{\tau}{c_bc_s}|\tilde{\boldsymbol x}-\mathbf{s}|)}{4\pi\mu|\tilde{\boldsymbol x}-\mathbf{s}|}\mathbf{p}+\nabla\nabla\frac{\hat{f}(t-\frac{\tau}{c_bc_s}|\tilde{\boldsymbol x}-\mathbf{s}|)}{4\pi c^{-2}_b\tau^2|\tilde{\boldsymbol x}-\mathbf{s}|}\mathbf{p}-\nabla\nabla\frac{\hat{f}(t-\frac{\tau}{c_bc_p}|\tilde{\boldsymbol x}-\mathbf{s}|)}{4\pi c^{-2}_b\tau^2|\tilde{\boldsymbol x}-\mathbf{s}|}\mathbf{p},
\end{align*}
where $\boldsymbol\Gamma^{\tilde{k}\tau}=(\Gamma^{\tilde{k}\tau}_{ij})^{3}_{i,j=1}$ is given by \eqref{Fundamental2}.
Due to \eqref{E:speed}, we deduce
\begin{align}\label{relation}
c_p=\sqrt{(\lambda+\mu)+\mu}=\sqrt{\frac{1}{3}\cdot3(\lambda+\mu)+\mu}\stackrel{\eqref{mu}}{\geq}\sqrt{\frac{1}{3}\mu+\mu}\geq c_s.
\end{align}
Recall $\hat{f}:t\mapsto \hat{f}(t)\in C^{\infty}_0([0,C_1])$. Observe that $\hat{f}$ and $\hat{f}''$ vanish for negative arguments. From the position of physics,  the direct signal has not reached the observation point yet if $t<\frac{\tau}{c_bc_p}|\tilde{\boldsymbol x}-\mathbf{s}|$.

In addition, we define
$\boldsymbol{\mathcal{D}}^{\tilde{k}\tau}(\boldsymbol z,\mathbf{s})=(\mathcal{D}^{\tilde{k}\tau}_{ij})^3_{i,j=1}$ as a $3\times3$ matrix with entries:
\begin{align*}
\mathcal{D}^{\tilde{k}\tau}_{ij}=\partial_{j}\Gamma^{\tilde{k}\tau}_{i1}(\boldsymbol z,\mathbf{s})p_1+\partial_{j}\Gamma^{\tilde{k}\tau}_{i2}(\boldsymbol z,\mathbf{s})p_2+\partial_{j}\Gamma^{\tilde{k}\tau}_{i3}(\boldsymbol z,\mathbf{s})p_3,
\end{align*}
where  $p_j$ is the $j$-th component of $\mathbf{p}$.

In the next lemma we establish  the following asymptotic expansion for the incident field  $\mathrm{F}$ given by \eqref{F:incident}.
\begin{lemm}
One has that for $0<\epsilon\leq\gamma_2/\rho$,
\begin{align*}
\mathrm{F}=&\delta\tau^2k^2\omega^2f(\omega)\boldsymbol\nu\cdot(\boldsymbol{\Gamma}^{\tilde{k}\tau}(\boldsymbol z,\mathbf{s})\mathbf{p})+\epsilon\delta\tau^2k^2\omega^2f(\omega)\boldsymbol\nu\cdot(\sum_{j=1}^3x_j \partial_{j}\boldsymbol\Gamma^{\tilde{k}\tau}(\boldsymbol z,\mathbf{s})\mathbf{p})\\
&+\epsilon\lambda\omega^2f(\omega)\big(-\frac{1}{2}\mathcal{I}+\mathscr{K}^{k_1,*}_{\partial B}\big)(\mathscr{S}^{k_1}_{\partial B})^{-1}[\sum^3_{i=1}\sum^3_{j=1}\partial_{{i}}\Gamma^{\tilde{k}\tau}_{ij}(\boldsymbol z,\mathbf{s})p_{j}]\\
&+\epsilon\mu\omega^2f(\omega)\big(-\frac{1}{2}\mathcal{I}+\mathscr{K}^{k_1,*}_{\partial B}\big)(\mathscr{S}^{k_1}_{\partial B})^{-1}[\boldsymbol\nu\cdot\big(\big( \boldsymbol{\mathcal{D}}^{\tilde{k}\tau}(\boldsymbol z,\mathbf{s})+\boldsymbol{\mathcal{D}}^{\tilde{k}\tau}(\boldsymbol z,\mathbf{s})^\top\big)\boldsymbol\nu\big)]+\mathcal{O}(\epsilon^2).
\end{align*}
\end{lemm}
\begin{proof}
Because of  \eqref{E:incident}, \eqref{E:KKK} and \eqref{F:incident}, we have
\begin{align*}
\mathrm{F}=\omega^2f(\omega)\delta\tau^2k^2\boldsymbol\nu\cdot(\boldsymbol{\Gamma}^{\tilde{k}\tau}(\boldsymbol z+\epsilon \boldsymbol x,\mathbf{s})\mathbf{p})+\omega^2f(\omega)\big(-\frac{1}{2}\mathcal{I}+\mathscr{K}^{k_1,*}_{\partial B}\big)(\mathscr{S}^{k_1}_{\partial B})^{-1}\boldsymbol\nu\cdot(\frac{\partial}{\partial\boldsymbol \nu}\boldsymbol{\Gamma}^{\tilde{k}\tau}(\boldsymbol z+\epsilon \boldsymbol x,\mathbf{s})\mathbf{p}).
\end{align*}
It follows from Taylor's expansion that
\begin{align*}
\boldsymbol\Gamma^{\tilde{k}\tau}(\boldsymbol z+\epsilon \boldsymbol x,\mathbf{s})\mathbf{p}\mid_{\boldsymbol x\in\partial B}&=\boldsymbol\Gamma^{\tilde{k}\tau}(\boldsymbol z,\mathbf{s})\mathbf{p}+\epsilon \sum_{j=1}^3x_j \partial_{j}\boldsymbol\Gamma^{\tilde{k}\tau}(\boldsymbol z,\mathbf{s})\mathbf{p}+\mathcal{O}(\epsilon^2),\\
\frac{\partial}{\partial\boldsymbol\nu_{\boldsymbol x}}\boldsymbol\Gamma^{\tilde{k}\tau}(\boldsymbol z+\epsilon \boldsymbol x,\mathbf{s})\mathbf{p}
&=\epsilon\lambda\sum^3_{i=1}\sum^3_{j=1}\partial_{{i}}\Gamma^{\tilde{k}\tau}_{ij}(\boldsymbol z,\mathbf{s})p_{j}\boldsymbol\nu_{\boldsymbol x}+\epsilon\mu\big( \boldsymbol{\mathcal{D}}^{\tilde{k}\tau}(\boldsymbol z,\mathbf{s})+\boldsymbol{\mathcal{D}}^{\tilde{k}\tau}(\boldsymbol z,\mathbf{s})^\top\big)\boldsymbol\nu_{\boldsymbol x}+\mathcal{O}(\epsilon^2).
\end{align*}
Hence we get the conclusion of the lemma.

This ends the proof.
\end{proof}

Moreover, if $g\in L^2(\partial B)$, then for $0\leq n\leq N,|m|\leq n$,
\begin{align*}
\langle \big(-\frac{1}{2}\mathcal{I}+\mathscr K^{k_1,*}_{\partial B}\big)(\mathscr S^{k_1}_{\partial B})^{-1}[g],Y^m_n\rangle_{L^2(\partial B)}Y^m_n&=(-\frac{1}{2}+\zeta_{n}(k_1))\xi^{-1}_n(k_1)\langle g,Y^m_n\rangle_{L^2(\partial B)}Y^m_n\\
&=(n-\frac{k^2c(\omega)}{2n+3}+\mathcal{O}(k^3))\langle g,Y^m_n\rangle_{L^2(\partial B)}Y^m_n.
\end{align*}
Under the assumption \eqref{hy} on $\mathrm{F}$, we consider that for $0<\epsilon\leq\gamma_2/\rho$,
\begin{align*}
\mathrm{F}
=&\delta\tau^2k^2\omega^2f(\omega)\sum_{n=0}^{N}\sum_{m=-n}^{n}\langle\boldsymbol\nu\cdot(\boldsymbol\Gamma^{\tilde{k}\tau}(\boldsymbol z,\mathbf{s})\mathbf{p}),Y^m_n\rangle_{L^2(\partial B)}Y^m_n\\
&+\epsilon\delta\tau^2k^2\omega^2f(\omega)\sum_{n=0}^{N}\sum_{m=-n}^{n}\langle\boldsymbol\nu\cdot(\sum^{3}_{j=1} x_j \partial_{j}\boldsymbol\Gamma^{\tilde{k}\tau}(\boldsymbol z,\mathbf{s})\mathbf{p}),Y^m_n\rangle_{L^2(\partial B)}Y^m_n\\
&+\epsilon \lambda\omega^2 f(\omega)\sum_{n=0}^{N}\sum_{m=-n}^{n}(n-\frac{k^2c(\omega)}{2n+3})\langle\sum^3_{i=1}\sum^3_{j=1}\partial_{{i}}\Gamma^{\tilde{k}\tau}_{ij}(\boldsymbol z,\mathbf{s})p_{j},Y^m_n\rangle_{L^2(\partial B)}Y^m_n\\
&+\epsilon \mu \omega^2f(\omega)\sum_{n=0}^{N}\sum_{m=-n}^{n}(n-\frac{k^2c(\omega)}{2n+3})\langle\boldsymbol\nu\cdot\big(\big( \boldsymbol{\mathcal{D}}^{\tilde{k}\tau}(\boldsymbol z,\mathbf{s})+\boldsymbol{\mathcal{D}}^{\tilde{k}\tau}(\boldsymbol z,\mathbf{s})^\top\big)\boldsymbol\nu\big),Y^m_n\rangle_{L^2(\partial B)}Y^m_n,
\end{align*}
namely,
\begin{align}\label{F:truncation}
\breve{\mathrm{F}}=&\delta\tau^2k^2\omega^2f(\omega)\sum_{n=0}^{N}\sum_{m=-n}^{n}\langle\boldsymbol\nu\cdot(\boldsymbol\Gamma^{\tilde{k}\tau}(\boldsymbol z,\mathbf{s})\mathbf{p}),\breve{Y}^m_n\rangle_{L^2(\partial D)}\breve{Y}^m_n\nonumber\\
&+\epsilon\delta\tau^2k^2\omega^2f(\omega)\sum_{n=0}^{N}\sum_{m=-n}^{n}\langle\boldsymbol\nu\cdot( \sum^3_{j=1}\frac{\tilde{x}_j-z_j}{\epsilon}\partial_{j}\boldsymbol\Gamma^{\tilde{k}\tau}(\boldsymbol z,\mathbf{s})\mathbf{p}),\breve{Y}^m_n\rangle_{L^2(\partial D)}\breve{Y}^m_n\nonumber\\
&+\epsilon \lambda \omega^2f(\omega)\sum_{n=0}^{N}\sum_{m=-n}^{n}(n-\frac{k^2c(\omega)}{2n+3})\langle\sum^3_{i=1}\sum^3_{j=1}\partial_{{i}}\Gamma^{\tilde{k}\tau}_{ij}(\boldsymbol z,\mathbf{s})p_{j},\breve{Y}^m_n\rangle_{L^2(\partial D)}\breve{Y}^m_n\nonumber\\
&+\epsilon \mu\omega^2 f(\omega)\sum_{n=0}^{N}\sum_{m=-n}^{n}(n-\frac{k^2c(\omega)}{2n+3})\langle\boldsymbol\nu\cdot\big(\big( \boldsymbol{\mathcal{D}}^{\tilde{k}\tau}(\boldsymbol z,\mathbf{s})+\boldsymbol{\mathcal{D}}^{\tilde{k}\tau}(\boldsymbol z,\mathbf{s})^\top\big)\boldsymbol\nu\big),\breve{Y}^m_n\rangle_{L^2(\partial D)}\breve{Y}^m_n,
\end{align}
where $\breve{\mathrm{F}}$ is as seen in \eqref{D:F}.

Recall that $\boldsymbol z$ is the center of the resonator, $\epsilon $ is its radius and $\mathbf{s}$ is the source location. Let us define
\begin{align*}
t^{-}_0:=&\frac{c^{-1}_b\tau|\boldsymbol z-\mathbf{s}|}{c_p}+\frac{c^{-1}_b\tau|\tilde{\boldsymbol x}-\boldsymbol z|}{c_p}-\frac{c^{-1}_b\tau\epsilon}{c_s}
-C_1,\\
t^{+}_0:=&\frac{c^{-1}_b\tau|\boldsymbol z-\mathbf{s}|}{c_s}+\frac{c^{-1}_b\tau|\tilde{\boldsymbol x}-\boldsymbol z|}{c_s}+\frac{c^{-1}_b\tau\epsilon}{c_s}
+C_1,
\end{align*}
where $\frac{c^{-1}_b\tau|\boldsymbol z-\mathbf{s}|+c^{-1}_b\tau|\tilde{\boldsymbol x}-\boldsymbol z|}{c_i},i=s,p$ signify the time it takes the wide-band signal to reach first the scatterer and then to the observation point $\tilde{\boldsymbol x}$ and $-\frac{c^{-1}_b\tau\epsilon}{c_s},\frac{c^{-1}_b\tau\epsilon}{c_s}$ stand for the maximal timespan spent inside the bubble. In addition, the truncated inverse Fourier transform of the scattered field $\breve{\mathbf{u}}^{\mathrm{sca}}$ is defined  by
\begin{align*}
\mathbf{P}_{\rho}[\breve{\mathbf{u}}^{\mathrm{sca}}](\tilde{\boldsymbol x},t):=\int^{\rho}_{-\rho}\breve{\mathbf{u}}^{\mathrm{sca}}(\tilde{\boldsymbol x},\omega)e^{-\mathrm{i}\omega t}\mathrm{d}\omega.
\end{align*}

The following theorem corresponds to the expression of the truncated scattered field in the time domain.

\begin{theo}\label{th:scattered}
Suppose that the incident wave $\breve{\mathrm{F}}$ has the form as \eqref{F:truncation}. Let $\Omega''_{0}$ be as seen in \eqref{E:Omega0}. For $0<\epsilon\leq\gamma_2/\rho$, there exists an integer $M\geq1$ such that if we set
\begin{align*}
\mathbf{e}_{0}(\tilde{\boldsymbol x}):=\mathbf{S}^{\mathrm{i}\Omega''_0 c^{-1}_b\tau}_{\partial D}[\breve{Y}^{0}_{0}\boldsymbol\nu](\tilde{\boldsymbol x}),
\end{align*}
then
the truncated scattered field has the following form in the time domain:

{\rm(i)} One has that for $\tilde{\boldsymbol x}\in\mathbb{R}^3\backslash \overline{D},t\leq t^{-}_0$,
\begin{align}\label{E:scattered}
\mathbf{P}_{\rho}[\breve{\mathbf{u}}^{\mathrm{sca}}](\tilde{\boldsymbol x},t)=
&\mathcal{O}\left(\epsilon ^3\rho^{-M}\right).
\end{align}

{\rm(ii)} One has that for $\tilde{\boldsymbol x}\in\mathbb{R}^3\backslash \overline{D},t\geq t^{+}_0$,
\begin{align}\label{EE:scattered}
\mathbf{P}_{\rho}[\breve{\mathbf{u}}^{\mathrm{sca}}](\tilde{\boldsymbol x},t)
=&\frac{1152\pi(\mu\gamma)^3(\lambda+2\mu)}{\epsilon(4\mu+3\delta\tau^2)^4}\Big(\delta\tau^2f(\mathrm{i}\Omega''_0)\langle\boldsymbol\nu\cdot(\boldsymbol\Gamma^{\mathrm{i}\Omega''_0 c^{-1}_b\tau}(\boldsymbol z,\mathbf{s})\mathbf{p}),\breve{Y}^0_0\rangle_{L^2(\partial D)}\nonumber\\
&+\epsilon\delta\tau^2f(\mathrm{i}\Omega''_0)\langle\boldsymbol\nu\cdot( \sum^3_{j=1}\frac{\tilde{x}_j-z_j}{\epsilon} \partial_{j}\boldsymbol\Gamma^{\mathrm{i}\Omega''_0 c^{-1}_b\tau}(\boldsymbol z,\mathbf{s})\mathbf{p}),\breve{Y}^0_0\rangle_{L^2(\partial D)}\nonumber\\
&-\frac{1}{3}\epsilon \lambda  f(\mathrm{i}\Omega''_0)c(\mathrm{i}\Omega''_0)\langle\sum^3_{i=1}\sum^3_{j=1}\partial_{{i}}\Gamma^{\mathrm{i}\Omega''_0 c^{-1}_b\tau}_{ij}(\boldsymbol z,\mathbf{s})p_{j},\breve{Y}^0_0\rangle_{L^2(\partial D)}\nonumber\\
&-\frac{1}{3}\epsilon \mu   f(\mathrm{i}\Omega''_0)c(\mathrm{i}\Omega''_0)\langle\boldsymbol\nu\cdot\big(\big( \boldsymbol{\mathcal{D}}^{\mathrm{i}\Omega''_0 c^{-1}_b\tau}(\boldsymbol z,\mathbf{s})+\boldsymbol{\mathcal{D}}^{\mathrm{i}\Omega''_0 c^{-1}_b\tau}(\boldsymbol z,\mathbf{s})^\top\big)\boldsymbol\nu\big),\breve{Y}^0_0\rangle_{L^2(\partial D)}\Big)\nonumber\\
&\times\mathbf{e}_{0}(\tilde{\boldsymbol x})e^{\Omega''_{0} t}+\mathcal{O}\left({\epsilon^3\rho^{-M}}/{t}\right).
\end{align}
\end{theo}
\begin{rema}
As shown  in  formula \eqref{EE:scattered}, the $0$-th mode  is only needed to reconstruct the information of the truncated scattered field of the bubble-elastic structure in the time domain.  The proof of Theorem \ref{th:scattered} will be given in the sequel.
\end{rema}
\begin{rema}
Theorem \ref{th:scattered} only presents an approximation for  the low-frequency part of the scattered field in time domain. Nevertheless, as shown in the numerical section 6.4.5 introduced by \cite{Baldassari2021Modal}, the low-frequency part of the scattered field is actually a good  approximation for the scattered field. Up to now, there is no mathematical justification for that.
\end{rema}

Thanks to Lemma \ref{le:gamma}, it is clear to derive that
\begin{align*}
\boldsymbol\Gamma^{\tilde{k}\tau}(\tilde{\boldsymbol x},\tilde{\boldsymbol y}) =&-e^{\mathrm{i}\frac{\tilde{k}\tau}{c_s}|\tilde{\boldsymbol x}-\tilde{\boldsymbol y}|}\frac{\boldsymbol{\mathcal{A}}(\tilde{\boldsymbol x},\tilde{\boldsymbol y},\frac{\tilde{k}\tau}{c_s})}{4\pi(\tilde{k}\tau)^2|\tilde{\boldsymbol x}-\tilde{\boldsymbol y}|}
-e^{\mathrm{i}\frac{\tilde{k}\tau}{c_p}|\tilde{\boldsymbol x}-\tilde{\boldsymbol y}|}\frac{\boldsymbol{\mathcal{B}}(\tilde{\boldsymbol x},\tilde{\boldsymbol y},\frac{\tilde{k}\tau}{c_p})}{4\pi(\tilde{k}\tau)^2|\tilde{\boldsymbol x}-\tilde{\boldsymbol y}|}.
\end{align*}
Therefore,
\begin{align*}
\mathbf{e}_{0}(\tilde{\boldsymbol x})=&\int_{\partial D}e^{\frac{-\Omega''_{0}c^{-1}_b\tau}{c_s}|\tilde{\boldsymbol x}-\tilde{\boldsymbol y}|}\frac{\boldsymbol{\mathcal{A}}(\tilde{\boldsymbol x},\tilde{\boldsymbol y},\frac{\mathrm{i}\Omega''_{0}c^{-1}_b\tau}{c_s})}{4\pi(\Omega''_{0}c^{-1}_b\tau)^2|\tilde{\boldsymbol x}-\tilde{\boldsymbol y}|}\breve{Y}^{0}_{0}(\tilde{\boldsymbol y})\boldsymbol\nu_{\tilde{\boldsymbol y}}\mathrm{d}\sigma(\tilde{\boldsymbol y})\\
&+\int_{\partial D}e^{\frac{-\Omega''_{0}c^{-1}_b\tau}{c_p}|\tilde{\boldsymbol x}-\tilde{\boldsymbol y}|}\frac{\boldsymbol{\mathcal{B}}(\tilde{\boldsymbol x},\tilde{\boldsymbol y},\frac{\mathrm{i}\Omega''_{0}c^{-1}_b\tau}{c_p})}{4\pi(\Omega''_{0}c^{-1}_b\tau)^2|\tilde{\boldsymbol x}-\tilde{\boldsymbol y}|}\breve{Y}^{0}_{0}(\tilde{\boldsymbol y})\boldsymbol\nu_{\tilde{\boldsymbol y}}\mathrm{d}\sigma(\tilde{\boldsymbol y}).
\end{align*}
By  the fact $\Omega''_{0}<0$,  it is straightforward to see that $\mathbf{e}_{0}(\tilde{\boldsymbol x})$ may tend to infinity as $|\tilde{\boldsymbol x}|$ tends to infinity. Even so, we will show that no terms  in \eqref{EE:scattered} diverge in next theorem.

\begin{theo}
Suppose that the incident wave $\breve{\mathrm{F}}$ has the form provided by \eqref{F:truncation}. Let $\Omega''_{0}$ be as seen in \eqref{E:Omega0}. If we define
\begin{align*}
\mathbf{E}_{0}(\tilde{\boldsymbol x}):=\mathbf{e}_{0}(\tilde{\boldsymbol x})e^{\frac{\Omega''_{0}c^{-1}_b\tau|\tilde{\boldsymbol x}-\boldsymbol z|}{c_s}},
\end{align*}
then there exists an integer $M\geq1$ such that for $0<\epsilon\leq\gamma_2/\rho$,
the truncated scattered field has the following form in the time domain:

{\rm(i)} One has that for $\tilde{\boldsymbol x}\in\mathbb{R}^3\backslash \overline{D},t\leq t^{-}_0$,
\begin{align*}
\mathbf{P}_{\rho}[\breve{\mathbf{u}}^{\mathrm{sca}}](\tilde{\boldsymbol x},t)=\mathcal{O}\left(\epsilon^3\rho^{-M}\right).
\end{align*}

{\rm(ii)} One has that for $\tilde{\boldsymbol x}\in\mathbb{R}^3\backslash \overline{D},t\geq t^{+}_0$,
\begin{align*}
\mathbf{P}_{\rho}[\breve{\mathbf{u}}^{\mathrm{sca}}](\tilde{\boldsymbol x},t)=&\mathscr{C}_{0}(\breve{\mathbf{u}}^{\mathbf{in}},\epsilon)\mathbf{E}_{0}(\tilde{\boldsymbol x})e^{\Omega''_{0} (t-t_0^+)}+\mathcal{O}\left({\epsilon^3\rho^{-M}}/{t}\right)
\end{align*}
with
\begin{align*}
\mathscr{C}_{0}(\breve{\mathbf{u}}^{\mathbf{in}},\epsilon)=&\frac{1152\pi(\mu\gamma)^3(\lambda+2\mu)C_{\epsilon}}{\epsilon(4\mu+3\delta\tau^2)^4}\Big(\delta\tau^2f(\mathrm{i}\Omega''_0)\langle\boldsymbol\nu\cdot(\boldsymbol\Gamma^{\mathrm{i}\Omega''_0 c^{-1}_b\tau}(\boldsymbol z,\mathbf{s})\mathbf{p}),\breve{Y}^0_0\rangle_{L^2(\partial D)}\nonumber\\
&+\epsilon\delta\tau^2f(\mathrm{i}\Omega''_0)\langle\boldsymbol\nu\cdot( \sum^3_{j=1}\frac{\tilde{x}_j-z_j}{\epsilon} \partial_{j}\boldsymbol\Gamma^{\mathrm{i}\Omega''_0 c^{-1}_b\tau}(\boldsymbol z,\mathbf{s})\mathbf{p}),\breve{Y}^0_0\rangle_{L^2(\partial D)}\nonumber\\
&-\frac{1}{3}\epsilon \lambda  f(\mathrm{i}\Omega''_0)c(\mathrm{i}\Omega''_0)\langle\sum^3_{i=1}\sum^3_{j=1}\partial_{{i}}\Gamma^{\mathrm{i}\Omega''_0 c^{-1}_b\tau}_{ij}(\boldsymbol z,\mathbf{s})p_{j},\breve{Y}^0_0\rangle_{L^2(\partial D)}\nonumber\\
&-\frac{1}{3}\epsilon \mu   f(\mathrm{i}\Omega''_0)c(\mathrm{i}\Omega''_0)\langle\boldsymbol\nu\cdot\big(\big( \boldsymbol{\mathcal{D}}^{\mathrm{i}\Omega''_0 c^{-1}_b\tau}(\boldsymbol z,\mathbf{s})+\boldsymbol{\mathcal{D}}^{\mathrm{i}\Omega''_0 c^{-1}_b\tau}(\boldsymbol z,\mathbf{s})^\top\big)\boldsymbol\nu\big),\breve{Y}^0_0\rangle_{L^2(\partial D)}\Big),
\end{align*}
where  $C_{\epsilon}=e^{\Omega''_{0} (\frac{c^{-1}_b\tau|\boldsymbol z-\mathbf{s}|}{c_s}+\frac{c^{-1}_b\tau\epsilon}{c_s}
+C_1)}$.
\end{theo}
\begin{proof}
It follows from \eqref{relation} that the term $\mathbf{E}_{0}(\tilde{\boldsymbol x})$ does not diverge. Moreover,
\begin{align*}
\mathbf{e}_{0}(\tilde{\boldsymbol x})e^{\Omega''_{0} t}&=\mathbf{e}_{0}(\tilde{\boldsymbol x})e^{\Omega''_{0} t^{+}_0}e^{\Omega''_{0} (t-t_0^+)}\\
&=\mathbf{e}_{0}(\tilde{\boldsymbol x})e^{\Omega''_{0} (\frac{c^{-1}_b\tau|\boldsymbol z-\mathbf{s}|}{c_s}+\frac{c^{-1}_b\tau|\tilde{\boldsymbol x}-\boldsymbol z|}{c_s}+\frac{c_b^{-1}\tau\epsilon}{c_s}
+C_1)}e^{\Omega''_{0} (t-t_0^+)}\\
&=C_{\epsilon}\mathbf{e}_{0}(\tilde{\boldsymbol x})e^{\frac{\Omega''_{0}c^{-1}_b\tau|\tilde{\boldsymbol x}-\boldsymbol z|}{c_s}}e^{\Omega''_{0} (t-t_0^+)}\\
&=C_{\epsilon}\mathbf{E}_{0}(\tilde{\boldsymbol x})e^{\Omega''_{0} (t-t_0^+)}.
\end{align*}
Combining this with Theorem \ref{th:scattered}, we can derive the conclusion of the theorem.

The proof is complete.
\end{proof}

\subsection{Proof of Theorem \ref{th:scattered}}
The rest of this section is to complete the proof of of Theorem \ref{th:scattered}.

\begin{proof}[Proof of Theorem \ref{th:scattered}]
Thanks to Proposition \ref{prop:approximation}, we recall the following spectral decomposition in the frequency domain
\begin{align*}
\breve{\mathbf{u}}^{\mathrm{sca}}(\tilde{\boldsymbol x},\omega)=\sum_{n=0}^{N}\sum_{m=-n}^{n}\boldsymbol{\Xi}^{m}_{n}(\tilde{\boldsymbol x},\omega),\quad \boldsymbol x\in\mathbb{R}^3\backslash \overline{D},
\end{align*}
where
\begin{align*}
\boldsymbol{\Xi}^{m}_{n}(\tilde{\boldsymbol x},\omega)=&\frac{1}{\epsilon}\frac{1}{\lambda_n(k)}\Big(\delta\tau^2k^2\omega^2f(\omega)\langle\boldsymbol\nu\cdot(\boldsymbol\Gamma^{\tilde{k}\tau}(\boldsymbol z,\mathbf{s})\mathbf{p}),\breve{Y}^m_n\rangle_{L^2(\partial D)}\nonumber\\
&+\epsilon\delta\tau^2k^2\omega^2f(\omega)\langle\boldsymbol\nu\cdot( \sum^3_{j=1}\frac{\tilde{x}_j-z_j}{\epsilon} \partial_{j}\boldsymbol\Gamma^{\tilde{k}\tau}(\boldsymbol z,\mathbf{s})\mathbf{p}),\breve{Y}^m_n\rangle_{L^2(\partial D)}\nonumber\\
&+\epsilon \lambda \omega^2f(\omega)(n-\frac{k^2c(\omega)}{2n+3})\langle\sum^3_{i=1}\sum^3_{j=1}\partial_{{i}}\Gamma^{\tilde{k}\tau}_{ij}(\boldsymbol z,\mathbf{s})p_{j},\breve{Y}^m_n\rangle_{L^2(\partial D)}\nonumber\\
&+\epsilon \mu \omega^2f(\omega)(n-\frac{k^2c(\omega)}{2n+3})\langle\boldsymbol\nu\cdot\big(\big( \boldsymbol{\mathcal{D}}^{\tilde{k}\tau}(\boldsymbol z,\mathbf{s})+\boldsymbol{\mathcal{D}}^{\tilde{k}\tau}(\boldsymbol z,\mathbf{s})^\top\big)\boldsymbol\nu\big),\breve{Y}^m_n\rangle_{L^2(\partial D)}\Big)\mathbf{S}^{\tilde{k}\tau }_{\partial D}[\breve{Y}^m_n\boldsymbol\nu](\tilde{\boldsymbol x}).
\end{align*}
It is straightforward to verify that
\begin{align*}
\mathbf{P}_{\rho}[\breve{\mathbf{u}}^{\mathrm{sca}}](\tilde{\boldsymbol x},t)=\int^{\rho}_{-\rho}\breve{\mathbf{u}}^{\mathrm{sca}}(\tilde{\boldsymbol x},\omega)e^{-\mathrm{i}\omega t}\mathrm{d}\omega
=&\sum_{n=0}^{N}\sum_{m=-n}^{n}\int^{\rho}_{-\rho}\boldsymbol{\Xi}^{m}_{n}(\tilde{\boldsymbol x},\omega)e^{-\mathrm{i}\omega t}\mathrm{d}\omega.
\end{align*}

Denote by the integration contour $\mathcal{C}^{\pm}_{\rho}$ a semicircular arc of radius $\rho$ in the upper $(+)$ or lower $(-)$ half-plane. Let $\mathcal{C}^{\pm}$ be the closed contour $\mathcal{C}^{\pm}_{\rho}\cup[-\rho,\rho]$.

Obviously, one obtains
\begin{align*}
\int^{\rho}_{-\rho}\boldsymbol{\Xi}^{m}_{n}(\tilde{\boldsymbol x},\omega)e^{-\mathrm{i}\omega t}\mathrm{d}\omega=\oint_{\mathcal{C}^\pm}\boldsymbol{\Xi}^{m}_{n}(\tilde{\boldsymbol x},\Omega)e^{-\mathrm{i}\Omega t}\mathrm{d}\Omega-\int_{\mathcal{C}^\pm_\rho}\boldsymbol{\Xi}^{m}_{n}(\tilde{\boldsymbol x},\Omega)e^{-\mathrm{i}\Omega t}\mathrm{d}\Omega.
\end{align*}
Since $\Omega_0=\Omega'_0+\mathrm{i}\Omega''_{0}$ with $\Omega''_{0}<0$, by the residue theorem, we have
\begin{align*}
\oint_{\mathcal{C}^+}\boldsymbol{\Xi}^{m}_{n}(\tilde{\boldsymbol x},\Omega)e^{-\mathrm{i}\Omega t}\mathrm{d}\Omega&=0,&&n\geq0,|m|\leq n,\\
\oint_{\mathcal{C}^-}\boldsymbol{\Xi}^{m}_{n}(\tilde{\boldsymbol x},\Omega)e^{-\mathrm{i}\Omega t}\mathrm{d}\Omega&=0,&&n\geq1,|m|\leq n,\\
\oint_{\mathcal{C}^-}\boldsymbol{\Xi}^{0}_{0}(\tilde{\boldsymbol x},\Omega)e^{-\mathrm{i}\Omega t}\mathrm{d}\Omega&=2\pi\mathrm{i}\mathrm{Res}(\boldsymbol{\Xi}^{0}_{0}(\tilde{\boldsymbol x},\Omega)e^{-\mathrm{i}\Omega t},\Omega_{0}),&&n=0,m=0.
\end{align*}
As $\Omega_{0}$ is a simple pole, we can derive that
\begin{align*}
&\oint_{\mathcal{C}^-}\boldsymbol{\Xi}^{0}_{0}(\tilde{\boldsymbol x},\Omega)e^{-\mathrm{i}\Omega t}\mathrm{d}\Omega\\
=&2\pi\mathrm{i}\mathrm{Res}(\boldsymbol{\Xi}^{0}_{0}(\tilde{\boldsymbol x},\Omega),\Omega_{0})e^{-\mathrm{i}\Omega_{0} t}\\
=&2\pi\mathrm{i}\Big(\frac{1}{\epsilon}\frac{1}{\lambda_0(\Omega\epsilon c^{-1}_b)}\Big(\delta\tau^2(\Omega\epsilon c^{-1}_b)^2\Omega^2f(\Omega)\langle\boldsymbol\nu\cdot(\boldsymbol\Gamma^{\Omega c^{-1}_b\tau}(\boldsymbol z,\mathbf{s})\mathbf{p}),\breve{Y}^0_0\rangle_{L^2(\partial D)}\nonumber\\
&+\epsilon\delta\tau^2(\Omega\epsilon c^{-1}_b)^2\Omega^2f(\Omega)\langle\boldsymbol\nu\cdot( \sum^3_{j=1}\frac{\tilde{x}_j-z_j}{\epsilon} \partial_{j}\boldsymbol\Gamma^{\Omega c^{-1}_b\tau}(\boldsymbol z,\mathbf{s})\mathbf{p}),\breve{Y}^0_0\rangle_{L^2(\partial D)}\nonumber\\
&-\frac{1}{3}\epsilon \lambda (\Omega\epsilon c^{-1}_b)^2\Omega^2 f(\Omega)c(\Omega)\langle\sum^3_{i=1}\sum^3_{j=1}\partial_{{i}}\Gamma^{\Omega c^{-1}_b\tau}_{ij}(\boldsymbol z,\mathbf{s})p_{j},\breve{Y}^0_0\rangle_{L^2(\partial D)}\nonumber\\
&-\frac{1}{3}\epsilon \mu  (\Omega\epsilon c^{-1}_b)^2\Omega^2 f(\Omega)c(\Omega)\langle\boldsymbol\nu\cdot\big(\big( \boldsymbol{\mathcal{D}}^{\Omega c^{-1}_b\tau}(\boldsymbol z,\mathbf{s})+\boldsymbol{\mathcal{D}}^{\Omega c^{-1}_b\tau}(\boldsymbol z,\mathbf{s})^\top\big)\boldsymbol\nu\big),\breve{Y}^0_0\rangle_{L^2(\partial D)}\Big)\nonumber\\
&\times\mathbf{S}^{\Omega c^{-1}_b\tau}_{\partial D}[\breve{Y}^0_0\boldsymbol\nu](\tilde{\boldsymbol x})(\Omega-\Omega_0)\Big)\Big|_{\Omega=\Omega_0}e^{-\mathrm{i}\Omega_{0} t}\\
=&\frac{1152\pi(\mu\gamma)^3(\lambda+2\mu)}{\epsilon(4\mu+3\delta\tau^2)^4}\Big(\delta\tau^2f(\mathrm{i}\Omega''_0)\langle\boldsymbol\nu\cdot(\boldsymbol\Gamma^{\mathrm{i}\Omega''_0 c^{-1}_b\tau}(\boldsymbol z,\mathbf{s})\mathbf{p}),\breve{Y}^0_0\rangle_{L^2(\partial D)}\nonumber\\
&+\epsilon\delta\tau^2f(\mathrm{i}\Omega''_0)\langle\boldsymbol\nu\cdot( \sum^3_{j=1}\frac{\tilde{x}_j-z_j}{\epsilon} \partial_{j}\boldsymbol\Gamma^{\mathrm{i}\Omega''_0 c^{-1}_b\tau}(\boldsymbol z,\mathbf{s})\mathbf{p}),\breve{Y}^0_0\rangle_{L^2(\partial D)}\nonumber\\
&-\frac{1}{3}\epsilon \lambda  f(\mathrm{i}\Omega''_0)c(\mathrm{i}\Omega''_0)\langle\sum^3_{i=1}\sum^3_{j=1}\partial_{{i}}\Gamma^{\mathrm{i}\Omega''_0 c^{-1}_b\tau}_{ij}(\boldsymbol z,\mathbf{s})p_{j},\breve{Y}^0_0\rangle_{L^2(\partial D)}\nonumber\\
&-\frac{1}{3}\epsilon \mu   f(\mathrm{i}\Omega''_0)c(\mathrm{i}\Omega''_0)\langle\boldsymbol\nu\cdot\big(\big( \boldsymbol{\mathcal{D}}^{\mathrm{i}\Omega''_0 c^{-1}_b\tau}(\boldsymbol z,\mathbf{s})+\boldsymbol{\mathcal{D}}^{\mathrm{i}\Omega''_0 c^{-1}_b\tau}(\boldsymbol z,\mathbf{s})^\top\big)\boldsymbol\nu\big),\breve{Y}^0_0\rangle_{L^2(\partial D)}\Big)\nonumber\\
&\times\mathbf{S}^{\mathrm{i}\Omega''_0 c^{-1}_b\tau}_{\partial D}[\breve{Y}^0_0\boldsymbol\nu](\tilde{\boldsymbol x})e^{\Omega''_{0} t}.
\end{align*}

It remains to estimate $\int_{\mathcal{C}^\pm_\rho}\boldsymbol{\Xi}^{m}_{n}(\tilde{\boldsymbol x},\Omega)e^{-\mathrm{i}\Omega t}\mathrm{d}\Omega$. From the expression of $\boldsymbol{\Xi}^{m}_{n}(\tilde{\boldsymbol x},\Omega)$, it follows that
\begin{align*}
&\int_{\mathcal{C}^{\pm}_{\rho}}\boldsymbol{\Xi}^{m}_{n}(\tilde{\boldsymbol x},\Omega)e^{-\mathrm{i}\Omega t}\mathrm{d}\Omega
\\
=&\int_{\mathcal{C}^\pm_\rho}\frac{1}{\epsilon}\frac{1}{\lambda_n(\Omega\epsilon c^{-1}_b)}\Big(\delta\tau^2(\Omega\epsilon c^{-1}_b)^2\Omega^2f(\Omega)\langle\boldsymbol\nu\cdot(\boldsymbol\Gamma^{\Omega c^{-1}_b\tau}(\boldsymbol z,\mathbf{s})\mathbf{p}),\breve{Y}^m_n\rangle_{L^2(\partial D)}\nonumber\\
&+\epsilon\delta\tau^2(\Omega\epsilon c^{-1}_b)^2\Omega^2f(\Omega)\langle\boldsymbol\nu\cdot( \sum^3_{j=1}\frac{\tilde{x}_j-z_j}{\epsilon} \partial_{j}\boldsymbol\Gamma^{\Omega c^{-1}_b\tau}(\boldsymbol z,\mathbf{s})\mathbf{p}),\breve{Y}^m_n\rangle_{L^2(\partial D)}\nonumber\\
&+\epsilon \lambda\Omega^2 f(\Omega)(n-\frac{(\Omega\epsilon c^{-1}_b)^2c(\Omega)}{2n+3})\langle\sum^3_{i=1}\sum^3_{j=1}\partial_{{i}}\Gamma^{\Omega c^{-1}_b\tau}_{ij}(\boldsymbol z,\mathbf{s})p_{j},\breve{Y}^m_n\rangle_{L^2(\partial D)}\nonumber\\
&+\epsilon \mu\Omega^2 f(\Omega)(n-\frac{(\Omega\epsilon c^{-1}_b)^2c(\Omega)}{2n+3})\langle\boldsymbol\nu\cdot\big(\big( \boldsymbol{\mathcal{D}}^{\Omega c^{-1}_b\tau}(\boldsymbol z,\mathbf{s})+\boldsymbol{\mathcal{D}}^{\Omega c^{-1}_b\tau}(\boldsymbol z,\mathbf{s})^\top\big)\boldsymbol\nu\big),\breve{Y}^m_n\rangle_{L^2(\partial D)}\Big)\\
&\times\mathbf{S}^{\Omega c^{-1}_b\tau }_{\partial D}[\breve{Y}^m_n\boldsymbol\nu](\tilde{\boldsymbol x})e^{-\mathrm{i}\Omega t}\mathrm{d}\Omega.
\end{align*}
It follows from  Lemma \ref{le:gamma} that
\begin{align*}
\boldsymbol\Gamma^{\Omega c^{-1}_b\tau}(\boldsymbol z,\mathbf{s})\mathbf{p}=&-e^{\mathrm{i}\frac{\Omega c^{-1}_b\tau}{c_s}|\boldsymbol z-\mathbf{s}|}\frac{\boldsymbol{\mathcal{A}}(\boldsymbol z,\mathbf{s},\frac{\Omega c^{-1}_b\tau}{c_s})\mathbf{p}}{4\pi(\Omega c^{-1}_b\tau)^2|\boldsymbol z-\mathbf{s}|}
-e^{\mathrm{i}\frac{\Omega c^{-1}_b\tau}{c_p}|\boldsymbol z-\mathbf{s}|}\frac{\boldsymbol{\mathcal{B}}(\boldsymbol z,\mathbf{s},\frac{\Omega c^{-1}_b\tau}{c_p})\mathbf{p}}{4\pi(\Omega c^{-1}_b\tau)^2|\boldsymbol z-\mathbf{s}|},\\
\sum^3_{j=1}\frac{\tilde{x}_j-z_j}{\epsilon} \partial_{j}\boldsymbol\Gamma^{\Omega c^{-1}_b\tau}(\boldsymbol z,\mathbf{s})\mathbf{p}=&-e^{\mathrm{i}\frac{\Omega c^{-1}_b\tau}{c_s}|\boldsymbol z-\mathbf{s}|}\frac{\sum^3_{j=1}\frac{\tilde{x}_j-z_j}{\epsilon} \boldsymbol{\mathscr{A}}^{(j)}(\boldsymbol z,\mathbf{s},\frac{\Omega c^{-1}_b\tau}{c_s})\mathbf{p}}{4\pi(\Omega c^{-1}_b\tau)^2|\boldsymbol z-\mathbf{s}|}\\
&-e^{\mathrm{i}\frac{\Omega c^{-1}_b\tau}{c_p}|\boldsymbol z-\mathbf{s}|}\frac{\sum^3_{j=1}\frac{\tilde{x}_j-z_j}{\epsilon} \boldsymbol{\mathscr{B}}^{(j)}(\boldsymbol z,\mathbf{s},\frac{\Omega c^{-1}_b\tau}{c_p})\mathbf{p}}{4\pi(\Omega c^{-1}_b\tau)^2|\boldsymbol z-\mathbf{s}|}.
\end{align*}
In addition,
\begin{align*}
\sum^3_{i,j=1}\partial_{{i}}\Gamma^{\Omega c^{-1}_b\tau}_{ij}(\boldsymbol z,\mathbf{s})p_{j}=&-e^{\mathrm{i}\frac{\Omega c^{-1}_b\tau}{c_s}|\boldsymbol z-\mathbf{s}|}
\frac{\mathscr{P}(\boldsymbol z,\mathbf{s},\frac{\Omega c^{-1}_b\tau}{c_s})}{4\pi(\Omega c^{-1}_b\tau)^2|\boldsymbol z-\mathbf{s}|}-
e^{\mathrm{i}\frac{\Omega c^{-1}_b\tau}{c_p}|\boldsymbol z-\mathbf{s}|}\frac{\mathscr{Q}(\boldsymbol z,\mathbf{s},\frac{\Omega c^{-1}_b\tau}{c_p})}{{4\pi(\Omega c^{-1}_b\tau)^2|\boldsymbol z-\mathbf{s}|}},\\
\big( \boldsymbol{\mathcal{D}}^{\Omega c^{-1}_b\tau}(\boldsymbol z,\mathbf{s})+\boldsymbol{\mathcal{D}}^{\Omega c^{-1}_b\tau}(\boldsymbol z,\mathbf{s})^\top\big)\boldsymbol\nu =&-e^{\mathrm{i}\frac{\Omega c^{-1}_b\tau}{c_s}|\boldsymbol z-\mathbf{s}|}\frac{\boldsymbol{\mathscr{M}}(\boldsymbol z,\mathbf{s},\frac{\Omega c^{-1}_b\tau}{c_s})\boldsymbol\nu}{4\pi(\Omega c^{-1}_b\tau)^2|\boldsymbol z-\mathbf{s}|}-e^{\mathrm{i}\frac{\Omega c^{-1}_b\tau}{c_p}|\boldsymbol z-\mathbf{s}|}\frac{\boldsymbol{\mathscr{N}}(\boldsymbol z,\mathbf{s},\frac{\Omega c^{-1}_b\tau}{c_p})\boldsymbol\nu}{4\pi(\Omega c^{-1}_b\tau)^2|\boldsymbol z-\mathbf{s}|},
\end{align*}
where $\mathscr{P}(\boldsymbol z,\mathbf{s},\frac{\Omega c^{-1}_b\tau}{c_s})=\sum^3_{i,j=1}\mathscr{P}_{ij}p_{j},\mathscr{Q}(\boldsymbol z,\mathbf{s},\frac{\Omega c^{-1}_b\tau}{c_p})= \sum^3_{i,j=1}\mathscr{Q}_{ij}p_{j},$
with
\begin{align*}
\mathscr{P}_{ii}=\mathscr{A}^{(\ell)}_{\ell\ell}|_{\ell=i,\tilde{\boldsymbol x}=\boldsymbol z},\quad\mathscr{P}_{ij}=\mathscr{A}^{(\ell)}_{\ell j}|_{\ell=i, \tilde{\boldsymbol x}=\boldsymbol z},i\neq j,\\
\mathscr{Q}_{ii}=\mathscr{B}^{(\ell)}_{\ell\ell}|_{\ell=i,\tilde{\boldsymbol x}=\boldsymbol z},\quad \mathscr{Q}_{ij}=\mathscr{B}^{(\ell)}_{\ell j}|_{\ell=i, \tilde{\boldsymbol x}=\boldsymbol z},i\neq j,
\end{align*}
and $\boldsymbol{\mathscr{M}}(\boldsymbol z,\mathbf{s},\frac{\Omega c^{-1}_b\tau}{c_s})=(\mathscr{M}_{ij})^3_{i,j=1},\boldsymbol{\mathscr{N}}(\boldsymbol z,\mathbf{s},\frac{\Omega c^{-1}_b\tau}{c_p})=(\mathscr{N}_{ij})^3_{i,j=1}$ are $3\times3$ matrix whose entries are composed of $\mathscr{A}^{(\ell)}_{ij}|_{\tilde{\boldsymbol x}=\boldsymbol z},\mathscr{B}^{(\ell)}_{ij}|_{\tilde{\boldsymbol x}=\boldsymbol z}$ and $p_j$, $i,j,\ell=1,2,3$.

Observe that
\begin{align*}
\mathbf{S}^{\Omega c^{-1}_b\tau }_{\partial D}[\breve{Y}^m_n\boldsymbol\nu](\tilde{\boldsymbol x})=&\int_{\partial D}-e^{\mathrm{i}\frac{\Omega c^{-1}_b\tau}{c_s}|\tilde{\boldsymbol x}-\tilde{\boldsymbol y}|}\frac{\boldsymbol{\mathcal{A}}(\tilde{\boldsymbol x},\tilde{\boldsymbol y},\frac{\Omega c^{-1}_b\tau}{c_s})}{4\pi(\Omega c^{-1}_b\tau)^2|\tilde{\boldsymbol x}-\tilde{\boldsymbol y}|}\breve{Y}^{m}_{n}(\tilde{\boldsymbol y})\boldsymbol\nu_{\tilde{\boldsymbol y}}\mathrm{d}\sigma(\tilde{\boldsymbol y})\\
&+\int_{\partial D}-e^{\mathrm{i}\frac{\Omega c^{-1}_b\tau}{c_p}|\tilde{\boldsymbol x}-\tilde{\boldsymbol y}|}\frac{\boldsymbol{\mathcal{B}}(\tilde{\boldsymbol x},\tilde{\boldsymbol y},\frac{\Omega c^{-1}_b\tau}{c_p})}{4\pi(\Omega c^{-1}_b\tau)^2|\tilde{\boldsymbol x}-\tilde{\boldsymbol y}|}\breve{Y}^{m}_{n}(\tilde{\boldsymbol y})\boldsymbol\nu_{\tilde{\boldsymbol y}}\mathrm{d}\sigma(\tilde{\boldsymbol y}).
\end{align*}

Set $\tilde{\mathbf v}=(\tilde{v}_j)_{j=1}^3\in\partial D$ and define
\begin{align*}
&\boldsymbol{\mathscr{P}}^m_{n}(\tilde{\mathbf{v}},\tilde{\boldsymbol y},\Omega)\\
=&\frac{\delta\tau^2(\Omega\epsilon c^{-1}_b)^2\Omega^2}{\epsilon\lambda_{n}(\Omega\epsilon c^{-1}_b)}\frac{\boldsymbol\nu_{\tilde{\mathbf{v}}}\cdot(\boldsymbol{\mathcal{A}}(\boldsymbol z,\mathbf{s},\frac{\Omega c^{-1}_b\tau}{c_s})\mathbf{p})\overline{\breve{Y}^{m}_{n}(\tilde{\mathbf{v}})}
\boldsymbol{\mathcal{A}}(\tilde{\boldsymbol x},\tilde{\boldsymbol y},\frac{\Omega c^{-1}_b\tau}{c_s})\breve{Y}^{m}_{n}(\tilde{\boldsymbol y})\boldsymbol\nu_{\tilde{\boldsymbol{y}}}}{16\pi^2(\Omega c^{-1}_b\tau)^4|\boldsymbol z-\mathbf{s}||\tilde{\boldsymbol x}-\tilde{\boldsymbol y}|}\\
&+\frac{\delta\tau^2(\Omega\epsilon c^{-1}_b)^2\Omega^2}{\lambda_{n}(\Omega\epsilon c^{-1}_b)}\frac{\boldsymbol\nu_{\tilde{\mathbf{v}}}\cdot(\sum^3_{j=1}\frac{\tilde{v}_j-z_j}{\epsilon} \boldsymbol{\mathscr{A}}^{(j)}(\boldsymbol z,\mathbf{s},\frac{\Omega c^{-1}_b\tau}{c_s})\mathbf{p})\overline{\breve{Y}^{m}_{n}(\tilde{\mathbf{v}})}
\boldsymbol{\mathcal{A}}(\tilde{\boldsymbol x},\tilde{\boldsymbol y},\frac{\Omega c^{-1}_b\tau}{c_s})\breve{Y}^{m}_{n}(\tilde{\boldsymbol y})\boldsymbol\nu_{\tilde{\boldsymbol{y}}}}{16\pi^2(\Omega c^{-1}_b\tau)^4|\boldsymbol z-\mathbf{s}||\tilde{\boldsymbol x}-\tilde{\boldsymbol y}|}\\
&+\frac{\lambda\Omega^2(n-\frac{(\Omega\epsilon c^{-1}_b)^2c(\Omega)}{2n+3})}{\lambda_n(\Omega\epsilon c^{-1}_b)}\frac{\mathscr{P}(\boldsymbol z,\mathbf{s},\frac{\Omega c^{-1}_b\tau}{c_s})\overline{\breve{Y}^{m}_{n}(\tilde{\mathbf{v}})}
\boldsymbol{\mathcal{A}}(\tilde{\boldsymbol x},\tilde{\boldsymbol y},\frac{\Omega c^{-1}_b\tau}{c_s})\breve{Y}^{m}_{n}(\tilde{\boldsymbol y})\boldsymbol\nu_{\tilde{\boldsymbol{y}}}}{16\pi^2(\Omega c^{-1}_b\tau)^4|\boldsymbol z-\mathbf{s}||\tilde{\boldsymbol x}-\tilde{\boldsymbol y}|}\\
&+\frac{ \mu\Omega^2(n-\frac{(\Omega\epsilon c^{-1}_b)^2c(\Omega)}{2n+3})}{\lambda_{n}(\Omega\epsilon c^{-1}_b)}\frac{\boldsymbol\nu_{\tilde{\mathbf{v}}}\cdot(\boldsymbol{\mathscr{M}}(\boldsymbol z,\mathbf{s},\frac{\Omega c^{-1}_b\tau}{c_s})\boldsymbol\nu_{\tilde{\mathbf{v}}})\overline{\breve{Y}^{m}_{n}(\tilde{\mathbf{v}})}
\boldsymbol{\mathcal{A}}(\tilde{\boldsymbol x},\tilde{\boldsymbol y},\frac{\Omega c^{-1}_b\tau}{c_s})\breve{Y}^{m}_{n}(\tilde{\boldsymbol y})\boldsymbol\nu_{\tilde{\boldsymbol y}}}{16\pi^2(\Omega c^{-1}_b\tau)^4|\boldsymbol z-\mathbf{s}||\tilde{\boldsymbol x}-\tilde{\boldsymbol y}|},\\
&\boldsymbol{\mathscr{Q}}^m_{n}(\tilde{\mathbf{v}},\tilde{\boldsymbol y},\Omega)\\
=&\frac{\delta\tau^2(\Omega\epsilon c^{-1}_b)^2\Omega^2}{\epsilon\lambda_{n}(\Omega\epsilon c^{-1}_b)}\frac{\boldsymbol\nu_{\tilde{\mathbf{v}}}\cdot(\boldsymbol{\mathcal{A}}(\boldsymbol z,\mathbf{s},\frac{\Omega c^{-1}_b\tau}{c_s})\mathbf{p})\overline{\breve{Y}^{m}_{n}(\tilde{\mathbf{v}})}
\boldsymbol{\mathcal{B}}(\tilde{\boldsymbol x},\tilde{\boldsymbol y},\frac{\Omega c^{-1}_b\tau}{c_p})\breve{Y}^{m}_{n}(\tilde{\boldsymbol y})\boldsymbol\nu_{\tilde{\boldsymbol{y}}}}{16\pi^2(\Omega c^{-1}_b\tau)^4|\boldsymbol z-\mathbf{s}||\tilde{\boldsymbol x}-\tilde{\boldsymbol y}|}\\
&+\frac{\delta\tau^2(\Omega\epsilon c^{-1}_b)^2\Omega^2}{\lambda_{n}(\Omega\epsilon c^{-1}_b)}\frac{\boldsymbol\nu_{\tilde{\mathbf{v}}}\cdot(\sum^3_{j=1}\frac{\tilde{v}_j-z_j}{\epsilon} \boldsymbol{\mathscr{A}}^{(j)}(\boldsymbol z,\mathbf{s},\frac{\Omega c^{-1}_b\tau}{c_s})\mathbf{p})\overline{\breve{Y}^{m}_{n}(\tilde{\mathbf{v}})}
\boldsymbol{\mathcal{B}}(\tilde{\boldsymbol x},\tilde{\boldsymbol y},\frac{\Omega c^{-1}_b\tau}{c_p})\breve{Y}^{m}_{n}(\tilde{\boldsymbol y})\boldsymbol\nu_{\tilde{\boldsymbol{y}}}}{16\pi^2(\Omega c^{-1}_b\tau)^4|\boldsymbol z-\mathbf{s}||\tilde{\boldsymbol x}-\tilde{\boldsymbol y}|}\\
&+\frac{\lambda\Omega^2(n-\frac{(\Omega\epsilon c^{-1}_b)^2c(\Omega)}{2n+3})}{\lambda_n(\Omega \epsilon c^{-1}_b)}\frac{\mathscr{P}(\boldsymbol z,\mathbf{s},\frac{\Omega c^{-1}_b\tau}{c_s})\overline{\breve{Y}^{m}_{n}(\tilde{\mathbf{v}})}
\boldsymbol{\mathcal{B}}(\tilde{\boldsymbol x},\tilde{\boldsymbol y},\frac{\Omega c^{-1}_b\tau}{c_p})\breve{Y}^{m}_{n}(\tilde{\boldsymbol y})\boldsymbol\nu_{\tilde{\boldsymbol{y}}}}{16\pi^2(\Omega c^{-1}_b\tau)^4|\boldsymbol z-\mathbf{s}||\tilde{\boldsymbol x}-\tilde{\boldsymbol y}|}\\
&+\frac{ \mu\Omega^2(n-\frac{(\Omega\epsilon c^{-1}_b)^2c(\Omega)}{2n+3})}{\lambda_n(\Omega \epsilon c^{-1}_b)}\frac{\boldsymbol\nu_{\tilde{\mathbf{v}}}\cdot(\boldsymbol{\mathscr{M}}(\boldsymbol z,\mathbf{s},\frac{\Omega c^{-1}_b\tau}{c_s})\boldsymbol\nu_{\tilde{\mathbf{v}}})\overline{\breve{Y}^{m}_{n}(\tilde{\mathbf{v}})}
\boldsymbol{\mathcal{B}}(\tilde{\boldsymbol x},\tilde{\boldsymbol y},\frac{\Omega c^{-1}_b\tau}{c_p})\breve{Y}^{m}_{n}(\tilde{\boldsymbol y})\boldsymbol\nu_{\tilde{\boldsymbol y}}}{16\pi^2(\Omega c^{-1}_b\tau)^4|\boldsymbol z-\mathbf{s}||\tilde{\boldsymbol x}-\tilde{\boldsymbol y}|},\\
&\boldsymbol{\mathscr{R}}^m_{n}(\tilde{\mathbf{v}},\tilde{\boldsymbol y},\Omega)\\
=&\frac{\delta\tau^2(\Omega\epsilon c^{-1}_b)^2\Omega^2}{\epsilon\lambda_{n}(\Omega\epsilon c^{-1}_b)}\frac{\boldsymbol\nu_{\tilde{\mathbf{v}}}\cdot(\boldsymbol{\mathcal{B}}(\boldsymbol z,\mathbf{s},\frac{\Omega c^{-1}_b\tau}{c_p})\mathbf{p})\overline{\breve{Y}^{m}_{n}(\tilde{\mathbf{v}})}
\boldsymbol{\mathcal{A}}(\tilde{\boldsymbol x},\tilde{\boldsymbol y},\frac{\Omega c^{-1}_b\tau}{c_s})\breve{Y}^{m}_{n}(\tilde{\boldsymbol y})\boldsymbol\nu_{\tilde{\boldsymbol{y}}}}{16\pi^2(\Omega c^{-1}_b\tau)^4|\boldsymbol z-\mathbf{s}||\tilde{\boldsymbol x}-\tilde{\boldsymbol y}|}\\
&+\frac{\delta\tau^2(\Omega\epsilon c^{-1}_b)^2\Omega^2}{\lambda_{n}(\Omega\epsilon c^{-1}_b)}\frac{\boldsymbol\nu_{\tilde{\mathbf{v}}}\cdot(\sum^3_{j=1}\frac{\tilde{v}_j-z_j}{\epsilon} \boldsymbol{\mathscr{B}}^{(j)}(\boldsymbol z,\mathbf{s},\frac{\Omega c^{-1}_b\tau}{c_p})\mathbf{p})\overline{\breve{Y}^{m}_{n}(\tilde{\mathbf{v}})}
\boldsymbol{\mathcal{A}}(\tilde{\boldsymbol x},\tilde{\boldsymbol y},\frac{\Omega c^{-1}_b\tau}{c_s})\breve{Y}^{m}_{n}(\tilde{\boldsymbol y})\boldsymbol\nu_{\tilde{\boldsymbol{y}}}}{16\pi^2(\Omega c^{-1}_b\tau)^4|\boldsymbol z-\mathbf{s}||\tilde{\boldsymbol x}-\tilde{\boldsymbol y}|}\\
&+\frac{ \lambda\Omega^2(n-\frac{(\Omega\epsilon c^{-1}_b)^2c(\Omega)}{2n+3})}{\lambda_n(\Omega\epsilon c^{-1}_b)}\frac{\mathscr{Q}(\boldsymbol z,\mathbf{s},\frac{\Omega c^{-1}_b\tau}{c_p})\overline{\breve{Y}^{m}_{n}(\tilde{\mathbf{v}})}
\boldsymbol{\mathcal{A}}(\tilde{\boldsymbol x},\tilde{\boldsymbol y},\frac{\Omega c^{-1}_b\tau}{c_s})\breve{Y}^{m}_{n}(\tilde{\boldsymbol y})\boldsymbol\nu_{\tilde{\boldsymbol{y}}}}{16\pi^2(\Omega c^{-1}_b\tau)^4|\boldsymbol z-\mathbf{s}||\tilde{\boldsymbol x}-\tilde{\boldsymbol y}|}\\
&+\frac{\mu\Omega^2(n-\frac{(\Omega\epsilon c^{-1}_b)^2c(\Omega)}{2n+3})}{\lambda_{n}(\Omega\epsilon c^{-1}_b)}\frac{\boldsymbol\nu_{\tilde{\mathbf{v}}}\cdot(\boldsymbol{\mathscr{N}}(\boldsymbol z,\mathbf{s},\frac{\Omega c^{-1}_b\tau}{c_p})\boldsymbol\nu_{\tilde{\mathbf{v}}})\overline{\breve{Y}^{m}_{n}(\tilde{\mathbf{v}})}
\boldsymbol{\mathcal{A}}(\tilde{\boldsymbol x},\tilde{\boldsymbol y},\frac{\Omega c^{-1}_b\tau}{c_s})\breve{Y}^{m}_{n}(\tilde{\boldsymbol y})\boldsymbol\nu_{\tilde{\boldsymbol y}}}{16\pi^2(\Omega c^{-1}_b\tau)^4|\boldsymbol z-\mathbf{s}||\tilde{\boldsymbol x}-\tilde{\boldsymbol y}|},
\end{align*}
and
\begin{align*}
&\boldsymbol{\mathscr{S}}^m_{n}(\tilde{\mathbf{v}},\tilde{\boldsymbol y},\Omega)\\
=&\frac{\delta\tau^2(\Omega\epsilon c^{-1}_b)^2\Omega^2}{\epsilon\lambda_{n}(\Omega\epsilon c^{-1}_b)}\frac{\boldsymbol\nu_{\tilde{\mathbf{v}}}\cdot(\boldsymbol{\mathcal{B}}(\boldsymbol z,\mathbf{s},\frac{\Omega c^{-1}_b\tau}{c_p})\mathbf{p})\overline{\breve{Y}^{m}_{n}(\tilde{\mathbf{v}})}
\boldsymbol{\mathcal{B}}(\tilde{\boldsymbol x},\tilde{\boldsymbol y},\frac{\Omega c^{-1}_b\tau}{c_p})\breve{Y}^{m}_{n}(\tilde{\boldsymbol y})\boldsymbol\nu_{\tilde{\boldsymbol{y}}}}{16\pi^2(\Omega c^{-1}_b\tau)^4|\boldsymbol z-\mathbf{s}||\tilde{\boldsymbol x}-\tilde{\boldsymbol y}|}\\
&+\frac{\delta\tau^2(\Omega\epsilon c^{-1}_b)^2\Omega^2}{\lambda_{n}(\Omega\epsilon c^{-1}_b)}\frac{\boldsymbol\nu_{\tilde{\mathbf{v}}}\cdot(\sum^3_{j=1}\frac{\tilde{v}_j-z_j}{\epsilon} \boldsymbol{\mathscr{B}}^{(j)}(\boldsymbol z,\mathbf{s},\frac{\Omega c^{-1}_b\tau}{c_p})\mathbf{p})\overline{\breve{Y}^{m}_{n}(\tilde{\mathbf{v}})}
\boldsymbol{\mathcal{B}}(\tilde{\boldsymbol x},\tilde{\boldsymbol y},\frac{\Omega c^{-1}_b\tau}{c_p})\breve{Y}^{m}_{n}(\tilde{\boldsymbol y})\boldsymbol\nu_{\tilde{\boldsymbol{y}}}}{16\pi^2(\Omega c^{-1}_b\tau)^4|\boldsymbol z-\mathbf{s}||\tilde{\boldsymbol x}-\tilde{\boldsymbol y}|}\\
&+\frac{\lambda\Omega^2(n-\frac{(\Omega\epsilon c^{-1}_b)^2c(\Omega)}{2n+3})}{\lambda_n(\Omega\epsilon c^{-1}_b)}\frac{\mathscr{Q}(\boldsymbol z,\mathbf{s},\frac{\Omega c^{-1}_b\tau}{c_p})\overline{\breve{Y}^{m}_{n}(\tilde{\mathbf{v}})}
\boldsymbol{\mathcal{B}}(\tilde{\boldsymbol x},\tilde{\boldsymbol y},\frac{\Omega c^{-1}_b\tau}{c_p})\breve{Y}^{m}_{n}(\tilde{\boldsymbol y})\boldsymbol\nu_{\tilde{\boldsymbol{y}}}}{16\pi^2(\Omega c^{-1}_b\tau)^4|\boldsymbol z-\mathbf{s}||\tilde{\boldsymbol x}-\tilde{\boldsymbol y}|}\\
&+\frac{\mu\Omega^2(n-\frac{(\Omega\epsilon c^{-1}_b)^2c(\Omega)}{2n+3})}{\lambda_n(\Omega\epsilon c^{-1}_b)}\frac{\boldsymbol\nu_{\tilde{\mathbf{v}}}\cdot(\boldsymbol{\mathscr{N}}(\boldsymbol z,\mathbf{s},\frac{\Omega c^{-1}_b\tau}{c_p})\boldsymbol\nu_{\tilde{\mathbf{v}}})\overline{\breve{Y}^{m}_{n}(\tilde{\mathbf{v}})}
\boldsymbol{\mathcal{B}}(\tilde{\boldsymbol x},\tilde{\boldsymbol y},\frac{\Omega c^{-1}_b\tau}{c_p})\breve{Y}^{m}_{n}(\tilde{\boldsymbol y})\boldsymbol\nu_{\tilde{\boldsymbol y}}}{16\pi^2(\Omega c^{-1}_b\tau)^4|\boldsymbol z-\mathbf{s}||\tilde{\boldsymbol x}-\tilde{\boldsymbol y}|},
 \end{align*}
where $\boldsymbol{\mathscr{P}}^m_{n}(\cdot,\cdot,\Omega),\boldsymbol{\mathscr{Q}}^m_{n}(\cdot,\cdot,\Omega),\boldsymbol{\mathscr{R}}^m_{n}(\cdot,\cdot,\Omega) ,\boldsymbol{\mathscr{S}}^m_{n}(\cdot,\cdot,\Omega)$ behave like a polynomial in $\Omega$ when $|\Omega|\rightarrow+\infty$. As a consequence,
\begin{align*}
\int_{\mathcal{C}^{\pm}_{\rho}}\boldsymbol{\Xi}^{m}_{n}(\tilde{\boldsymbol x},\Omega)e^{-\mathrm{i}\Omega t}\mathrm{d}\Omega=&\boldsymbol{\Phi}^m_{1,n}+\boldsymbol{\Phi}^m_{2,n}+\boldsymbol{\Phi}^m_{3,n}+\boldsymbol{\Phi}^m_{4,n},
\end{align*}
where
\begin{align*}
\boldsymbol{\Phi}^m_{1,n}&=\int_{\mathcal{C}^{\pm}_{\rho}}f(\Omega)\int_{\partial D\times\partial D}\boldsymbol{\mathscr{P}}^m_{n}(\tilde{\mathbf{v}},\tilde{\boldsymbol y},\Omega)e^{\mathrm{i}\Omega
(\frac{c^{-1}_b\tau|\boldsymbol z-\mathbf{s}|}{c_s}+\frac{c^{-1}_b\tau|\tilde{\boldsymbol x}-\tilde{\boldsymbol y}|}{c_s}-t)}\mathrm{d}\sigma(\tilde{\mathbf{v}})\mathrm{d}\sigma(\tilde{\boldsymbol y})\mathrm{d}\Omega,\\
\boldsymbol{\Phi}^m_{2,n}&=\int_{\mathcal{C}^{\pm}_{\rho}}f(\Omega)\int_{\partial D\times\partial D}\boldsymbol{\mathscr{Q}}^m_{n}(\tilde{\mathbf{v}},\tilde{\boldsymbol y},\Omega)e^{\mathrm{i}\Omega
(\frac{c^{-1}_b\tau|\boldsymbol z-\mathbf{s}|}{c_s}+\frac{c^{-1}_b\tau|\tilde{\boldsymbol x}-\tilde{\boldsymbol y}|}{c_p}-t)}\mathrm{d}\sigma(\tilde{\mathbf{v}})\mathrm{d}\sigma(\tilde{\boldsymbol y})\mathrm{d}\Omega,\\
\boldsymbol{\Phi}^m_{3,n}&=\int_{\mathcal{C}^{\pm}_{\rho}}f(\Omega)\int_{\partial D\times\partial D}\boldsymbol{\mathscr{R}}^m_{n}(\tilde{\mathbf{v}},\tilde{\boldsymbol y},\Omega)e^{\mathrm{i}\Omega
(\frac{c^{-1}_b\tau|\boldsymbol z-\mathbf{s}|}{c_p}+\frac{c^{-1}_b\tau|\tilde{\boldsymbol x}-\tilde{\boldsymbol y}|}{c_s}-t)}\mathrm{d}\sigma(\tilde{\mathbf{v}})\mathrm{d}\sigma(\tilde{\boldsymbol y})\mathrm{d}\Omega,\\
\boldsymbol{\Phi}^m_{4,n}&=\int_{\mathcal{C}^{\pm}_{\rho}}f(\Omega)\int_{\partial D\times\partial D}\boldsymbol{\mathscr{S}}^m_{n}(\tilde{\mathbf{v}},\tilde{\boldsymbol y},\Omega)e^{\mathrm{i}\Omega
(\frac{c^{-1}_b\tau|\boldsymbol z-\mathbf{s}|}{c_p}+\frac{c^{-1}_b\tau|\tilde{\boldsymbol x}-\tilde{\boldsymbol y}|}{c_p}-t)}\mathrm{d}\sigma(\tilde{\mathbf{v}})\mathrm{d}\sigma(\tilde{\boldsymbol y})\mathrm{d}\Omega.
\end{align*}

Furthermore, since $\hat{f}:t\mapsto \hat{f}(t)\in C^{\infty}_0([0,C_1])$, by the Paley-Wiener theorem, we have the decay property of its Fourier transform at infinity. For any integer $M>0$, there exists some constant $C_{M}>0$ such that for all $\Omega\in\mathbb{C}$,
\begin{align*}
|f(\Omega)|\leq C_{M}(1+|\Omega|)^{-M}e^{C_1|\Im(\Omega)|}.
\end{align*}

The remainder of the discussion is divided into the following two cases.

Case {\rm(i)}: $t<t^{-}_0$. The upper-half integration contour $\mathcal{C}^+$ is considered. By the polar coordinate transform
\begin{align*}
\Omega=\rho e^{\mathrm{i}\theta},\quad \theta\in[0,\pi],
\end{align*}
one has
\begin{align*}
\boldsymbol\Phi^m_{1,n}&=\int_{0}^{\pi}\mathrm{i}\rho e^{\mathrm{i}\theta}f(\rho e^{\mathrm{i}\theta})\int_{\partial D\times\partial D}\boldsymbol{\mathscr{P}}^m_{n}(\tilde{\mathbf{v}},\tilde{\boldsymbol y},\rho e^{\mathrm{i}\theta})e^{\mathrm{i}\rho e^{\mathrm{i}\theta}
(\frac{c^{-1}_b\tau|\boldsymbol z-\mathbf{s}|}{c_s}+\frac{c^{-1}_b\tau|\tilde{\boldsymbol x}-\tilde{\boldsymbol y}|}{c_s}-t)}\mathrm{d}\sigma(\tilde{\mathbf{v}})\mathrm{d}\sigma(\tilde{\boldsymbol y})\mathrm{d}\theta,\\
\boldsymbol\Phi^m_{2,n}&=\int_{0}^{\pi}\mathrm{i}\rho e^{\mathrm{i}\theta}f(\rho e^{\mathrm{i}\theta})\int_{\partial D\times\partial D}\boldsymbol{\mathscr{Q}}^m_{n}(\tilde{\mathbf{v}},\tilde{\boldsymbol y},\rho e^{\mathrm{i}\theta})e^{\mathrm{i}\rho e^{\mathrm{i}\theta}
(\frac{c^{-1}_b\tau|\boldsymbol z-\mathbf{s}|}{c_s}+\frac{c^{-1}_b\tau|\tilde{\boldsymbol x}-\tilde{\boldsymbol y}|}{c_p}-t)}\mathrm{d}\sigma(\tilde{\mathbf{v}})\mathrm{d}\sigma(\tilde{\boldsymbol y})\mathrm{d}\theta,\\
\boldsymbol\Phi^m_{3,n}&=\int_{0}^{\pi}\mathrm{i}\rho e^{\mathrm{i}\theta}f(\rho e^{\mathrm{i}\theta})\int_{\partial D\times\partial D}\boldsymbol{\mathscr{R}}^m_{n}(\tilde{\mathbf{v}},\tilde{\boldsymbol y},\rho e^{\mathrm{i}\theta})e^{\mathrm{i}\rho e^{\mathrm{i}\theta}
(\frac{c^{-1}_b\tau|\boldsymbol z-\mathbf{s}|}{c_p}+\frac{c^{-1}_b\tau|\tilde{\boldsymbol x}-\tilde{\boldsymbol y}|}{c_s}-t)}\mathrm{d}\sigma(\tilde{\mathbf{v}})\mathrm{d}\sigma(\tilde{\boldsymbol y})\mathrm{d}\theta,\\
\boldsymbol\Phi^m_{4,n}&=\int_{0}^{\pi}\mathrm{i}\rho e^{\mathrm{i}\theta}f(\rho e^{\mathrm{i}\theta})\int_{\partial D\times\partial D}\boldsymbol{\mathscr{S}}^m_{n}(\tilde{\mathbf{v}},\tilde{\boldsymbol y},\rho e^{\mathrm{i}\theta})e^{\mathrm{i}\rho e^{\mathrm{i}\theta}
(\frac{c^{-1}_b\tau|\boldsymbol z-\mathbf{s}|}{c_p}+\frac{c^{-1}_b\tau|\tilde{\boldsymbol x}-\tilde{\boldsymbol y}|}{c_p}-t)}\mathrm{d}\sigma(\tilde{\mathbf{v}})\mathrm{d}\sigma(\tilde{\boldsymbol y})\mathrm{d}\theta.
\end{align*}

Let us  estimate $\boldsymbol\Phi^m_{1,n}$. If $t<\frac{c^{-1}_b\tau|\boldsymbol z-\mathbf{s}|}{c_s}+\frac{c^{-1}_b\tau|\tilde{\boldsymbol x}-\boldsymbol z|}{c_s}-\frac{c^{-1}_b\tau\epsilon}{c_s}-C_1$, then for all $\tilde{\boldsymbol y}\in\partial D$,
\begin{align*}
|e^{\mathrm{i}\rho e^{\mathrm{i}\theta}
(\frac{c^{-1}_b\tau|\boldsymbol z-\mathbf{s}|}{c_s}+\frac{c^{-1}_b\tau|\tilde{\boldsymbol x}-\tilde{\boldsymbol y}|}{c_s}-C_1-t)}|\leq e^{-\rho\sin \theta(\frac{c^{-1}_b\tau|\boldsymbol z-\mathbf{s}|}{c_s}+\frac{c^{-1}_b\tau|\tilde{\boldsymbol x}-\boldsymbol z|}{c_s}-\frac{c^{-1}_b\tau\epsilon}{c_s}-C_1-t)}.
\end{align*}
Observe that $\sin\theta\geq\frac{2}{\pi}\theta,\forall \theta\in[0,\frac{\pi}{2}].$ Hence, for $t<\frac{c^{-1}_b\tau|\boldsymbol z-\mathbf{s}|}{c_s}+\frac{c^{-1}_b\tau|\tilde{\boldsymbol x}-\boldsymbol z|}{c_s}-\frac{c^{-1}_b\tau\epsilon}{c_s}-C_1$,
\begin{align*}
e^{-\rho\sin \theta(\frac{c^{-1}_b\tau|\boldsymbol z-\mathbf{s}|}{c_s}+\frac{c^{-1}_b\tau|\tilde{\boldsymbol x}-\boldsymbol z|}{c_s}-\frac{c^{-1}_b\tau\epsilon}{c_s}-C_1-t)}\leq e^{-\frac{2\rho}{\pi}\theta(\frac{c^{-1}_b\tau|\boldsymbol z-\mathbf{s}|}{c_s}+\frac{c^{-1}_b\tau|\tilde{\boldsymbol x}-\boldsymbol z|}{c_s}-\frac{c^{-1}_b\tau\epsilon}{c_s}-C_1-t)}.
\end{align*}
Consequently,
\begin{align*}
\int^{\frac{\pi}{2}}_0e^{-\rho\sin \theta(\frac{c^{-1}_b\tau|\boldsymbol z-\mathbf{s}|}{c_s}+\frac{c^{-1}_b\tau|\boldsymbol x-\boldsymbol z|}{c_s}-\frac{c^{-1}_b\tau\epsilon}{c_s}-C_1-t)}\mathrm{d}\theta\leq & \int^{\frac{\pi}{2}}_0e^{-\frac{2\rho}{\pi}\theta(\frac{c^{-1}_b\tau|\boldsymbol z-\mathbf{s}|}{c_s}+\frac{c^{-1}_b\tau|\boldsymbol x-\boldsymbol z|}{c_s}-\frac{c^{-1}_b\tau\epsilon}{c_s}-C_1-t)}\mathrm{d}\theta\\
=&\frac{\pi\big(1-e^{-\rho(\frac{c^{-1}_b\tau|\boldsymbol z-\mathbf{s}|}{c_s}+\frac{c^{-1}_b\tau|\tilde{\boldsymbol x}-\boldsymbol z|}{c_s}-\frac{c^{-1}_b\tau\epsilon}{c_s}-C_1-t)}\big)}{2\rho\big(\frac{c^{-1}_b\tau|\boldsymbol z-\mathbf{s}|}{c_s}+\frac{c^{-1}_b\tau|\tilde{\boldsymbol x}-\boldsymbol z|}{c_s}-\frac{c^{-1}_b\tau\epsilon}{c_s}-C_1-t\big)}.
\end{align*}
As a result,
\begin{align*}
|\boldsymbol\Phi^m_{1,n}|
\leq&C_M\epsilon^3\rho(1+\rho)^{-M}\max_{\theta\in[0,\pi]}\|\boldsymbol{\mathscr{P}}^m_{n}(\cdot,\cdot,\rho  e^{\mathrm{i}\theta})\|_{L^{\infty}(\partial D\times \partial D)}\\
&\times\int^{\pi}_{0}e^{-\rho\sin \theta(\frac{c^{-1}_b\tau|\boldsymbol z-\mathbf{s}|}{c_s}+\frac{c^{-1}_b\tau|\tilde{\boldsymbol x}-\boldsymbol z|}{c_s}-\frac{c^{-1}_b\tau\epsilon}{c_s}-C_1-t)}\mathrm{d}\theta\\
\leq&2 C_M\epsilon^3\rho(1+\rho)^{-M}\max_{\theta\in[0,\pi]}\|\boldsymbol{\mathscr{P}}^m_{n}(\cdot,\cdot,\rho  e^{\mathrm{i}\theta})\|_{L^{\infty}(\partial D\times \partial D)}\\
&\times\int^{\frac{\pi}{2}}_{0}e^{-\rho\sin \theta(\frac{c^{-1}_b\tau|\boldsymbol z-\mathbf{s}|}{c_s}+\frac{c^{-1}_b\tau|\tilde{\boldsymbol x}-\boldsymbol z|}{c_s}-\frac{c^{-1}_b\tau\epsilon}{c_s}-C_1-t)}\mathrm{d}\theta\\
\leq& C_M\epsilon^3(1+\rho)^{-M}\max_{\theta\in[0,\pi]}\|\boldsymbol{\mathscr{P}}^m_{n}(\cdot,\cdot,\rho  e^{\mathrm{i}\theta})\|_{L^{\infty}(\partial D\times \partial D)}\\
&\times\frac{\pi\big(1-e^{-\rho(\frac{c^{-1}_b\tau|\boldsymbol z-\mathbf{s}|}{c_s}+\frac{c^{-1}_b\tau|\tilde{\boldsymbol x}-\boldsymbol z|}{c_s}-\frac{c^{-1}_b\tau\epsilon}{c_s}-C_1-t)}\big)}{\frac{c^{-1}_b\tau|\boldsymbol z-\mathbf{s}|}{c_s}+\frac{c^{-1}_b\tau|\tilde{\boldsymbol x}-\boldsymbol z|}{c_s}-\frac{c^{-1}_b\tau\epsilon}{c_s}-C_1-t}.
\end{align*}
Therefore,
\begin{align*}
|\boldsymbol\Phi^m_{1,n}|=\mathcal{O}\big(\frac{\epsilon^3\rho^{-M}}{\frac{c^{-1}_b\tau|\boldsymbol z-\mathbf{s}|}{c_s}+\frac{c^{-1}_b\tau|\tilde{\boldsymbol x}-\boldsymbol z|}{c_s}-\frac{c^{-1}_b\tau\epsilon}{c_s}-C_1-t}\big).
\end{align*}
Proceeding the similar technique as above yields the following facts:

{\rm (i)} If $t<\frac{c^{-1}_b\tau|\boldsymbol z-\mathbf{s}|}{c_s}+\frac{c^{-1}_b\tau|\tilde{\boldsymbol x}-\boldsymbol z|}{c_p}-\frac{c^{-1}_b\tau\epsilon}{c_p}-C_1$, then
\begin{align*}
|\boldsymbol\Phi^m_{2,n}|=\mathcal{O}\big(\frac{\epsilon^3\rho^{-M}}{\frac{c^{-1}_b\tau|\boldsymbol z-\mathbf{s}|}{c_s}+\frac{c^{-1}_b\tau|\tilde{\boldsymbol x}-\boldsymbol z|}{c_p}-\frac{c^{-1}_b\tau\epsilon}{c_p}-C_1-t}\big).
\end{align*}

{\rm (ii)} If $t<\frac{c^{-1}_b\tau|\boldsymbol z-\mathbf{s}|}{c_p}+\frac{c^{-1}_b\tau|\tilde{\boldsymbol x}-\boldsymbol z|}{c_s}-\frac{c^{-1}_b\tau\epsilon}{c_s}-C_1$, then
\begin{align*}
&|\boldsymbol\Phi^m_{3,n}|=\mathcal{O}\big(\frac{\epsilon^3\rho^{-M}}{\frac{c^{-1}_b\tau|\boldsymbol z-\mathbf{s}|}{c_p}+\frac{c^{-1}_b\tau|\tilde{\boldsymbol x}-\boldsymbol z|}{c_s}-\frac{c^{-1}_b\tau\epsilon}{c_s}-C_1-t}\big).
\end{align*}

{\rm (iii)} If $t<\frac{c^{-1}_b\tau|\boldsymbol z-\mathbf{s}|}{c_p}+\frac{c^{-1}_b\tau|\tilde{\boldsymbol x}-\boldsymbol z|}{c_p}-\frac{c^{-1}_b\tau\epsilon}{c_p}-C_1$, then
\begin{align*}
|\boldsymbol\Phi^m_{4,n}|=\mathcal{O}\big(\frac{\epsilon^3\rho^{-M}}{\frac{c^{-1}_b\tau|\boldsymbol z-\mathbf{s}|}{c_p}+\frac{c^{-1}_b\tau|\tilde{\boldsymbol x}-\boldsymbol z|}{c_p}-\frac{c^{-1}_b\tau\epsilon}{c_p}-C_1-t}\big).
\end{align*}
Since $c_p\geq c_s$ by \eqref{relation}, formula \eqref{E:scattered} holds for all $t\leq t^{-}_0$.

Case {\rm(ii)}: $t>t^{+}_0$. Let us study the lower-half integration contour $\mathcal{C}^-$. Using the polar coordinate transform
\begin{align*}
\Omega=\rho e^{\mathrm{i}\theta},\quad \theta\in[\pi,2\pi]
\end{align*}
yields that
\begin{align*}
\boldsymbol\Phi^m_{1,n}=&\int_{\pi}^{2\pi}\mathrm{i}\rho e^{\mathrm{i}\theta}f(\rho e^{\mathrm{i}\theta})\int_{\partial D\times\partial D}\boldsymbol{\mathscr{P}}^m_{n}(\tilde{\mathbf{v}},\tilde{\boldsymbol y},\rho e^{\mathrm{i}\theta})e^{\mathrm{i}\rho e^{\mathrm{i}\theta}(\frac{c^{-1}_b\tau|\boldsymbol z-\mathbf{s}|}{c_s}+\frac{c^{-1}_b\tau|\tilde{\boldsymbol x}-\tilde{\boldsymbol y}|}{c_s}-t)}\mathrm{d}\sigma(\tilde{\mathbf{v}})\mathrm{d}\sigma(\tilde{\boldsymbol y})\mathrm{d}\theta,\\
\boldsymbol\Phi^m_{2,n}=&\int_{\pi}^{2\pi}\mathrm{i}\rho e^{\mathrm{i}\theta}f(\rho e^{\mathrm{i}\theta})\int_{\partial D\times\partial D}\boldsymbol{\mathscr{Q}}^m_{n}(\tilde{\mathbf{v}},\tilde{\boldsymbol y},\rho e^{\mathrm{i}\theta})e^{\mathrm{i}\rho e^{\mathrm{i}\theta}
(\frac{c^{-1}_b\tau|\boldsymbol z-\mathbf{s}|}{c_s}+\frac{c^{-1}_b\tau|\tilde{\boldsymbol x}-\tilde{\boldsymbol y}|}{c_p}-t)}\mathrm{d}\sigma(\tilde{\mathbf{v}})\mathrm{d}\sigma(\tilde{\boldsymbol y})\mathrm{d}\theta,\\
\boldsymbol\Phi^m_{3,n}=&\int_{\pi}^{2\pi}\mathrm{i}\rho e^{\mathrm{i}\theta}f(\rho e^{\mathrm{i}\theta})\int_{\partial D\times\partial D}\boldsymbol{\mathscr{R}}^m_{n}(\tilde{\mathbf{v}},\tilde{\boldsymbol y},\rho e^{\mathrm{i}\theta})e^{\mathrm{i}\rho e^{\mathrm{i}\theta}
(\frac{c^{-1}_b\tau|\boldsymbol z-\mathbf{s}|}{c_p}+\frac{c^{-1}_b\tau|\tilde{\boldsymbol x}-\tilde{\boldsymbol y}|}{c_s}-t)}\mathrm{d}\sigma(\tilde{\mathbf{v}})\mathrm{d}\sigma(\tilde{\boldsymbol y})\mathrm{d}\theta,\\
\boldsymbol\Phi^m_{4,n}=&\int_{\pi}^{2\pi}\mathrm{i}\rho e^{\mathrm{i}\theta}f(\rho e^{\mathrm{i}\theta})\int_{\partial D\times\partial D}\boldsymbol{\mathscr{S}}^m_{n}(\tilde{\mathbf{v}},\tilde{\boldsymbol y},\rho e^{\mathrm{i}\theta})e^{\mathrm{i}\rho e^{\mathrm{i}\theta}
(\frac{c^{-1}_b\tau|\boldsymbol z-\mathbf{s}|}{c_p}+\frac{c^{-1}_b\tau|\tilde{\boldsymbol x}-\tilde{\boldsymbol y}|}{c_p}-t)}\mathrm{d}\sigma(\tilde{\mathbf{v}})\mathrm{d}\sigma(\tilde{\boldsymbol y})\mathrm{d}\theta.
\end{align*}

Let us give the upper bound of $|\boldsymbol\Phi^m_{1,n}|$. If  $t>\frac{c^{-1}_b\tau|\boldsymbol z-\mathbf{s}|}{c_s}+\frac{c^{-1}_b\tau|\tilde{\boldsymbol x}-\boldsymbol z|}{c_s}+\frac{c^{-1}_b\tau\epsilon}{c_s}+C_1$, then for all $\tilde{\boldsymbol y}\in\partial D$,
\begin{align*}
|e^{\mathrm{i}\rho e^{\mathrm{i}\theta}
(\frac{c^{-1}_b\tau|\boldsymbol z-\mathbf{s}|}{c_s}+\frac{c^{-1}_b\tau|\tilde{\boldsymbol x}-\tilde{\boldsymbol y}|}{c_s}+C_1-t)}|
\leq &e^{\rho\sin \theta(t-(\frac{c^{-1}_b\tau|\boldsymbol z-\mathbf{s}|}{c_s}+\frac{c^{-1}_b\tau|\tilde{\boldsymbol x}-\tilde{\boldsymbol y}|}{c_s}+C_1))} \\
\leq& e^{\rho\sin \theta(t-(\frac{c^{-1}_b\tau|\boldsymbol z-\mathbf{s}|}{c_s}+\frac{c^{-1}_b\tau|\tilde{\boldsymbol x}-\boldsymbol z|}{c_s}+\frac{c^{-1}_b\tau\epsilon}{c_s}+C_1))}.
\end{align*}
Moreover, it is straightforward to see that $-\frac{2}{\pi}\theta+2\geq\sin\theta, \forall \theta\in[\pi,\frac{3\pi}{2}].$ One shows that for $t>\frac{c^{-1}_b\tau|\boldsymbol z-\mathbf{s}|}{c_s}+\frac{c^{-1}_b\tau|\tilde{\boldsymbol x}-\boldsymbol z|}{c_s}+\frac{c^{-1}_b\tau\epsilon}{c_s}+C_1$,
\begin{align*}
e^{\rho\sin \theta(t-(\frac{c^{-1}_b\tau|\boldsymbol z-\mathbf{s}|}{c_s}+\frac{c^{-1}_b\tau|\tilde{\boldsymbol x}-\boldsymbol z|}{c_s}+\frac{c^{-1}_b\tau\epsilon}{c_s}+C_1))}\leq e^{\rho(-\frac{2}{\pi}\theta+2)(t-(\frac{c^{-1}_b\tau|\boldsymbol z-\mathbf{s}|}{c_s}+\frac{c^{-1}_b\tau|\tilde{\boldsymbol x}-\boldsymbol z|}{c_s}+\frac{c^{-1}_b\tau\epsilon}{c_s}+C_1))},
\end{align*}
which leads to
\begin{align*}
\int^{\frac{3\pi}{2}}_{\pi}e^{\rho\sin \theta(t-(\frac{c^{-1}_b\tau|\boldsymbol z-\mathbf{s}|}{c_s}+\frac{c^{-1}_b\tau|\tilde{\boldsymbol x}-\boldsymbol z|}{c_s}+\frac{c^{-1}_b\tau\epsilon}{c_s}+C_1))}\mathrm{d}\theta\leq & \int^{\frac{3\pi}{2}}_{\pi}e^{\rho(-\frac{2}{\pi}\theta+2)(t-(\frac{c^{-1}_b\tau|\boldsymbol z-\mathbf{s}|}{c_s}+\frac{c^{-1}_b\tau|\tilde{\boldsymbol x}-\boldsymbol z|}{c_s}+\frac{c^{-1}_b\tau\epsilon}{c_s}+C_1))}\mathrm{d}\theta\\
=&\frac{\pi\big(1-e^{-\rho(t-(\frac{c^{-1}_b\tau|\boldsymbol z-\mathbf{s}|}{c_s}+\frac{c^{-1}_b\tau|\tilde{\boldsymbol x}-\boldsymbol z|}{c_s}+\frac{c^{-1}_b\tau\epsilon}{c_s}+C_1))}\big)}{2\rho\big(t-(\frac{c^{-1}_b\tau|\boldsymbol z-\mathbf{s}|}{c_s}+\frac{c^{-1}_b\tau|\tilde{\boldsymbol x}-\boldsymbol z|}{c_s}+\frac{c^{-1}_b\tau\epsilon}{c_s}+C_1)\big)}.
\end{align*}
We can derive that
\begin{align*}
|\boldsymbol\Phi^m_{1,n}|
\leq&C_M\epsilon^3\rho(1+\rho)^{-M}\max_{\theta\in[\pi,2\pi]}\|\boldsymbol{\mathscr{P}}^m_{n}(\cdot,\cdot,\rho  e^{\mathrm{i}\theta})\|_{L^{\infty}(\partial D\times \partial D)}\\
&\times\int^{2\pi}_{\pi}e^{\rho\sin \theta(t-(\frac{c^{-1}_b\tau|\boldsymbol z-\mathbf{s}|}{c_s}+\frac{c^{-1}_b\tau|\tilde{\boldsymbol x}-\boldsymbol z|}{c_s}+\frac{c^{-1}_b\tau\epsilon}{c_s}+C_1))}\mathrm{d}\theta\\
\leq&2 C_M\epsilon^3\rho(1+\rho)^{-M}\max_{\theta\in[\pi,2\pi]}\|\boldsymbol{\mathscr{P}}^m_{n}(\cdot,\cdot,\rho  e^{\mathrm{i}\theta})\|_{L^{\infty}(\partial D\times \partial D)}\\
&\times\int^{\frac{3\pi}{2}}_{\pi}e^{\rho\sin \theta(t-(\frac{c^{-1}_b\tau|\boldsymbol z-\mathbf{s}|}{c_s}+\frac{c^{-1}_b\tau|\tilde{\boldsymbol x}-\boldsymbol z|}{c_s}+\frac{c^{-1}_b\tau\epsilon}{c_s}+C_1))}\mathrm{d}\theta\\
\leq& C_M\epsilon^3(1+\rho)^{-M}\max_{\theta\in[\pi,2\pi]}\|\boldsymbol{\mathscr{P}}^m_{n}(\cdot,\cdot,\rho  e^{\mathrm{i}\theta})\|_{L^{\infty}(\partial D\times \partial D)}\\
&\times\frac{\pi\big(1-e^{-\rho(t-(\frac{c^{-1}_b\tau|\boldsymbol z-\mathbf{s}|}{c_s}+\frac{c^{-1}_b\tau|\tilde{\boldsymbol x}-\boldsymbol z|}{c_s}+\frac{c^{-1}_b\tau\epsilon}{c_s}+C_1))}\big)}{t-(\frac{c^{-1}_b\tau|\boldsymbol z-\mathbf{s}|}{c_s}+\frac{c^{-1}_b\tau|\tilde{\boldsymbol x}-\boldsymbol z|}{c_s}+\frac{c^{-1}_b\tau\epsilon}{c_s}+C_1)},
\end{align*}
which then gives
\begin{align*}
|\boldsymbol\Phi^m_{1,n}|=\mathcal{O}\big(\frac{\epsilon^3\rho^{-M}}{t-(\frac{c^{-1}_b\tau|\boldsymbol z-\mathbf{s}|}{c_s}+\frac{c^{-1}_b\tau|\tilde{\boldsymbol x}-\boldsymbol z|}{c_s}+\frac{c^{-1}_b\tau\epsilon}{c_s}+C_1)}\big).
\end{align*}
By the similar argument as above, the following facts hold:

{\rm (i)} If $t>\frac{c^{-1}_b\tau|\boldsymbol z-\mathbf{s}|}{c_s}+\frac{c^{-1}_b\tau|\tilde{\boldsymbol x}-\boldsymbol z|}{c_p}+\frac{c^{-1}_b\tau\epsilon}{c_p}+C_1$, then
\begin{align*}
|\boldsymbol\Phi^m_{2,n}|=\mathcal{O}\big(\frac{\epsilon^3\rho^{-M}}{t-(\frac{c^{-1}_b\tau|\boldsymbol z-\mathbf{s}|}{c_s}+\frac{c^{-1}_b\tau|\tilde{\boldsymbol x}-\boldsymbol z|}{c_p}+\frac{c^{-1}_b\tau\epsilon}{c_p}+C_1)}\big).
\end{align*}

{\rm (ii)} If $t>\frac{c^{-1}_b\tau|\boldsymbol z-\mathbf{s}|}{c_p}+\frac{c^{-1}_b\tau|\tilde{\boldsymbol x}-\boldsymbol z|}{c_s}+\frac{c^{-1}_b\tau\epsilon}{c_s}+C_1$, then
\begin{align*}
|\boldsymbol\Phi^m_{3,n}|=\mathcal{O}\big(\frac{\epsilon^3\rho^{-M}}{t-(\frac{c^{-1}_b\tau|\boldsymbol z-\mathbf{s}|}{c_p}+\frac{c^{-1}_b\tau|\tilde{\boldsymbol x}-\boldsymbol z|}{c_s}+\frac{c^{-1}_b\tau\epsilon}{c_s}+C_1)}\big).
\end{align*}

{\rm (iii)} If $t>\frac{c^{-1}_b\tau|\boldsymbol z-\mathbf{s}|}{c_p}+\frac{c^{-1}_b\tau|\tilde{\boldsymbol x}-\boldsymbol z|}{c_p}+\frac{c^{-1}_b\tau\epsilon}{c_p}+C_1$, then
\begin{align*}
|\boldsymbol\Phi^m_{4,n}|=\mathcal{O}\big(\frac{\epsilon^3\rho^{-M}}{t-(\frac{c^{-1}_b\tau|\boldsymbol z-\mathbf{s}|}{c_p}+\frac{c^{-1}_b\tau|\tilde{\boldsymbol x}-\boldsymbol z|}{c_p}+\frac{c^{-1}_b\tau\epsilon}{c_p}+C_1)}\big).
\end{align*}
Hence formula \eqref{EE:scattered} holds for all $t\geq t^{+}_0$.


The proof of the theorem is now complete.
\end{proof}

\appendix

\section{Proofs of some lemmas}\label{App:A}

In the Appendix \ref{App:A}, we will supplement the proof of Lemma \ref{le:NUS}, \ref{le:NUK} and \ref{le:resonance} for completeness.

\subsection{Proof of Lemma \ref{le:NUS}}\label{App:A.1}
\begin{proof}
From the expressions of $\mathcal{I}^{m}_{n}$ and $\mathcal{N}^{m}_{n}$, it follows that
\begin{align*}
\mathcal{I}^{m}_{n-1}&=\nabla_{\partial B} Y^m_{n}+nY^m_n\boldsymbol \nu,\\
\mathcal{N}^{m}_{n+1}&=-\nabla_{\partial B} Y^m_{n}+(n+1)Y^m_n\boldsymbol \nu.
\end{align*}
In view of \eqref{E:SI} and \eqref{E:SN}, we have
\begin{align*}
&\mathbf{S}^{k}_{\partial B}[\nabla_{\partial B} Y^m_{n}+nY^m_n\boldsymbol \nu])\\
&=c_{1n}(k)(\nabla_{\partial B} Y^m_{n}+nY^m_n\boldsymbol \nu)+d_{1n}(k)(-\nabla_{\partial B} Y^m_{n}+(n+1)Y^m_n\boldsymbol \nu),\\
&\mathbf{S}^{k}_{\partial B}[-\nabla_{\partial B} Y^m_{n}+(n+1)Y^m_n\boldsymbol \nu]\\
&=c_{2n}(k)(\nabla_{\partial B} Y^m_{n}+nY^m_n\boldsymbol \nu)+d_{2n}(k)(-\nabla_{\partial B} Y^m_{n}+(n+1)Y^m_n\boldsymbol \nu).
\end{align*}
This leads to
\begin{align*}
\boldsymbol\nu \cdot\mathbf{S}^{k}_{\partial B}[\nabla_{\partial B} Y^m_{n}+nY^m_n\boldsymbol \nu]&=nc_{1n}(k)Y^m_n+(n+1)d_{1n}(k)Y^m_n,\\
\boldsymbol\nu \cdot\mathbf{S}^{k}_{\partial B}[-\nabla_{\partial B} Y^m_{n}+(n+1)Y^m_n\boldsymbol \nu]&=nc_{2n}(k)Y^m_n+(n+1)d_{2n}(k)Y^m_n.
\end{align*}
As a result,
\begin{align*}
\boldsymbol\nu \cdot\mathbf{S}^{k}_{\partial B}[Y^m_n\boldsymbol \nu]=\frac{n(c_{1n}(k)+c_{2n}(k))+(n+1)(d_{1n}(k)+d_{2n}(k))}{2n+1} Y^m_n,
\end{align*}
which then gives \eqref{E:NUS} and \eqref{G:etan}.

The remainder of the proof is to check \eqref{S:eta0} and \eqref{S:eta}. Using \eqref{E:bessel} and \eqref{E:hankel} shows that
\begin{align*}
\eta_0(k)&=\frac{n(c_{1n}(k)+c_{2n}(k))+(n+1)(d_{1n}(k)+d_{2n}(k))}{2n+1}\big|_{n=0}\\
&=-\mathrm{i}\frac{j_1(k_p)h_1(k_p)k_p}{\lambda+2\mu}\nonumber\\
&=-\frac{1}{3(\lambda+2\mu)}-\frac{2}{15(\lambda+2\mu)}k^2_p+\mathcal{O}(k^3).
\end{align*}
Hence \eqref{S:eta0} holds.

Moreover, by \eqref{E:recu1} and \eqref{E:recu2}, we derive that for $n\geq1$,
\begin{align*}
\eta_{n}(k)=&-\frac{\mathrm{i}}{2n+1}\bigg(\frac{n(n+1)(j_{n-1}(k_s)+j_{n+1}(k_s))(h_{n-1}(k_s)+h_{n+1}(k_s))k_s}{\mu(2n+1)}\\
&-
\frac{n(n+1)(j_{n-1}(k_p)+j_{n+1}(k_p))(h_{n-1}(k_p)+h_{n+1}(k_p))k_p}{(\lambda+2\mu)(2n+1)}\\
&+\frac{n(2n+1)j_{n-1}(k_p)h_{n-1}(k_p)k_p}{(\lambda+2\mu)(2n+1)}+\frac{(n+1)(2n+1)j_{n+1}(k_p)h_{n+1}(k_p)k_p}{(\lambda+2\mu)(2n+1)}\bigg)\\
=&-\mathrm{i}\bigg(\frac{n(n+1)j_n(k_s)h_n(k_s)}{\mu k_s}-\frac{n(n+1)j_n(k_p)h_n(k_p)}{(\lambda+2\mu)k_p}\nonumber\\
&+\frac{n j_{n-1}(k_p)h_{n-1}(k_p)k_p }{(\lambda+2\mu)(2n+1)}+\frac{(n+1)j_{n+1}(k_p)h_{n+1}(k_p)k_p}{(\lambda+2\mu)(2n+1)}\bigg).
\end{align*}
Thanks to \eqref{E:bessel}--\eqref{E:hankel} and \eqref{E:wave-ps}, one has that for fixed $n\geq1$,
\begin{align*}
\eta_n(k)=&-\bigg(\frac{n(n+1)}{\mu(2n+1)}\bigg(\frac{1}{k^2_s}+\frac{2}{(2n+3)(2n-1)}+\mathcal{O}(k)\bigg)\\
&-\frac{n(n+1)}{(\lambda+2\mu)(2n+1)}\bigg(\frac{1}{k^2_p}+\frac{2}{(2n+3)(2n-1)}+\mathcal{O}(k)\bigg)\\
&+\frac{n}{(\lambda+2\mu)(2n+1)(2n-1)}(1+\mathcal{O}(k))+\frac{n+1}{(\lambda+2\mu)(2n+3)(2n+1)}(1+\mathcal{O}(k))\bigg)\\
=&-\bigg(\frac{2(\lambda+\mu)n(n+1)}{\mu(\lambda+2\mu)(2n+3)(2n+1)(2n-1)}+\frac{\mu n(2n+3)}{\mu(\lambda+2\mu)(2n+3)(2n+1)(2n-1)}\\
&+\frac{\mu(n+1)(2n-1)}{\mu(\lambda+2\mu)(2n+3)(2n+1)(2n-1)}+\mathcal{O}(k)\bigg)\\
=&-\frac{2(\lambda+\mu)n(n+1)+\mu(4n^4+4n-1)}{\mu(\lambda+2\mu)(2n+3)(2n+1)(2n-1)}+\mathcal{O}(k).
\end{align*}
The proof is complete.
\end{proof}


\subsection{Proof of Lemma \ref{le:NUK}}\label{App:A.2}

\begin{proof}
From \eqref{Jump2}, \eqref{E:PSI} and \eqref{E:PSN}, it is easy to verify that
\begin{align*}
\big(\frac{1}{2}\mathbf{I}+\mathbf{K}^{k,*}_{\partial B}\big)[\mathcal{I}^{m}_{n-1}]&=\mathfrak{c}_{1n}(k)\mathcal{I}^{m}_{n-1}+\mathfrak{d}_{1n}(k)\mathcal{N}^m_{n+1},\\
\big(\frac{1}{2}\mathbf{I}+\mathbf{K}^{k,*}_{\partial B}\big)[\mathcal{N}^{m}_{n+1}]&=\mathfrak{c}_{2n}(k)\mathcal{I}^{m}_{n-1}+\mathfrak{d}_{2n}(k)\mathcal{N}^m_{n+1}.
\end{align*}
Because of the expressions of $\mathcal{I}^{m}_{n-1}$ and $\mathcal{N}^{m}_{n+1}$, one arrives at
\begin{align*}
\boldsymbol\nu \cdot\big(\frac{1}{2}\mathbf{I}+\mathbf{K}^{k,*}_{\partial B}\big)[Y^m_n\boldsymbol \nu]=\frac{n(\mathfrak{c}_{1n}(k)+\mathfrak{c}_{2n}(k))+(n+1)(\mathfrak{d}_{1n}(k)+\mathfrak{d}_{2n}(k))}{2n+1}Y^m_n.
\end{align*}
Thus we get  \eqref{E:NUK} and \eqref{G:rhon}.

Moreover, observe that
\begin{align*}
\rho_0(k)&=\frac{n(\mathfrak{c}_{1n}(k)+\mathfrak{c}_{2n}(k))+(n+1)(\mathfrak{d}_{1n}(k)+\mathfrak{d}_{2n}(k))}{2n+1}\big|_{n=0}\\
&=\mathrm{i}\left(\frac{4j_1(k_p)h_1(k_p)k_p\mu}{\lambda+2\mu}-j_1(k_p)h_0(k_p)k_p^2\right)\\
&=\frac{4\mu}{3(\lambda+2\mu)}-\frac{5\lambda+2\mu}{15(\lambda+2\mu)}k^2_p+\mathcal{O}(k^3).
\end{align*}
As a result, \eqref{S:rho0} holds.

Finally, let us  check  \eqref{S:rho}. It follows from \eqref{E:recu1} and \eqref{E:recu2} that
\begin{align*}
\frac{n(\mathfrak{c}_{1n}(k)+\mathfrak{c}_{2n}(k))}{2n+1}\big|_{n=1}&=\frac{\mathrm{i}}{9}\left(4j_0(k_s)h_1(k_s)k^2_s-j_0(k_p)h_1(k_p)k^2_p\right)\\
&=\frac{1}{9}\left(3+\frac{4}{3}k^2_s-\frac{1}{3}k^2_p+\mathcal{O}(k^3)\right),
\end{align*}
and that for fixed $n\geq2$,
\begin{align*}
\frac{n(\mathfrak{c}_{1n}(k)+\mathfrak{c}_{2n}(k))}{2n+1}=&\frac{\mathrm{i}n}{2n+1}\bigg(-2 (n^2-1)j_{n}(k_s)h_{n-1}(k_s)-\frac{2\mu n(n-1)j_{n}(k_p)h_{n-1}(k_p)}{\lambda+2\mu}\\
&\qquad\qquad+\frac{2\mu(n-1)j_{n+1}(k_p)h_{n-1}(k_p)k_p}{\lambda+2\mu}+\frac{2(n+1)j_{n-1}(k_s)h_{n}(k_s)k^2_s}{2n+1}\\
&\qquad\qquad-\frac{j_{n-1}(k_p)h_{n}(k_p)k^2_p}{2n+1}\bigg)\\
=&\frac{n}{2n+1}\bigg(\frac{-2(n^2-1)}{(2n+1)(2n-1)}\bigg(1+\frac{3}{(2n+3)(2n-3)}k^2_s+\mathcal{O}(k^3)\bigg)\\
&+\frac{-2\mu n(n-1)}{(\lambda+2\mu)(2n+1)(2n-1)}\bigg(1+\frac{3}{(2n+3)(2n-3)}k^2_p+\mathcal{O}(k^3)\bigg)\\
&+\frac{2\mu(n-1)}{(\lambda+2\mu)(2n+3)(2n+1)(2n-1)}(k^2_p+\mathcal{O}(k^3))\\
&+\frac{2(n+1)}{2n+1}\bigg(1+\frac{1}{(2n+1)(2n-1)}k^2_s+\mathcal{O}(k^3)\bigg)\\
&-\frac{1}{2n+1}\bigg(1+\frac{1}{(2n+1)(2n-1)}k^2_p+\mathcal{O}(k^3)\bigg)\bigg)\\
=&\frac{(\lambda+\mu )n(2n^2+1)+\mu n(2n+1)}{(\lambda+2\mu)(2n+1)^2(2n-1)}+\frac{-2n(n+1)(2n^2-3n+6)}{(2n+3)(2n+1)^3(2n-1)(2n-3)}k^2_s\\
&+\frac{-2\mu n(n-1)(2n^2+7n+3)-(\lambda+2\mu)n(2n+3)(2n-3)}{(\lambda+2\mu)(2n+3)(2n+1)^3(2n-1)(2n-3)}k^2_p+\mathcal{O}(k^3).
\end{align*}
It remains to calculate
\begin{align*}
\frac{(n+1)(\mathfrak{d}_{1n}(k)+\mathfrak{d}_{2n}(k))}{2n+1},\quad n\geq1.
\end{align*}
We can derive that for fixed $n\in\mathbb{N}$ and $0<|z|\ll 1$,
\begin{align*}
j_n(z)&=\frac{z^n}{(2n+1)!!}\left(1-\frac{z^2}{2(2n+3)}+\frac{z^4}{8(2n+3)(2n+5)}+\mathcal{O}(z^5)\right),
\end{align*}
and that for fixed $n\in\mathbb{N}$ and $0<|z|\ll 1$,
\begin{align}
h_0(z)&=\frac{1}{\mathrm{i} z}\left(1+\mathrm{i}z-\frac{z^2}{2}-\frac{\mathrm{i}z^3}{6}+\frac{z^4}{24}+\mathcal{O}(z^5)\right),\label{E:hzero}\\
h_1(z)&=\frac{1}{\mathrm{i} z^2}\left(1+\frac{z^2}{2}+\frac{\mathrm{i}z^3}{3}-\frac{z^4}{8}+\mathcal{O}(z^5)\right),\nonumber\\
h_n(z)&=\frac{(2n-1)!!}{\mathrm{i}z^{n+1}}\left(1-\frac{z^2}{2(-2n+1)}+\frac{z^4}{8(-2n+1)(-2n+3)}+\mathcal{O}(z^5)\right),\quad n\geq2.\nonumber
\end{align}
Therefore, for fixed $n\geq1$,
\begin{align*}
&\frac{(n+1)(\mathfrak{d}_{1n}(k)+\mathfrak{d}_{2n}(k))}{2n+1}\\
=&\frac{\mathrm{i}(n+1)}{2n+1}\bigg(2n(n+2)j_n(k_s)h_{n+1}(k_s)-\frac{2\mu n(n+2)j_n(k_p)h_{n+1}(k_p)}{\lambda+2\mu}\\
&+\frac{2\mu(n+2)j_{n+1}(k_p)h_{n+1}(k_p)k_p}{\lambda+2\mu}-nj_{n}(k_s)h_n(k_s)k_s+nj_{n}(k_p)h_n(k_p)k_p-j_{n+1}(k_p)h_{n}(k_p)k^2_p\bigg)\\
=&\frac{(n+1)}{2n+1}\bigg(2n(n+2)\bigg(\frac{1}{k^2_s}+\frac{1}{(2n+3)(2n+1)}+\frac{3}{(2n+5)(2n+3)(2n+1)(2n-1)}k^2_s+\mathcal{O}(k^3)\bigg)\\
&-\frac{2\mu n(n+2)}{\lambda+2\mu}\bigg(\frac{1}{k^2_p}+\frac{1}{(2n+3)(2n+1)}+\frac{3}{(2n+5)(2n+3)(2n+1)(2n-1)}k^2_p+\mathcal{O}(k^3)\bigg)\\
&+\frac{2\mu (n+2)}{(\lambda+2\mu)(2n+3)}\bigg(1+\frac{2}{(2n+5)(2n+1)}k^2_p+\mathcal{O}(k^3)\bigg)\\
&-\frac{n}{2n+1}\bigg(1+\frac{2}{(2n+3)(2n-1)}k^2_s+\mathcal{O}(k^3)\bigg)+\frac{n}{2n+1}\bigg(1+\frac{2}{(2n+3)(2n-1)}k^2_p+\mathcal{O}(k^3)\bigg)\\
&-\frac{1}{(2n+3)(2n+1)}\bigg(k^2_p+\mathcal{O}(k^3)\bigg)\bigg)\\
=&\frac{2(\lambda+\mu) (n+2)(n+1)n+2\mu(2n+1)(n+2)(n+1)}{(\lambda+2\mu)(2n+3)(2n+1)^2}+\frac{2n(n+1)^2}{(2n+5)(2n+3)(2n+1)^2(2n-1)}k^2_s\\
&+\frac{2\mu(n+2)(n+1)(n-2)+(\lambda+2\mu)(2n+5)(n+1)}{(\lambda+2\mu)(2n+5)(2n+3)(2n+1)^2(2n-1)}k^2_p+\mathcal{O}(k^3).
\end{align*}
Hence we obtain \eqref{S:rho}.

The proof is complete.
\end{proof}


\subsection{Proof of Lemma \ref{le:resonance}}\label{App:A.3}

\begin{proof}
Observe that for $n=0$,
\begin{align*}
j'_0(k)=-\frac{k}{3}+\frac{k^3}{30}+\mathcal{O}(k^4).
\end{align*}
By \eqref{E:hzero}, it is straightforward to compute that $0<k\ll1$,
\begin{align*}
\zeta_0(k)=\frac{1}{2}+k^2(\frac{1}{3}+\frac{\mathrm{i}k}{3}-\frac{k^2}{5}+\mathcal{O}(k^3)),
\end{align*}
which is a higher order expansion than the form given by \eqref{E:zeta0}.
 Combining this with \eqref{E:KK}, \eqref{E:xi0}, \eqref{S:eta0} and \eqref{S:rho0} yields that
\begin{align*}
\lambda_{0}(k)=&(-\frac{1}{2}+\zeta_0(k_1))\xi^{-1}_0(k_1)\rho_0(k\tau)+\delta\tau^2k^2\eta_0(k\tau)\\
=&k^2(\frac{1}{3}c(\omega)+\frac{\mathrm{i}k}{3}(c(\omega))^{\frac32}-\frac{k^2}{5}(c(\omega))^2+\mathcal{O}(k^3(c(\omega))^{\frac52}))\\
&\times(-1-\mathrm{i}k(c(\omega))^{\frac{1}{2}}+\frac{2}{3}k^2c(\omega)+\mathcal{O}(k^3(c(\omega))^{\frac32}))^{-1}\\
&\times(\frac{4\mu}{3(\lambda+2\mu)}-\frac{5\lambda+2\mu}{15(\lambda+2\mu)}\frac{k^2\tau^2}{\lambda+2\mu}+\mathcal{O}(k^3\tau^3))\\
&+\delta\tau^2k^2(-\frac{1}{3(\lambda+2\mu)}-\frac{2}{15(\lambda+2\mu)}\frac{k^2\tau^2}{\lambda+2\mu}+\mathcal{O}(k^3\tau^3))\\
=&k^2\Big(-\frac{4\mu c(\omega)}{9(\lambda+2\mu)}-\frac{\delta\tau^2}{3(\lambda+2\mu)}+\frac{(5\lambda+2\mu)\tau^2 k^2 c(\omega)}{45(\lambda+2\mu)^2}-\frac{34\mu k^2 (c(\omega))^2}{135(\lambda+2\mu)}\\
&-\frac{2\delta \tau^4 k^2}{15(\lambda+2\mu)^2}+\mathcal{O}(k^3)\Big).
\end{align*}

On the other hand, by using \eqref{E:xin}, \eqref{E:zetan}, \eqref{S:eta}  and \eqref{S:rho}, $\lambda_n(k)$ with $1\leq n\leq N$ has  the asymptotic expansion as follows:
\begin{align*}
\lambda_{n}(k)=&(-\frac{1}{2}+\zeta_n(k_1))\xi^{-1}_n(k_1)\rho_n(k\tau)+\delta\tau^2k^2\eta_n(k\tau)\\
=&\frac{(\lambda+\mu)n^2(8n^3+16n^2+4n-1)+\mu n(2n+1)(4n^3+12n^2+5n-4)}{(\lambda+2\mu)(2n+3)(2n+1)^2(2n-1)}\\
&+\frac{-2n^2(n+1)(4n^2+4n+33)}{(2n+5)(2n+3)(2n+1)^3(2n-1)(2n-3)}\frac{k^2\tau^2}{\mu}\nonumber\\
&+\frac{-2\mu n(2n+1)(10n^3+15n^2-19n-12)+(\lambda+2\mu)n(2n+5)(2n-3)}{(\lambda+2\mu)(2n+5)(2n+3)(2n+1)^3(2n-1)(2n-3)}\frac{k^2\tau^2}{\lambda+2\mu}\\
&-\frac{(\lambda+\mu)n(8n^3+16n^2+4n-1)+\mu(2n+1)(4n^3+12n^2+5n-4)}{(\lambda+2\mu)(2n+3)^2(2n+1)^2(2n-1)}k^2c(\omega)\\
&-\delta\frac{2(\lambda+\mu)n(n+1)+\mu(4n^4+4n-1)}{(\lambda+2\mu)(2n+3)(2n+1)(2n-1)}\frac{k^2\tau^2}{\mu}
+\mathcal{O}(k^3).
\end{align*}
The proof is now complete.
\end{proof}

\section{Auxiliary results}\label{App:B}

In the Appendix \ref{App:B}, we will supplement the auxiliary lemma.

\begin{lemm}\label{le:gamma}
One has
\begin{align*}
\boldsymbol\Gamma^{\tilde{k}\tau}(\tilde{\boldsymbol x},\mathbf{s}) =&-e^{\mathrm{i}\frac{\tilde{k}\tau}{c_s}|\tilde{\boldsymbol x}-\mathbf{s}|}\frac{\boldsymbol{\mathcal{A}}(\tilde{\boldsymbol x},\mathbf{s},\frac{\tilde{k}\tau}{c_s})}{4\pi(\tilde{k}\tau)^2|\tilde{\boldsymbol x}-\mathbf{s}|}
-e^{\mathrm{i}\frac{\tilde{k}\tau}{c_p}|\tilde{\boldsymbol x}-\mathbf{s}|}\frac{\boldsymbol{\mathcal{B}}(\tilde{\boldsymbol x},\mathbf{s},\frac{\tilde{k}\tau}{c_p})}{4\pi(\tilde{k}\tau)^2|\tilde{\boldsymbol x}-\mathbf{s}|},
\end{align*}
where
$\boldsymbol{\mathcal{A}}(\tilde{\boldsymbol x},\mathbf{s},\frac{\tilde{k}\tau}{c_s})=(\mathcal{A}_{ij})^3_{i,j=1},\boldsymbol{\mathcal{B}}(\tilde{\boldsymbol x},\mathbf{s},\frac{\tilde{k}\tau}{c_p})=(\mathcal{B}_{ij})^3_{i,j=1}$ are two $3\times3$ matrices with entries
\begin{align*}
\mathcal{A}_{ii}=&\frac{1}{|\tilde{\boldsymbol x}-\mathbf{s}|^4}\big(3(\tilde{x}_i-s_i)^2-3\mathrm{i}\frac{\tilde{k}\tau}{c_s}(\tilde{x}_i-s_i)^2|\tilde{\boldsymbol x}-\mathbf{s}|-(\frac{\tilde{k}\tau}{c_s})^2(\tilde{x}_i-s_i)^2|\tilde{\boldsymbol x}-\mathbf{s}|^2\\
&-|\tilde{\boldsymbol x}-\mathbf{s}|^2+\mathrm{i}\frac{\tilde{k}\tau}{c_s}|\tilde{\boldsymbol x}-\mathbf{s}|^3
+(\frac{\tilde{k}\tau}{c_s})^2|\tilde{\boldsymbol x}-\mathbf{s}|^4\big),\\
\mathcal{A}_{ij}=&\frac{(\tilde{x}_i-s_i)(\tilde{x}_j-s_j)}{|\tilde{\boldsymbol x}-\mathbf{s}|^4}\big(3-3\mathrm{i}\frac{\tilde{k}\tau}{c_s}|\tilde{\boldsymbol x}-\mathbf{s}|
-(\frac{\tilde{k}\tau}{c_s})^2|\tilde{\boldsymbol x}-\mathbf{s}|^2\big),i\neq j,
\end{align*}
and
\begin{align*}
\mathcal{B}_{ii}=&-\frac{1}{|\tilde{\boldsymbol x}-\mathbf{s}|^4}\big(3(\tilde{x}_i-s_i)^2-3\mathrm{i}\frac{\tilde{k}\tau}{c_p}(\tilde{x}_i-s_i)^2|\tilde{\boldsymbol x}-\mathbf{s}|-(\frac{\tilde{k}\tau}{c_p})^2(\tilde{x}_i-s_i)^2|\tilde{\boldsymbol x}-\mathbf{s}|^2\\
&-|\tilde{\boldsymbol x}-\mathbf{s}|^2+\mathrm{i}\frac{\tilde{k}\tau}{c_p}|\tilde{\boldsymbol x}-\mathbf{s}|^3\big),\\
\mathcal{B}_{ij}=&-\frac{(\tilde{x}_i-s_i)(\tilde{x}_j-s_j)}{|\tilde{\boldsymbol x}-\mathbf{s}|^4}\big(3-3\mathrm{i}\frac{\tilde{k}\tau}{c_p}|\tilde{\boldsymbol x}-\mathbf{s}|
-(\frac{\tilde{k}\tau}{c_p})^2|\tilde{\boldsymbol x}-\mathbf{s}|^2\big),i\neq j.
\end{align*}

Moreover, one has that for $\ell=1,2,3$,
\begin{align*}
\partial_{\ell}\boldsymbol\Gamma^{\tilde{k}\tau}(\tilde{\boldsymbol x},\mathbf{s})=-e^{\mathrm{i}\frac{\tilde{k}\tau}{c_s}|\tilde{\boldsymbol x}-\mathbf{s}|}\frac{\boldsymbol{\mathscr{A}}^{(\ell)}(\tilde{\boldsymbol x},\mathbf{s},\frac{\tilde{k}\tau}{c_s})}{4\pi(\tilde{k}\tau)^2|\tilde{\boldsymbol x}-\mathbf{s}|}
-e^{\mathrm{i}\frac{\tilde{k}\tau}{c_p}|\tilde{\boldsymbol x}-\mathbf{s}|}\frac{\boldsymbol{\mathscr{B}}^{(\ell)}(\tilde{\boldsymbol x},\mathbf{s},\frac{\tilde{k}\tau}{c_p})}{4\pi(\tilde{k}\tau)^2|\tilde{\boldsymbol x}-\mathbf{s}|},
\end{align*}
where $\boldsymbol{\mathscr{A}}^{(\ell)}(\tilde{\boldsymbol x},\mathbf{s},\frac{\tilde{k}\tau}{c_s})=(\mathscr{A}^{(\ell)}_{ij})^3_{i,j=1},\boldsymbol{\mathscr{B}}^{(\ell)}(\tilde{\boldsymbol x},\mathbf{s},\frac{\tilde{k}\tau}{c_p})=(\mathscr{B}^{(\ell)}_{ij})^3_{i,j=1},\ell=1,2,3$ are $3\times3$ matrices with entries
\begin{align*}
\mathscr{A}^{(\ell)}_{\ell\ell}=&\frac{\tilde{x}_\ell-s_\ell}{|\tilde{\boldsymbol x}-\mathbf{s}|^6}(-15(\tilde{x}_\ell-s_\ell)^2+15\mathrm{i}\frac{\tilde{k}\tau}{c_s}(\tilde{x}_\ell-s_\ell)^2|\tilde{\boldsymbol x}-\mathbf{s}|+6(\frac{\tilde{k}\tau}{c_s})^2(\tilde{x}_\ell-s_\ell)^2|\tilde{\boldsymbol x}-\mathbf{s}|^2+9|\tilde{\boldsymbol x}-\mathbf{s}|^2\\
&-\mathrm{i}(\frac{\tilde{k}\tau}{c_s})^3(\tilde{x}_\ell-s_\ell)^2|\tilde{\boldsymbol x}-\mathbf{s}|^3-9\mathrm{i}\frac{\tilde{k}\tau}{c_s}|\tilde{\boldsymbol x}-\mathbf{s}|^3
-4(\frac{\tilde{k}\tau}{c_s})^2|\tilde{\boldsymbol x}-\mathbf{s}|^4+\mathrm{i}(\frac{\tilde{k}\tau}{c_s})^3|\tilde{\boldsymbol x}-\mathbf{s}|^5),\\
\mathscr{A}^{(\ell)}_{ii}=&\frac{\tilde{x}_\ell-s_\ell}{|\tilde{\boldsymbol x}-\mathbf{s}|^6}(-15(\tilde{x}_\ell-s_\ell)^2+15\mathrm{i}\frac{\tilde{k}\tau}{c_s}(\tilde{x}_i-s_i)^2|\tilde{\boldsymbol x}-\mathbf{s}|+6(\frac{\tilde{k}\tau}{c_s})^2(\tilde{x}_i-s_i)^2|\tilde{\boldsymbol x}-\mathbf{s}|^2+3|\tilde{\boldsymbol x}-\mathbf{s}|^2\\
&-\mathrm{i}(\frac{\tilde{k}\tau}{c_s})^3(\tilde{x}_i-s_i)^2|\tilde{\boldsymbol x}-\mathbf{s}|^3-3\mathrm{i}\frac{\tilde{k}\tau}{c_s}|\tilde{\boldsymbol x}-\mathbf{s}|^3
-2(\frac{\tilde{k}\tau}{c_s})^2|\tilde{\boldsymbol x}-\mathbf{s}|^4+\mathrm{i}(\frac{\tilde{k}\tau}{c_s})^3|\tilde{\boldsymbol x}-\mathbf{s}|^5),i\neq \ell,\\
\mathscr{A}^{(\ell)}_{i\ell}=&\frac{\tilde{x}_i-s_i}{|\tilde{\boldsymbol x}-\mathbf{s}|^4}(3-3\mathrm{i}\frac{\tilde{k}\tau}{c_s}|\tilde{\boldsymbol x}-\mathbf{s}|
-(\frac{\tilde{k}\tau}{c_s})^2|\tilde{\boldsymbol x}-\mathbf{s}|^2)\\
&+\frac{(x_i-s_i)(x_\ell-s_\ell)^2}{|\tilde{\boldsymbol x}-\mathbf{s}|^6}(-15+15\mathrm{i}\frac{\tilde{k}\tau}{c_s}|\tilde{\boldsymbol x}-\mathbf{s}|
+6(\frac{\tilde{k}\tau}{c_s})^2|\tilde{\boldsymbol x}-\mathbf{s}|^2-\mathrm{i}(\frac{\tilde{k}\tau}{c_s})^3|\tilde{\boldsymbol x}-\mathbf{s}|^3),i\neq \ell,\\
\mathscr{A}^{(\ell)}_{\ell j}=&\frac{\tilde{x}_j-s_j}{|\tilde{\boldsymbol x}-\mathbf{s}|^4}(3-3\mathrm{i}\frac{\tilde{k}\tau}{c_s}|\tilde{\boldsymbol x}-\mathbf{s}|
-(\frac{\tilde{k}\tau}{c_s})^2|\tilde{\boldsymbol x}-\mathbf{s}|^2)\\
&+\frac{(\tilde{x}_j-s_j)(\tilde{x}_\ell-s_\ell)^2}{|\tilde{\boldsymbol x}-\mathbf{s}|^6}(-15+15\mathrm{i}\frac{\tilde{k}\tau}{c_s}|\tilde{\boldsymbol x}-\mathbf{s}|
+6(\frac{\tilde{k}\tau}{c_s})^2|\tilde{\boldsymbol x}-\mathbf{s}|^2-\mathrm{i}(\frac{\tilde{k}\tau}{c_s})^3|\tilde{\boldsymbol x}-\mathbf{s}|^3),j\neq \ell, \\
\mathscr{A}^{(\ell)}_{ij}=&\frac{(\tilde{x}_i-s_i)(\tilde{x}_j-s_j)(\tilde{x}_\ell-s_\ell)}{|\tilde{\boldsymbol x}-\mathbf{s}|^6}(-15+15\mathrm{i}\frac{\tilde{k}\tau}{c_s}|\tilde{\boldsymbol x}-\mathbf{s}|
+6(\frac{\tilde{k}\tau}{c_s})^2|\tilde{\boldsymbol x}-\mathbf{s}|^2-\mathrm{i}(\frac{\tilde{k}\tau}{c_s})^3|\tilde{\boldsymbol x}-\mathbf{s}|^3),\\
&i\neq j,i\neq \ell,j\neq \ell,
\end{align*}
and
\begin{align*}
\mathscr{B}^{(\ell)}_{\ell\ell}=&-\frac{\tilde{x}_\ell-s_\ell}{|\tilde{\boldsymbol x}-\mathbf{s}|^6}(-15(\tilde{x}_\ell-s_\ell)^2+15\mathrm{i}\frac{\tilde{k}\tau}{c_p}(\tilde{x}_\ell-s_\ell)^2|\tilde{\boldsymbol x}-\mathbf{s}|+6(\frac{\tilde{k}\tau}{c_p})^2(\tilde{x}_\ell-s_\ell)^2|\tilde{\boldsymbol x}-\mathbf{s}|^2+9|\tilde{\boldsymbol x}-\mathbf{s}|^2\\
&-\mathrm{i}(\frac{\tilde{k}\tau}{c_p})^3(\tilde{x}_\ell-s_\ell)^2|\tilde{\boldsymbol x}-\mathbf{s}|^3-9\mathrm{i}\frac{\tilde{k}\tau}{c_p}|\tilde{\boldsymbol x}-\mathbf{s}|^3
-3(\frac{\tilde{k}\tau}{c_p})^2|\tilde{\boldsymbol x}-\mathbf{s}|^4),\\
\mathscr{B}^{(\ell)}_{ii}=&-\frac{\tilde{x}_\ell-s_\ell}{|\tilde{\boldsymbol x}-\mathbf{s}|^6}(-15(\tilde{x}_\ell-s_\ell)^2+15\mathrm{i}\frac{\tilde{k}\tau}{c_p}(\tilde{x}_i-s_i)^2|\tilde{\boldsymbol x}-\mathbf{s}|+6(\frac{\tilde{k}\tau}{c_p})^2(\tilde{x}_i-s_i)^2|\tilde{\boldsymbol x}-\mathbf{s}|^2+3|\tilde{\boldsymbol x}-\mathbf{s}|^2\\
&-\mathrm{i}(\frac{\tilde{k}\tau}{c_p})^3(\tilde{x}_i-s_i)^2|\tilde{\boldsymbol x}-\mathbf{s}|^3-3\mathrm{i}\frac{\tilde{k}\tau}{c_p}|\tilde{\boldsymbol x}-\mathbf{s}|^3
-(\frac{\tilde{k}\tau}{c_p})^2|\tilde{\boldsymbol x}-\mathbf{s}|^4),i\neq \ell,\\
\mathscr{B}^{(\ell)}_{i\ell}=&-\frac{\tilde{x}_i-s_i}{|\tilde{\boldsymbol x}-\mathbf{s}|^4}(3-3\mathrm{i}\frac{\tilde{k}\tau}{c_p}|\tilde{\boldsymbol x}-\mathbf{s}|
-(\frac{\tilde{k}\tau}{c_p})^2|\tilde{\boldsymbol x}-\mathbf{s}|^2)\\
&-\frac{(\tilde{x}_i-s_i)(\tilde{x}_\ell-s_\ell)^2}{|\tilde{\boldsymbol x}-\mathbf{s}|^6}(-15+15\mathrm{i}\frac{\tilde{k}\tau}{c_p}|\tilde{\boldsymbol x}-\mathbf{s}|
+6(\frac{\tilde{k}\tau}{c_p})^2|\tilde{\boldsymbol x}-\mathbf{s}|^2-\mathrm{i}(\frac{\tilde{k}\tau}{c_p})^3|\tilde{\boldsymbol x}-\mathbf{s}|^3),i\neq \ell, \\
\mathscr{B}^{(\ell)}_{\ell j}=&-\frac{\tilde{x}_j-s_j}{|\tilde{\boldsymbol x}-\mathbf{s}|^4}(3-3\mathrm{i}\frac{\tilde{k}\tau}{c_p}|\tilde{\boldsymbol x}-\mathbf{s}|
-(\frac{\tilde{k}\tau}{c_p})^2|\tilde{\boldsymbol x}-\mathbf{s}|^2)\\
&-\frac{(\tilde{x}_j-s_j)(\tilde{x}_\ell-s_\ell)^2}{|\tilde{\boldsymbol x}-\mathbf{s}|^6}(-15+15\mathrm{i}\frac{\tilde{k}\tau}{c_p}|\tilde{\boldsymbol x}-\mathbf{s}|
+6(\frac{\tilde{k}\tau}{c_p})^2|\tilde{\boldsymbol x}-\mathbf{s}|^2-\mathrm{i}(\frac{\tilde{k}\tau}{c_p})^3|\tilde{\boldsymbol x}-\mathbf{s}|^3),j\neq \ell, \\
\mathscr{B}^{(\ell)}_{ij}=&-\frac{(\tilde{x}_i-s_i)(\tilde{x}_j-s_j)(\tilde{x}_\ell-s_\ell)}{|\tilde{\boldsymbol x}-\mathbf{s}|^6}(-15+15\mathrm{i}\frac{\tilde{k}\tau}{c_p}|\tilde{\boldsymbol x}-\mathbf{s}|
+6(\frac{\tilde{k}\tau}{c_p})^2|\tilde{\boldsymbol x}-\mathbf{s}|^2-\mathrm{i}(\frac{\tilde{k}\tau}{c_p})^3|\tilde{\boldsymbol x}-\mathbf{s}|^3),\\
&i\neq j,i\neq \ell,j\neq \ell.
\end{align*}
\end{lemm}
\begin{proof}
By formula \eqref{Fundamental2}, it is direct to obtain the conclusion of the lemma.
\end{proof}


\end{document}